\documentclass[12pt]{amsart}
\usepackage[T1]{fontenc}
\usepackage[latin1]{inputenc}
\usepackage{typearea}
\usepackage{geometry}
\usepackage{ulem}
\usepackage{amsmath}
\usepackage{amssymb}
\usepackage{latexsym}
\usepackage{enumerate}
\usepackage{amsthm}
\usepackage[all]{xy}
\usepackage{hhline}
\usepackage{epsf} 
\usepackage{cite}
\usepackage{verbatim}
\usepackage{mathtools}
\usepackage{comment}
\usepackage{dsfont}
\usepackage{mathrsfs}

\usepackage{color}

\newcommand{\Gammabar}{\bar{\Gamma}}
\renewcommand{\1}{\mathds{1}}

\newtheorem{theorem}{{\sc Theorem}}[section]
\newtheorem{cor}[theorem]{{\sc Corollary}}

\newtheorem{lemma}[theorem]{{\sc Lemma}}
\newtheorem{prop}[theorem]{{\sc Proposition}}

\theoremstyle{remark}
\newtheorem{remark}[theorem]{{\sc Remark}}

\newtheorem{assumption}[theorem]{\sc Assumption}

\theoremstyle{definition}

\newtheorem{example}[theorem]{\sc example}

\newcommand{\R}{\mathbb{R} }
\newcommand{\N}{\mathbb{N} }

\newcommand{\A}{\mathcal{A}}
\newcommand{\B}{\mathcal{B}}
\newcommand{\F}{\mathcal{F}}
\newcommand{\G}{\mathcal{G}}
\newcommand{\C}{\mathbb{C}}
\newcommand{\D}{\mathcal{D}}
\newcommand{\W}{\mathcal{W}}

\newcommand{\im}{\textnormal{im}}

\newcommand{\Prob}{\mathbb{P}}
\newcommand{\E}{\mathbb{E}}

\newcommand{\Pot}{\mathcal{P}}

\providecommand{\abs}[1]{\lvert #1\rvert}
\providecommand{\babs}[1]{\bigl\lvert #1\bigr\rvert}
\providecommand{\Babs}[1]{\Bigl\lvert #1\Bigr\rvert}

\providecommand{\norm}[1]{\lVert #1\rVert}
\providecommand{\fnorm}[1]{\lVert #1\rVert_\infty}

\providecommand{\Enorm}[1]{\lVert #1\rVert_2}

\DeclareMathOperator{\Irr}{Irr}
\DeclareMathOperator{\diag}{diag}
\DeclareMathOperator{\Var}{Var}

\DeclareMathOperator{\dom}{dom}
\DeclareMathOperator{\Cov}{Cov}

\DeclareMathOperator{\Lip}{Lip}

\DeclareMathOperator{\Id}{Id}
\DeclareMathOperator{\spec}{spec}
\DeclarePairedDelimiter{\ceil}{\lceil}{\rceil}

\renewcommand{\phi}{\varphi}
\renewcommand{\epsilon}{\varepsilon}

\renewcommand{\rho}{\varrho}
\renewcommand{\P}{\mathbb{P}}

\begin{document}
\title[Non-linear exchangeable pairs]{Normal approximation via non-linear exchangeable pairs}
\author{Christian D\"obler}
\thanks{\noindent Universit\"at Osnabr\"uck, Fachbereich Mathematik/Informatik \\
E-mails: christian\_doebler@gmx.de }
\begin{abstract}  
We propose a new functional analytic approach to Stein's method of exchangeable pairs that does not require the pair at hand to satisfy any approximate linear regression property. We make use of this theory in order to derive abstract bounds on the normal and Gamma approximation of certain functionals in the Wasserstein distance. Moreover, we illustrate the relevance of this approach by means of three instances of situations to which it can be applied: Functionals of independent random variables, finite population statistics and functionals on finite groups. In the independent case, and in particular for symmetric $U$-statistics, we demonstrate in which respect this approach yields fundamentally better bounds than those in the existing literature. Finally, we apply our results to provide Wasserstein bounds in a CLT for subgraph counts in geometric random graphs based on $n$ i.i.d. points in Euclidean space as well as to the normal approximation of Pearson's statistic.
\end{abstract}

\maketitle

\section{Introduction}\label{intro}
\subsection{Overview and motivation}
In Stein's method of normal approximation, one of the most prominent coupling constructions is the technique of \textit{exchangeable pairs}. In a nutshell, if $W$ is a mean zero and variance one, real-valued random variable on a probability space $(\Omega,\F,\P)$, the approach consists in the construction (maybe on an enlarged probability space) of a further random variable $W'$ such that the pair $(W,W')$ is exchangeable, i.e. such that $(W,W')$ and $(W',W)$ are identically distributed. In general, one may think of $W'$ as a 
small random perturbation of $W$ that still has the same probabilistic properties. This coupling has been systematically introduced in Stein's monograph \cite{St86}, where it is called \textit{auxiliary randomization}, and it is also discussed in detail in the modern monograph \cite{CGS} as well as in the introductory survey \cite{Ross}. One of the main reasons for the success of this approach -- and its main advantage over other coupling techniques -- is that, in many situations, there is a straightforward way of constructing an exchangeable pair. Although, in this work, we focus on univariate normal approximation, we mention that the method of exchangeable pairs has also been established for other univariate absolutely continuous limiting distributions\cite{CFR11, ChSh, EiLo10, DoeBeta, DP18b} as well as for the multivariate normal distribution \cite{ChaMe,ReiRoe09, Me09} and some more general distributions with log-concave density \cite{FSX}. Very recently, in \cite{Thom} a version of this method has been established for the approximation by certain measures on a Riemannian manifold.

Recall that the Wasserstein distance between two probability measures $P$ and $Q$ on $(\R,\B(\R))$ with finite means can be defined by 
\begin{equation*}
 d_\W(P,Q):=\sup_{h\in\Lip(1)}\Babs{\int_\R hdP-\int_\R hdQ}\,,
\end{equation*}
where $\Lip(1)$ denotes the class of all $1$-Lipschitz functions on $\R$. With some abuse, for two integrable, real random variables $Y,Z$ with distributions $P$ and $Q$, respectively, we also write $d_\W(Y,Z)$ for $d_\W(P,Q)$ so that  
\[d_\W(Y,Z)=\sup_{h\in\Lip(1)}\babs{\E[h(Y)]-\E[h(Z)]}\,.\]

In order to obtain a general bound on the accuracy of the normal approximation of $W$, in \cite{St86} the following \textit{linear regression property} or \textit{linearity condition} for an exchangeable pair $(W,W')$ such that $\E[W]=0$ and $\Var(W)=1$ was introduced: 
For some (typically small) constant $\lambda>0$ it holds that
\begin{equation}\label{clinreg}
\E\bigl[W'-W\,\bigl|\, W\bigr]=-\lambda W\,.
\end{equation} 
Then, one has the following bound, which is essentially due to Stein (see\cite[Lecture 3, Theorem 1]{St86}):
\begin{align}\label{cb1}
d_\W(W,Z)&\leq \sqrt{\frac{2}{\pi}}\E\Babs{1-\frac{1}{2\lambda}\E\bigl[\bigl(W'-W\bigr)^2\,\bigl|\, W\bigr]}+\frac{1}{2\lambda} \E\babs{W'-W}^3\notag\\
&\leq \sqrt{\frac{2}{\pi}}\biggl(\Var\Bigl(\frac{1}{2\lambda}\E\bigl[\bigl(W'-W\bigr)^2\,\bigl|\, W\bigr]\Bigr)\biggl)^{1/2}
+\frac{1}{2\lambda} \E\babs{W'-W}^3\,,
\end{align}
where, here and in what follows, $Z$ denotes a standard normal random variable.
This bound implies that under condition \eqref{clinreg}, the normal approximation of $W$ is accurate whenever the difference $W'-W$ is  reasonably small and the quantity 
\[\frac{1}{2\lambda}\E\bigl[\bigl(W'-W\bigr)^2\,\bigl|\, W\bigr]\]
is close to its expected value one.

Despite the usefulness of this bound, it was pointed out by Rinott and Rotar \cite{RiRo97} that condition \eqref{clinreg} is quite restrictive. 
Thus, they suggested the following modification to it: For some constant $\lambda>0$ and some random variable $R$ one has 
\begin{equation}\label{linreg}
\E\bigl[W'-W\,\bigl|\, W\bigr]=-\lambda W +R\,.
\end{equation}
Of course, in contrast to \eqref{clinreg}, condition \eqref{linreg} is not really a condition, since for any choice of $\lambda>0$, the random variable $R$ could just be defined by this equation. 
If \eqref{linreg} holds for the exchangeable pair $(W,W')$, then \eqref{cb1} becomes
\begin{align}\label{cb2}
d_\W(W,Z)&\leq \sqrt{\frac{2}{\pi}}\E\Babs{1-\frac{1}{2\lambda}\E\bigl[(W'-W)^2\,\bigl|\,W\bigr]} +\frac{1}{3\lambda}\E\babs{W'-W}^3+\frac{1}{\lambda}\E\abs{R}\notag\\
&\leq \sqrt{\frac{2}{\pi}}\biggl(\Var\Bigl(\frac{1}{2\lambda}\E\bigl[\bigl(W'-W\bigr)^2\,\bigl|\, W\bigr]\Bigr)\biggl)^{1/2}
+\frac{1}{2\lambda} \E\babs{W'-W}^3 \notag\\
&\;+\frac{2}{\lambda}\sqrt{\E\bigl[R^2\bigr]} \,.
\end{align}
Note that this bound can only be small if $\lambda^{-1}R$ is of negligible order. 
Hence, although \eqref{linreg} significantly extends the scope of applicability of the method of exchangeable pairs, the previous bound is still only useful if the linearity condition \eqref{clinreg} holds at least approximately. We remark that the bounds \eqref{cb1} and \eqref{cb2} remain valid if in \eqref{clinreg}-\eqref{cb2} conditioning on $\sigma(W)$ is replaced with conditioning on some sub-$\sigma$-field 
$\mathcal{G}\subseteq \F$ such that $ \sigma(W)\subseteq\mathcal{G}$.

In the present work we suggest a new \textit{functional analytic approach} connected with exchangeable pairs, which does not rely on such a linearity condition. The main idea behind our approach is that, associated to every exchangeable pair $(X,X')$ of random variables with values in an arbitrary measurable space $(E,\mathcal{E})$, there is a natural reversible Markov generator $L$ acting on $L^2(\mu)$, where $\mu$ is the marginal distribution of $X$. Namely, for $F\in L^2(\mu)$ we let
\begin{equation}\label{Lintro}
LF(x):=\E\bigl[F(X')-F(X)\,\bigl|\,X=x\bigr]=KF(x)-F(x)\,,
\end{equation}
where $KF(x)=\E[F(X')\,|\,X=x]$. Due to its reversibility, the operator $L$ satisfies a crucial \textit{integration by parts formula} which will be very useful when dealing with the terms arising from the Stein equation and which replaces the central identity 
\[ \E\bigl[Wf(W)\bigr]=\frac{1}{2\lambda}\E\bigl[(W'-W)(f(W')-f(W)\bigr]+\frac{1}{\lambda}\E\bigl[f(W)R\bigr]\]
for an exchangeable pair satisfying \eqref{linreg}. A second cornerstone of our approach is the observation that, in practice, many exchangeable pairs are of the form $(W,W')=(F(X),F(X'))$, i.e. the pair $(W,W')$ of interest is obtained by applying a functional to an exchangeable pair 
$(X,X')$ on a more abstract space $(E,\mathcal{E})$. Finally, as it turns out, in many cases of interest the space $L^2(\mu)$ is the (finite or countable) orthogonal sum of the eigenspaces of $L$ and, hence, every $F\in L^2(\mu)$ can be uniquely written as a sum or series of eigenfunctions of $L$, i.e. the self-adjoint operator $L$ is diagonalizable. As will be shown below, in this terminology, the classical linear regression property then just expresses that the functional $F$ is an eigenfunction of $L$ with corresponding eigenvalue $\lambda$. 
In particular, our method encompasses the classical exchangeable pairs approach as a special case. In a sense, one of the main purposes of this work is to demonstrate that the 
auxiliary randomization that is inherent in the coupling $(X,X')$ (or, equivalently, in the operator $L$) may provide an abundance of information about the structure of the whole original 
probability space $(E,\mathcal{E},\mu)$. In this paper, this structural information is only used to derive bounds on distributional approximation. However, it might well be useful for other purposes as well. For instance, it was shown in \cite{Chacon} how one can use exchangeable pairs in order to derive concentration inequalities in various models featuring dependence.

We refer to Section \ref{setup} for more details, precise statements and for the technical elaboration of our ideas. Here, we content ourselves with stating a special version of one of our abstract results, which is an immediate generalization of the bound \eqref{cb1} and whose statement does not make use of the rather operator theoretic language introduced in Section \ref{setup}.    

\begin{theorem}\label{genexpair}
 Let $(E,\mathcal{E})$ be a measurable space and suppose that, on some probability space $(\Omega,\F,\P)$,  $(X,X')$ is an exchangeable pair of $E$-valued random variables with marginal distribution $\mu$ on $(E,\mathcal{E})$. Moreover, assume that $m$ is a positive integer and that, for $1\leq p\leq m$, $F_p\in L^3(\mu)$ is a functional such that the exchangeable pairs $(W_p,W_p'):=(F_p(X),F_p(X'))$ satisfy 
 \[\E\bigl[W_p'-W_p\,\bigl|\,X\bigr]=-\lambda_p W_p\,,\]
 where $\lambda_1,\dotsc,\lambda_m$ are pairwise distinct positive numbers. Define $W:=\sum_{p=1}^m W_p$ and assume that $\E[W]=0$ and 
 $\Var(W)
=1$. Then, one has the bound 
 \begin{align*}
  d_\W(W,Z)&\leq \sqrt{\frac{2}{\pi}}\biggl(\Var\Bigl(\sum_{p,q=1}^m\frac{1}{2\lambda_p}\E\bigl[\bigl(W_p'-W_p\bigr)\bigl(W_q'-W_q\bigr)\,\bigl|\, X\bigr]\Bigr)\biggr)^{1/2}\notag\\
&\;+\frac12\sum_{p,q,r=1}^m\lambda_p^{-1}\E\bigl[\bigl(W'_q-W_q\bigr)\bigl(W'_r-W_r\bigr)\babs{W'_p-W_p}\bigr]\,,
 \end{align*}
where $Z$ is a standard normal random variable.
\end{theorem}

\subsection{Further links to the existing literature on Stein's method}
Apart from the obvious connection to the exchangeable pairs approach highlighted above, the present paper is also closely related to the part of Stein's method that 
deals with the distributional approximation of functionals of the invariant measure of some reversible Markov generator. This line of research was initiated by Nourdin and Peccati in the seminal paper \cite{NouPec09a} which deals with the normal and Gamma approximation of smooth functionals of an abstract Gaussian process $\mathfrak{X}$. In \cite{NouPec09a} the operator $L$ is the infinite-dimensional Ornstein-Uhlenbeck generator whose invariant measure is given by the law of $\mathfrak{X}$. However, the methodology of \cite{NouPec09a} largely depends
on the well established Malliavin calculus and Wiener-It\^{o} chaos decomposition on Gaussian spaces. In particular, it relies on the Malliavin partial integration formula as well as on the relationship $L=-\delta D$ between the three Malliavin operators. More recently however, in the papers \cite{Led12, ACP} it has been demonstrated how, in the more general case of a diffusive Markov generator $L$, one can dispense with the Malliavin operators $\delta, D$ and still carry out a similar analysis as in \cite{NouPec09a} by resorting to a different integration by parts formula that rather involves the \textit{carr\'{e} du champ} operator $\Gamma$ associated to $L$ than the Malliavin operators $D$ and $\delta$. The operator Gamma is defined for suitable functionals by 
\[\Gamma(F,G):=\frac{1}{2}\bigl(L(FG)-FLG-GLF\bigr)\]
and $L$ is called \textit{diffusive} if 
\[\Gamma(\psi(F),G)=\psi'(F)\Gamma(F,G)\]
for smooth functions $\psi$. The \textit{integration by parts formula} in this setting then reads 
\[\int_E FLG\,d\mu=-\int_E\Gamma(F,G)d\mu\,,\]
where $\mu$ is the invariant distribution of $L$.
Moreover, in typical situations, the Wiener-It\^{o} chaos decomposition of functionals of $\mathfrak{X}$ may be replaced by the decomposition of the space $L^2(\mu)$ into eigenspaces of $L$. Despite the existence of Malliavin structures 
on the Poisson space and for general Rademacher sequences, this \textit{spectral viewpoint} and in particular the carr\'{e} du champ integration by parts formula was key to investigating 
the so-called \textit{fourth moment phenomenon} on such spaces \cite{DP18, DVZ18, DK19}. 
Without going into details we remark that the Ornstein-Uhlenbeck generators in the Poisson and Rademacher settings are not diffusive and that an essential part of the technicality in the papers \cite{DP18, DK19} dealt with controlling their non-diffusiveness. 

The operator $L$ defined in \eqref{Lintro} is also non-diffusive and we will see in Section \ref{setup} how we can measure its non-smoothness. As in the papers \cite{DP18, DK19}, this will be done via an alternative expression for $\Gamma(F,G)$.

 As in the diffusive situation it is very important to understand under what assumptions there exists a \textit{pseudo-inverse} $L^{-1}$ to $L$. This is related to the question under which additional condition the image of $L$ is  equal to the entire subspace of mean zero functionals in $L^2(\mu)$. In Section \ref{setup} we will see that this holds whenever the operator $L$ satisfies a \textit{Poincar\'{e} inequality}, which itself is equivalent to $L$ having a spectral gap. In practice this condition will most often be easily verified. Moreover, we will explain how this observation answers an old question of Stein \cite{St86}.

Finally, we would like to stress that, actually, the theory developed in this paper provides bounds on the normal approximation of functionals $F\in L^2(\mu)$, whenever 
$(E,\mathcal{E},\mu)$ is a probability space such that $\mu$ is the reversible distribution of a certain Markov kernel $K:E\times \mathcal{E}\rightarrow\R$. 
In particular, the whole theory developed here could have been phrased uniquely in terms of the data $\mu$ and $K$ (or $L=K-\Id$), without explicitly mentioning any exchangeable pair $(X,X')$. We refrained from doing so for mainly two reasons: Firstly, the original idea was to extend the exchangeable pairs approach beyond the (approximately) linear situation and by keeping the familiar terminology our results are more easily comparable to the existing theory. Secondly, and more importantly in fact, given $\mu$ there might be many possible choices for the reversible kernel $K$ and the first step consists in choosing an appropriate one for the problem at hand. In fact, in practice one is usually just given a random variable $W=F(X)$ 
and, as it turns out, in many situation there is a \textit{natural} way to construct an $X'$ such that $(X,X')$ is exchangeable, which then leads to the kernel $K$. Hence, this form of presentation also appears more honest as far as applications are concerned.

\subsection{Structure of the paper} 
The paper is structured as follows. In Section \ref{setup} we describe and investigate the fundamental theoretical framework of this paper and state our main abstract results on normal approximation in this context. We also demonstrate by means of the centered Gamma distribution, how this strategy can be extended to other absolutely-continuous distributions and, finally, we essentially solve an old problem on exchangeable pairs posed by Stein \cite{St86}. In Section \ref{spaces} we discuss three example cases of probability spaces that allow for a natural exchangeable pair structure as in Section \ref{setup}: product spaces, finite populations and finite groups equipped with the uniform distribution. For functionals on product spaces, we derive bounds on the normal approximation in descending levels of generality, from general functionals of independent random variables to symmetric functionals of i.i.d. variables to symmetric $U$-statistics. Then, in Section \ref{apps} we apply our bounds 
to derive rates of convergence on the Wasserstein distance for two examples of statistics: the normalized subgraph counts in geometric random graphs and Pearson's well-known chi-square statistic, which may be asymptotically normal when the number of classes diverges to infinity with the sample size. Finally, in Section \ref{proofs} we provide 
the proofs of several technical results in the paper.

\section{Setup and main abstract results}\label{setup}
\subsection{Basic setup and assumptions}
In this subsection we carefully introduce the objects and the notation needed in what follows. Our basic assumption is that $(E,\mathcal{E},\mu)$ is a given probability space and that, on some other probability space $(\Omega,\F,\Prob)$, we can construct an exchangeable pair $(X,X')$ of 
$E$-valued random variables such that $\P\circ X^{-1}=\mu$. Moreover, for technical reasons we will assume that a regular conditional distribution $K$ of $X'$ given $X$ exists, i.e. that there is a Markov kernel $K:E\times \mathcal{E}\rightarrow [0,1]$ from $(E,\mathcal{E})$ to $(E,\mathcal{E})$ such that for all $ B\in\mathcal{E}$
 \[K(X,B)=\Prob(X'\in B|X)\quad \Prob\text{-a.s.}\] 
This holds for example, whenever $(E,\mathcal{E})$ is a Borel space.

In what follows we will sometimes write $E_\mu$, $\Var_\mu$ and $\Cov_\mu$ for the expectation, the variance and the covariance taken with respect to $\mu$, whereas we reserve the symbols $\E$, $\Var$ and $\Cov$ for the respective operators with respect to $\P$.

Given the exchangeable pair $(X,X')$ with marginal distribution $\mu$, we suggest the following probabilistic construction of an operator $L:L^2(\mu)\rightarrow L^2(\mu)$. For $x\in E$ we define 
\begin{equation}\label{defL}
  LF(x):=\E\bigl[F(X')-F(X)\,\bigl|\,X=x\bigr]=\E\bigl[F(X')\,\bigl|\,X=x\bigr]-F(x)\,,
\end{equation}
i.e. $L=K-\Id$, where we use the standard notation 
\[KF(x):=\int_E F(y) K(x,dy)\]
and denote by $\Id$ the identity operator on $L^2(\mu)$. From well-known properties of the conditional expectation it is easy to see that $L$ is indeed well-defined, i.e. if $F\in L^2(\mu)$ then $LF$ is again in $L^2(\mu)$ and a simple computation shows that 
$\norm{LF}_{L^2(\mu)}\leq 2\norm{F}_{L^2(\mu)}$. In particular, $L$ is a bounded linear operator on $L^2(\mu)$. Note further that $L$ really depends on the choice of the exchangeable pair $(X,X')$ (or the Markov kernel $K$) and not just on the probability measure $\mu=\Prob\circ X^{-1}$. 
Since $L=K-\Id$ we can always consider $L$ as the infinitesimal generator of a Markov jump process $(X_t)_{t\geq 0}$, see Equation (2.1) in Chapter 4 of \cite{EtKu}, which will be important in the sequel.
Moreover, in \cite[Chapter 4]{EtKu}, the following well-known probabilistic construction of $(X_t)_{t\geq 0}$ is also given: 
Let $(Y(n))_{n\in\N_0}$ be a discrete time Markov chain on $(E,\mathcal{E})$ with transition kernel $K$ and let $(N_t)_{t\geq0}$ be a homogeneous Poisson process on $[0,\infty)$ with rate $1$ which is independent of $(Y(n))_{n\in\N_0}$. 
Then, the process $(X_t)_{t\geq 0}$ with $X_t:=Y(N_t)$, $t\geq0$, is a Markov jump process which has infinitesimal generator $L$. Furthermore, if we define the kernels $K^{(n)}$, $n\in\N_0$ by $K^{(0)}(x,B):=1_B(x)$, $K^{(1)}:=K$ and, inductively, by 
 \[K^{(n+1)}(x,B):=\int_E K^{(n)}(y,B)K(x,dy)\,,\quad n\geq1\,,\]
 then 
 the semigroup $(T_t)_{t\geq0}$ belonging to $(X_t)_{t\geq 0}$ is given on $F\in L^2(\mu)$ by
 \begin{equation*}
  T_tF(x)=e^{-t}\sum_{n=0}^\infty\frac{t^n}{n!}\int_E F(y)K^{(n)}(x,dy)=e^{-t}\sum_{n=0}^\infty\frac{t^n}{n!}\bigl(K^{(n)}F\bigr)(x)\,,
 \end{equation*}
 where the limit is to be understood in the $L^2(\mu)$-sense.
Note that, in general, the \textit{domain} $\dom A$ of the generator $A$ of a (strongly continuous) semigroup $(T_t)_{t\geq0}$ of bounded linear operators on a Banach space $B$ is only a dense subset of $B$ on which $A$ acts as a closed operator. 
In particular, $A$ is an unbounded operator, whenever $\dom A\subsetneq B$ (see e.g. \cite[Chapter 1]{EtKu}). On the contrary, as remarked above, the operator $L$ defined by \eqref{defL} is always a bounded operator with $\dom L=L^2(\mu)$ so that we need not worry about domains. 
Note that the operator $L$ given by \eqref{defL} also defines a bounded operator on $L^1(\mu)$ such that $\norm{LF}_{L^1(\mu)}\leq 2\norm{F}_{L^1(\mu)}$ for $F\in L^1(\mu)$.
This will be used implicitly in the sequel.

Recall that the probability measure $\mu$ is called \textit{reversible} with respect to $L$, if for all $F,G\in L^2(\mu)$ the identity 
\begin{equation}\label{rev}
 \int_E G(LF)d\mu=\int_E F(LG)d\mu
\end{equation}
holds. Alternatively, this just means that $L$ is a (bounded) self-adjoint operator with domain $L^2(\mu)$. 
Since in our situation $L1=0$, \eqref{rev} implies the \textit{invariance} of $L$ with respect to $\mu$, i.e. for all $F\in L^2(\mu)$ 
\begin{equation}\label{inv}
 \int_E LF d\mu=0\,.
\end{equation}
For such an operator $L$ we define the corresponding symmetric and bilinear \textit{carr\'{e} du champ} operator $\Gamma:L^2(\mu)\times L^2(\mu)\rightarrow L^1(\mu)$  by
\begin{equation}\label{cdc}
 \Gamma(F,G):=\frac{1}{2}\bigl(L(FG)-F(LG)-G(LF)\bigr)\,.
\end{equation}
We write $\Gamma(F):=\Gamma(F,F)$ for short.
From \eqref{rev} and \eqref{inv} we then conclude the following important \textit{integration by parts formula}: For all $F,G\in L^2(\mu)$
\begin{equation}\label{intparts}
 \int_E \Gamma(F,G)d\mu=-\int_E G(LF)d\mu=-\int_E F(LG)d\mu\,.
\end{equation}
Reversibility and an explicit formula for $\Gamma$ are provided by the following result.

\begin{theorem}\label{abstheo}
 Under the above assumptions, the linear operator $L$ defined by \eqref{defL} satisfies $L1=0$ and $\mu$ is reversible with respect to $L$. Furthermore, the carr\'{e} du champ $\Gamma$ associated to $L$ is given by
\begin{align*}
 \Gamma(F,G)(x)
 &=\frac12\E\bigl[\bigl(F(X')-F(X)\bigr)\bigl(G(X')-G(X)\bigr)\,\bigl|\,X=x\bigr]\\
 &=\frac12\int_E\bigl(F(y)-F(x)\bigr)\bigl(G(y)-G(x)\bigr)K(x,dy) \,.
\end{align*}
\end{theorem}

\begin{proof}
We have already observed that $L$ is well-defined, i.e. maps into $L^2(\mu)$, linear and satisfies $L1=0$. 
To show reversibility let $F,G\in L^2(\mu)$. Then, as $\Prob\circ X^{-1}=\mu$,
 \begin{align}
  \int_E G(LF)d\mu&=\E\bigl[G(X)\bigl(LF(X)\bigr)\bigr]=\E\bigl[G(X)\bigl(\E[F(X')|X]-F(X)\bigr)\bigr]\notag\\
  &=\E\bigl[\E[G(X)F(X')|X]\bigr]-\E[G(X)F(X)]\notag\\
  &=\E[G(X)F(X')]\bigr]-\E[G(X)F(X)]\notag\\
  &=\E[F(X)G(X')]-\E[G(X)F(X)]\label{at2}\,,
 \end{align}
where we used exchangeability to obtain the last identity. Now, by the same computation with $F$ and $G$ interchanged we find that
\begin{equation*}
 \int_E F(LG)d\mu=\E[F(X)G(X')]-\E[G(X)F(X)]\,,
\end{equation*}
proving the reversibility of $\mu$ with respect to $L$. It remains to derive the precise representation of the carr\'{e} du champ operator $\Gamma$. 
By definition we have 
\begin{align*}
&2\Gamma(F,G)(x)=L(FG)(x)-F(x)LG(x)-G(x)LF(x)\\
&=K(FG)(x)-F(x)G(x)-F(x)KG(x)+F(x)G(x)-G(x)KF(x)+F(x)G(x)\\
&=K(FG)(x)-F(x)KG(x)-G(x)KF(x)+F(x)G(x)\\
&=\E\bigl[G(X')\bigl(F(X')-F(x)\bigr)\,\bigl|\,X=x\bigr]-G(x)\E\bigl[F(X')-F(x)\,\bigl|\,X=x\bigr]\\
&=\E\bigl[\bigl(G(X')-G(x)\bigr)\bigl(F(X')-F(x)\bigr)\,\bigl|\,X=x\bigr]\,.
\end{align*}
This concludes the proof.
\end{proof}

\begin{cor}\label{abscor}
 Under the assumptions of Theorem \ref{abstheo} we have for every $F\in L^2(\mu)$ that $\Gamma(F)\geq0$ pointwise. This implies that $-L$ is a bounded, positive and self-adjoint operator on $L^2(\mu)$ in the sense that 
 \begin{equation*}
  \int_E F\bigl(-L F\bigr)d\mu\geq0
 \end{equation*}
holds for all $F\in L^2(\mu)$.
\end{cor}

\begin{proof}
 We already know that $-L$ is a self-adjoint operator because $\mu$ is reversible with respect to it. Further, for $F\in L^2(\mu)$ and $x\in E$ we have 
 \begin{equation*}
 \Gamma(F)(x)=\E\bigl[\bigl(F(X')-F(X)\bigr)^2\,\bigl|\,X=x\bigr]\geq 0\,.
 \end{equation*}
Hence, the integration by parts formula \eqref{intparts} yields
\begin{equation*}
 0\leq \int_E\Gamma(F)d\mu=\int_E F\bigl(-L F\bigr)d\mu\,.
\end{equation*}
\end{proof}

\begin{remark}
\begin{enumerate}[(a)]
\item If we only assumed that $X$ and $X'$ were identically distributed, then, as 
\begin{equation*}
 \int_E LFd\mu=\E\bigl[LF(X)\bigr]=\E\Bigl[\E\bigl[F(X')-F(X)\,\bigl|\,X\bigr]\Bigr]=\E\bigl[F(X')-F(X)\bigr]=0\,,
\end{equation*}
the measure $\mu$ would still be invariant but in general no longer reversible with respect to $L$. In particular, the integration by parts formula \eqref{intparts}, which will be crucial for our further analysis, would no longer hold. Thus, we emphasize that the exchangeability assumption is really essential for what follows. In contrast, in \cite{Roll} a method has been developed that allows one to reduce the condition of exchangeability of $(W,W')$ to the identity of distributions $\mathcal{L}(W')=\mathcal{L}(W)$, if either \eqref{clinreg} or \eqref{linreg} holds. This strategy can be successfully adapted to most cases of approximation by a univariate absolutely continuous distribution (see e.g. \cite{DoeBeta, DP18b} for further examples). As can be seen from the proofs in \cite{Roll}, the (approximate) linearity of the 
exchangeable pair is crucial for this strategy to be successful. In \cite{ReiRoe09} it is moreover explained, why the identity of distributions is not enough in the general exchangeable pairs approach for multivariate normal distributions.
\item The quantity 
\[\mathcal{E}(F,F):=\int_E\Gamma(F)d\mu=\int_E F\bigl(-L F\bigr)d\mu\]
is usually called the \textit{energy} of the functional $F$ and the symmetric, bilinear mapping $\mathcal{E}$ with 
\[\mathcal{E}(F,G)=\int_E \Gamma(F,G)d\mu=-\int_E FLGd\mu\]
 is the \textit{Dirichlet form} associated to $L$ (or $\Gamma$). From Theorem \ref{abstheo} we obtain the formula
\[\mathcal{E}(F,G)=\frac12\int_E\int_E\bigl(F(y)-F(x)\bigr)\bigl(G(y)-G(x)\bigr)K(x,dy)\mu(dx) \,,\]
which reduces to the well-known expression 
\[\mathcal{E}(F,G)=\frac12\sum_{x,y\in E}\bigl(F(y)-F(x)\bigr)\bigl(G(y)-G(x)\bigr)\mu(x)K(x,y)\]
for countable state spaces $E$ (see e.g. \cite{DiaStr}).
\end{enumerate}
\end{remark}

For the development of our general theory it is important to find conditions under which the operator $L$ allows for a \textit{pseudo-inverse} 
$L^{-1}$. In order to do this, we will make use of the fact that $L$ is the generator of a continuous-time Markov process $(X_t)_{t\geq0}$ as above.  
Recall that the operator $L$ and also the Markov process $(X_t)_{t\geq0}$ are called \textit{ergodic}, if, for $F\in L^2(\mu)$, $LF=0$ implies that $F$ is $\mu$-a.e. equal to a constant $c\in\R$, i.e. if the kernel $\ker(L)$ of $L$ coincides with the class of $\mu$-a.e. constant functions on $E$. It is well-known that, for the generator $L$ of a Markov process $(X_t)_{t\geq0}$ which has a reversible probability distribution $\mu$, ergodicity of $L$ is equivalent to 
\begin{equation*}
 \lim_{t\to\infty}T_tF=\lim_{t\to\infty}\E\bigl[F(X_t)\,\bigl|\, X_0=(\cdot)\bigr]=\int_E F d\mu\quad\text{in }L^2(\mu)
 \end{equation*}
for each $F\in L^2(\mu)$. By Proposition 2.3 (b) of \cite{Bhat82} ergodicity of $(X_t)_{t\geq0}$ is also equivalent to the condition that the \textit{image} $\im(L)$ be dense in $\ker(E_\mu)=\{1\}^\perp$.
Furthermore, by Proposition 2.3 (c') of the same paper \cite{Bhat82}, under the assumption of ergodicity we have $\im(L)=\ker(E_\mu)=\{1\}^\perp$ if and only if $L$ has a spectral gap, i.e. if $0$ is an isolated point of the spectrum 
\begin{equation*}
\spec(L):=\{\lambda\in\R\,: (L-\lambda\Id)^{-1}\text{ does not exist as a bounded operator on }L^2(\mu)\}\,.
\end{equation*}
The existence of a spectral gap for a generator $L$ of a Markov semigroup with stationary distribution $\mu$ is in turn equivalent to the condition that a \textit{Poincar\'{e} inequality} for $\mu$ holds, i.e. that there is a $c\in[0,\infty)$ such that 
for all $F\in L^2(F)$ we have 
\begin{equation}\label{poincare}
 \Var_\mu(F)\leq c E_\mu\bigl[\Gamma(F,F)\bigr]\,,
\end{equation}
where $\Gamma$ is the carr\'{e} du champ associated to $L$ via \eqref{cdc}, see e.g. \cite{BGL14}. Note that for the generator $L$ given by \eqref{defL}, Theorem \ref{abstheo} implies that inequality \eqref{poincare} can be written as 
\begin{equation}\label{pcex}
 \Var\bigl(F(X)\bigr)\leq\frac{c}{2}\E\bigl[(F(X')-F(X)\bigr)^2\bigr]\,.
\end{equation}
Hence, we state the following additional assumption that will be needed for certain results.

\begin{assumption}\label{specgap}
The operator $L$ given by \eqref{defL} is ergodic, i.e. its kernel $\ker(L)$ only consists of the constant functions.
Furthermore, $L$ has a spectral gap at $0$ and we denote by $c^*\in(0,\infty)$ its \textit{Poincar\'{e} constant}, i.e. the smallest possible constant such that \eqref{pcex} holds for all $F\in L^2(\mu)$ with $c=c^*$.  
\end{assumption}
If Assumption \ref{specgap} holds, then 
\begin{equation*}
L_{|\{1\}^\perp}:\{1\}^\perp\rightarrow \{1\}^\perp
\end{equation*}
is a bijective and bounded linear operator whose image $\{1\}^\perp=\ker(E_\mu)$ is a close subspace of $L^2(\mu)$. Hence, its inverse 
$L_{|\{1\}^\perp}^{-1}$ exists as an operator on $\{1\}^\perp$ and, as a consequence of the open mapping theorem, it is bounded.  
We can extend it to an operator $L^{-1}$ on all of $L^2(\mu)$, called the \textit{pseudo-inverse} of $L$, by defining 
\[L^{-1}G:=L_{|\{1\}^\perp}^{-1}(G-E_\mu[G])\,,\quad G\in L^2(\mu)\,.\]
Then, $L^{-1}$ has the same operator norm as $L_{|\{1\}^\perp}^{-1}$. Indeed, if we denote the latter by $M$, then for all $G\in L^2(\mu)$ we have 
\begin{equation}\label{Linvbound}
\norm{L^{-1}G}_{L^2(\mu)}=\norm{L^{-1}(G-E_\mu[G])}_{L^2(\mu)}\leq M\norm{G-E_\mu[G]}_{L^2(\mu)}\leq M\norm{G}_{L^2(\mu)}\,.
\end{equation}

As in the diffusive framework, we have the following covariance lemma.
\begin{lemma}\label{covlemma}
 Let $F,G\in L^2(\mu)$. Then, we have 
 \begin{equation*}
  \Cov_{\mu}(F,G)=E_\mu\bigl[\Gamma(G,-L^{-1}F)\bigr]\,.
 \end{equation*}
In particular, 
\begin{equation*}
 \Var_\mu(F)=E_\mu\bigl[\Gamma(F,-L^{-1}F)\bigr]\,.
\end{equation*}
\end{lemma}

\begin{proof}
By the definition of $L^{-1}$ and the partial integration formula \eqref{intparts} we have that 
\begin{align*}
 \Cov_{\mu}(F,G)&=E_\mu\bigl[G(F-E_\mu[F])\bigr]=E_\mu\bigl[G \bigl(LL^{-1}F\bigr)\bigr]\\
 &=-E_\mu\Bigl[\Gamma\bigl(G, L^{-1}F\bigr)\Bigr]=E_\mu\bigl[\Gamma(G,-L^{-1}F)\bigr]\,.
\end{align*}

\end{proof}

Recall that, if $A$ is the infinitesimal generator of a Markov diffusion, then for each smooth enough function $\psi$ and all $F,G\in\A$, where $\A$ is a dense sub-algebra of $\dom A$,  and the associated carr\'{e} du champ $\Gamma_A$ we have
\begin{equation}\label{diffusive}
 \Gamma_A\bigl(\psi(F),G\bigr)=\psi'(F)\Gamma_A(F,G)\,.
\end{equation}
Due to its discrete construction, the operator $L$ from Theorem \ref{abstheo} is not diffusive in the sense of \eqref{diffusive} but it might be close enough to being diffusive. For a smooth enough function $\psi$, let us thus define a remainder term 
$R_\psi(F,G)$ by 
\begin{equation}\label{remdiff}
 R_\psi(F,G):=\Gamma\bigl(\psi(F),G\bigr)-\psi'(F)\Gamma(F,G)\,.
\end{equation}
If $F=G$, then we write $R_\psi(F):=R_\psi(F,F)$. Suppose now that $\psi$ is a continuously differentiable function on $\R$ whose derivative $\psi'$ is Lipschitz-continuous. Then, using Taylor's formula, for $x\in E$ we have 
\begin{align*}
 \Gamma(\psi(F),G)(x)&=\frac12\E\bigl[\bigl(\psi(F(X'))-\psi(F(X))\bigr)\bigl(G(X')-G(X)\bigr)\,\bigl|\,X=x\bigr]\notag\\
 &=\frac12\E\bigl[\psi'(F(X))\bigl(F(X')-F(X)\bigr)\bigl(G(X')-G(X)\bigr)\,\bigl|\,X=x\bigr]\notag\\
 &\;+\frac12\E\bigl[r_\psi\bigl(F(X),F(X')-F(X)\bigr)\bigl(G(X')-G(X)\bigr)\,\bigl|\,X=x\bigr]\notag\\
&=\psi'(F(x))\Gamma(F,G)(x)\notag\\
&\;+\frac12\E\bigl[r_\psi\bigl(F(X),F(X')-F(X)\bigr)\bigl(G(X')-G(X)\bigr)\,\bigl|\,X=x\bigr]\,,
\end{align*}
where 
\[r_\psi(x,y):=\psi(x+y)-\psi(x)-\psi'(x)y=\int_0^y(y-s)\psi''(x+s)ds\,.\]
Hence, $r_\psi$ 
satisfies 
\[\abs{r_\psi(x,y)}\leq \frac{\fnorm{\psi''}}{2}y^2\,,\quad x,y\in\R\,,\]
and
\begin{equation*}
 \fnorm{\psi''}:=\sup_{x\not=y}\frac{\abs{\psi'(y)-\psi'(x)}}{\abs{y-x}}
\end{equation*}
denotes the smallest Lipschitz constant of $\psi'$.
Hence, for such a function $\psi$, we obtain for each $x\in E$ that 
\begin{align}\label{boundRpsiFG}
 \babs{R_\psi(F,G)(x)}&\leq\frac{\fnorm{\psi''}}{4}\E\Bigl[\babs{G(X')-G(X)}\bigl(F(X')-F(X)\bigr)^2\,\Bigl|\, X=x\Bigr]\,. 
\end{align}
In particular, we have 
\begin{equation}\label{boundRpsiF}
 \babs{R_\psi(F)(x)}\leq\frac{\fnorm{\psi''}}{4}\E\Bigl[\babs{F(X')-F(X)}^3\,\Bigl|\,X=x\Bigr]\,,\quad x\in E\,.
\end{equation}

\subsection{Normal approximation results}\label{ssnormal}
In this subsection we give the main abstract results on normal approximation of this paper and relate them to existing ones. The proofs are genrally deferred to Section \ref{proofs} in order to increase readability. We have chosen to present our results both in abstract terms of operators and functionals as well as in terms of exchangeable pairs. For this purpose the following lemma, which 
is a direct generalization of \cite[Lemma 2.2]{DP16}, will be very useful.

\begin{lemma}\label{remlemma}
Assume that $F\in L^4(\mu)$. Then, the following identity holds true:
\begin{align*}
\E\bigl[\bigl(F(X')-F(X)^4\bigr)\bigr]&=
4\Bigl(E_\mu\bigl[F^3 LF\bigr]+3E_\mu\bigl[F^2\Gamma(F,F)\bigr]\Bigr)\\
&=4\Bigl(3E_\mu\bigl[F^2\Gamma(F,F)\bigr]-E_\mu\bigl[\Gamma(F^3,F)\bigr]\Bigr)\,.
\end{align*}
If, moreover, $F$ is an eigenfunction of $-L$ corresponding to the eigenvalue $\lambda>0$, then 
\begin{align*}
\E\bigl[\bigl(F(X')-F(X)^4\bigr)\bigr]&=4\lambda \Bigl(3E_\mu\bigl[F^2\lambda^{-1}\Gamma(F,F)\bigr]- E_\mu\bigl[F^4\bigr] \Bigr)\,.
\end{align*}
\end{lemma}

\begin{remark}\label{restrem}
 Note that, if $L$ were diffusive, then the rightmost term in Lemma \ref{remlemma} would be equal to zero. Hence, if it is close enough to being diffusive, then one should at least be able to control this quantity. 
\end{remark}

From now on, we denote by $\gamma$ the standard normal distribution on $\R$ and by $\phi$ we denote its continuous density function, i.e.
\[\phi(x)=\frac{1}{\sqrt{2\pi}}e^{-x^2/2}\quad\text{and}\quad \gamma(dx)=\phi(x)dx\,.\]
In what follows we make use of Stein's method in order to estimate the Wasserstein distance of the distribution of a given random variable $F\in L^2(\mu)$ to the standard normal distribution $\gamma$. 
For $h\in\Lip(1)$ consider the corresponding standard normal Stein equation
\begin{equation}\label{steineq}
 \psi'(x)-x\psi(x)=h(x)-\gamma(h)\,,
\end{equation}
where we write $\gamma(h)=\int hd\gamma$. It is well-known (see e.g. \cite{CGS}) that the function $\psi_h$ given by 
\begin{equation}\label{steinsol}
 \psi_h(x)=\frac{1}{\phi(x)}\int_{-\infty}^x\bigl(h(t)-\gamma(h)\bigr)\phi(t)dt\,,\quad x\in\R\,,
\end{equation}
is a bounded, continuously differentiable solution to \eqref{steineq} such that both $\psi_h$ and $\psi_h'$ are Lipschitz-continuous. More prescisely, one has that 
\begin{equation}\label{bounds}
 \fnorm{\psi_h}\leq1\,,\quad\fnorm{\psi_h'}\leq\sqrt{\frac{2}{\pi}}\quad\text{and}\quad \fnorm{\psi_h''}\leq2\,.
\end{equation}

The following theorems provide upper bounds on the quantity $d_\W(\mu\circ F^{-1},\gamma)$ for general random variables $F\in L^2(\mu)$. We will always assume that  
\begin{equation}\label{stand}
 E_\mu[F]=0\quad\text{and}\quad\Var_\mu(F)=E_\mu[F^2]=1\,.
\end{equation}

\begin{theorem}\label{genbound1}
 Assume that Assumption \ref{specgap} holds. Then, for $F\in L^3(\mu)$ which satisfies \eqref{stand} we have the bound 
 \begin{align}
 d_\W(\mu\circ F^{-1},\gamma)&\leq\sqrt{\frac{2}{\pi}} E_\mu\babs{1-\Gamma(F,-L^{-1}F)}\notag\\
 &\;+\frac12\E\Bigl[\babs{L^{-1}F(X')-L^{-1}F(X)}\bigl(F(X')-F(X)\bigr)^2\Bigr]\label{gb11}
\end{align}
If, in fact $F\in L^4(\mu)$, then one has the further bound 
\begin{align}
d_\W(\mu\circ F^{-1},\gamma)&\leq \sqrt{\frac{2}{\pi}}\sqrt{\Var_\mu\bigl(\Gamma(F,-L^{-1}F)\bigr)}\notag\\
&\;+\sqrt{-E_\mu\bigl[FL^{-1}F\bigr]}\sqrt{E_\mu\bigl[F^3 LF\bigr]+3E_\mu\bigl[F^2\Gamma(F,F)\bigr]}\label{gb12}\,.
 \end{align}
 \end{theorem}

\begin{remark}\label{gb1rem}
\begin{enumerate}[(a)]
\item In view of Remark \ref{restrem} the abstract bound \eqref{gb12} in Theorem \ref{genbound1} can be interpreted as follows: The functional $F$ is approximately standard normally distributed, whenever 
$\Gamma(F,-L^{-1}F)$ is close to its expected value $1$ and if $L$ is not too far from being diffusive. Note that the latter condition rather concerns the exchangeable pair $(X,X')$ than the functional $F$ itself. Hence, it is not surprising that the same condition appears when approximating by other distributions (see Subsection \ref{ssgamma} below). Typically, in such cases the first condition must be 
replaced by demanding that $\Gamma(F,-L^{-1}F)$ be close to some non-constant polynomial $\tau(F)$. 
\item In this paper we focus on normal approximation bounds in terms of the Wasserstein distance. For several applications, in particular those coming from statistics, the \textit{Kolmogorov distance}, which is the supremum norm distance between the involved distribution functions, is more natural. In Stein's method it is in general much more difficult to provide bounds on the Kolmogorov distance, which are of the same order as the Wasserstein bounds. This is because the solution to the Stein equation for the Kolmogorov distance does not have a Lipschitz derivative and this lack of smoothness needs to be addressed by a separate means. In certain scenarios people have been succesful in providing Kolmogorov bounds of the same order (see e.g. \cite{CheSha, Schulte, DP18, DK19}). 
However, in the exchangeable pairs approach, even under the (approximate) linearity condition \eqref{linreg}, to date Kolmogorov bounds of a comparable order as for the Wasserstein distance have only been obtained under additional assumptions like the boundedness of the difference $W'-W$ or some extra sign conditions, for instance (see e.g. \cite{RiRo97, ShaZh}). Thus, we consider obtaining comparable bounds on the Kolmogorov distance as an (important) separate problem for future work on the topic.
\end{enumerate}
\end{remark}

The following result is most useful for applications. Below we explain in which respect it is an essential and practical
generalization of known "plug-in results" for exchangeable pairs satisfying a linear regression property. 

\begin{theorem}\label{genbound2}
 Assume that $F\in L^3(\mu)$ satisfying \eqref{stand} can be written as $F=\sum_{p=1}^m F_p$, where $m\in\N$ and $F_p\in\ker(L+\lambda_p\Id)$ are eigenfunctions corresponding to pairwise distinct eigenvalues $\lambda_p>0$ of $-L$
 , $p=1,\dotsc,m$. Then, we have the bound
 \begin{align}
  &d_\W(\mu\circ F^{-1},\gamma)\leq\sqrt{\frac{2}{\pi}} E_\mu\babs{1-\Gamma(F,-L^{-1}F)}\notag\\
	&\;+\frac12\sum_{p,q,r=1}^m\lambda_p^{-1}\E\bigl[\bigl(F_q(X')-F_q(X)\bigr)\bigl(F_r(X')-F_r(X)\bigr)\babs{F_p(X')-F_p(X)}\bigr]\label{gb21}
	\end{align}
	If, in fact $F\in L^4(\mu)$, then we have the further bound
	\begin{align}
	&d_\W(\mu\circ F^{-1},\gamma)	\leq\sqrt{\frac{2}{\pi}}\biggl(\Var_\mu\Bigl(\sum_{p,q=1}^m\lambda_p^{-1}\Gamma(F_p,F_q)\Bigr)\biggr)^{1/2}\notag\\
	&\;+\sqrt{2}\sum_{p=1}^m\lambda_p^{-1/2}\sqrt{E_\mu[F_p^2]} \Biggl(\sum_{q=1}^m \lambda_q^{1/4}\Bigl(3E_\mu\bigl[F_q^2\lambda_q^{-1}\Gamma(F_q,F_q)\bigr]- E_\mu\bigl[F_q^4\bigr]\Bigr)^{1/4}\Biggr)^2\label{gb22}\,.
 \end{align}
\end{theorem}

The following result is a rephrased version of Theorem \ref{genbound2} that is more similar to existing normal approximation results for exchangeable pairs.  
\begin{cor}\label{genexpair2}
 Suppose that, on some probability $(\Omega,\F,\P)$,  $(X,X')$ is an exchangeable pair of $E$-valued random variables with marginal distribution $\mu$. Moreover, assume that $m$ is a positive integer and that, for $1\leq p\leq m$, $F_p\in L^2(\mu)$ is a functional such that the 
 exchangeable pairs $(W_p,W_p'):=(F_p(X),F_p(X'))$ satisfy 
 \[\E\bigl[W_p'-W_p\,\bigl|\,X\bigr]=-\lambda_p W_p\,,\]
 where $\lambda_1,\dotsc,\lambda_m$ are pairwise distinct positive numbers. Define $W:=\sum_{p=1}^m W_p$ and assume that $\E[W]=0$ and 
 $\Var(W)=\sum_{p=1}^m\Var(W_p)=1$. Then, one has the bounds 
 \begin{align*}
  &d_\W(\P\circ W^{-1},\gamma)\leq \sqrt{\frac{2}{\pi}}\biggl(\Var\Bigl(\sum_{p,q=1}^m\frac{1}{2\lambda_p}\E\bigl[\bigl(W_p'-W_p\bigr)\bigl(W_q'-W_q\bigr)\,\bigl|\, X\bigr]\Bigr)\biggr)^{1/2}\notag\\
&\;+\frac12\sum_{p,q,r=1}^m\lambda_p^{-1}\E\bigl[\bigl(W'_q-W_q\bigr)\bigl(W'_r-W_r\bigr)\babs{W'_p-W_p}\bigr]\notag\\
&\leq \sqrt{\frac{2}{\pi}}\biggl(\Var\Bigl(\sum_{p,q=1}^m\frac{1}{2\lambda_p}\E\bigl[\bigl(W_p'-W_p\bigr)\bigl(W_q'-W_q\bigr)\,\bigl|\, X\bigr]\Bigr)\biggr)^{1/2}\notag\\
&\;+\sqrt{2}\sum_{p=1}^m\lambda_p^{-1/2}\sqrt{\E[W_p^2]} \Biggl(\sum_{q=1}^m \lambda_q^{1/4}\Bigl(3\E\Bigl[W_q^2\frac{1}{2\lambda_q}\E\bigl[\bigl(W_q'-W_q\bigr)^2\,\bigl|\, X\bigr]\Bigr]- \E\bigl[W_q^4\bigr]\Bigr)^{1/4}\Biggr)^2\,.
 \end{align*}

\end{cor}

\begin{remark}\label{expairrem}
\begin{enumerate}[(a)]
\item We stress here again that the bound given in Corollary \ref{genexpair2} is a direct generalization of the classical bound \eqref{cb1} on the Wasserstein distance for exchangeable pairs satisfying the linear regression property \eqref{clinreg}. 
\item It can be further estimated by using the inequality
\begin{align*}
 &\biggl(\Var\Bigl(\sum_{p,q=1}^m\frac{1}{2\lambda_p}\E\bigl[\bigl(W_p'-W_p\bigr)\bigl(W_q'-W_q\bigr)\,\bigl|\, X\bigr]\Bigr)\biggr)^{1/2}\\
 &\leq \sum_{p,q=1}^m \sqrt{\Var\Bigl(\frac{1}{2\lambda_p}\E\bigl[\bigl(W_p'-W_p\bigr)\bigl(W_q'-W_q\bigr)\,\bigl|\, X\bigr]\Bigr)}
\end{align*}
\item Note that under the assumptions of Corollary \ref{genexpair2}, the summands $W_1,\dotsc,W_m$ making up $W$ are automatically orthogonal in $L^2(\P)$. In order for this to hold it is crucial that the pair $(X,X')$ with values 
in the abstract space $E$ is indeed exchangeable and not just identically distributed. Our abstract, functional analytic viewpoint clarifies this by identifying $F_1,\dotsc,F_m$ as eigenfunctions of the self-adjoint operator $-L$
corresponding to different eigenvalues $\lambda_1,\dotsc,\lambda_p$.

\end{enumerate}
\end{remark}

The next result does not make use of the operator $L^{-1}$. Hence, it also holds in situations, where Assumption \eqref{specgap} does not hold or where the computation of $L^{-1}F$ is not tractable.

\begin{theorem}\label{genbound3}
Let $F\in L^3(\mu)$ satisfy \eqref{stand} and define $\lambda:=E_\mu\bigl[\Gamma(F)\bigr]=\mathcal{E}(F,F)\geq0$. Then, assuming that $\lambda>0$, we have the bounds
\begin{align}
 d_\W(\mu\circ F^{-1},\gamma)&\leq\sqrt{\frac{2}{\pi}}\frac{1}{\lambda}E_\mu\babs{\lambda -\Gamma(F)}
 +\frac{1}{2\lambda}\E\babs{F(X')-F(X)}^3\notag\\
 &\;+E_\mu\babs{\frac{1}{\lambda}LF+F}\label{gb31}\\
 &\leq \sqrt{\frac{2}{\pi}}\frac{1}{\lambda}\Bigl(\Var_\mu\bigl(\Gamma(F)\bigr)\Bigr)^{1/2}+\frac{1}{2\lambda}\E\babs{F(X')-F(X)}^3\notag\\
 &\;+\Bigl(\frac{1}{\lambda^2}E_\mu\bigl[(LF)^2\bigr]-1\Bigr)^{1/2}\label{gb32}\,.
\end{align}
If $F\in L^4(\mu)$, then \eqref{gb32} can be further bounded to yield
\begin{align}\label{gb33}
d_\W(\mu\circ F^{-1},\gamma)&\leq\sqrt{\frac{2}{\pi}}\frac{1}{\lambda}\Bigl(\Var_\mu\bigl(\Gamma(F)\bigr)\Bigr)^{1/2}+\Bigl(\frac{1}{\lambda^2}E_\mu\bigl[(LF)^2\bigr]-1\Bigr)^{1/2}\notag\\
&\;+ \frac{\sqrt{2}}{\sqrt{\lambda}}\sqrt{E_\mu\bigl[F^3 LF\bigr]+3E_\mu\bigl[F^2\Gamma(F,F)\bigr]}\,.
\end{align}
\end{theorem}

Let us now link the above Theorems \ref{genbound1}, \ref{genbound2} and \ref{genbound3} to the existing existing literature on Stein's method of exchangeable pairs. In order to see the connection, first assume that  
the function $F\in L^2(\mu)$ is an eigenfunction of the operator $-L$ corresponding to an eigenvalue $\lambda>0$, i.e.
\begin{equation}\label{eigeneq}
 -\lambda F=LF=\E\bigl[F(X')\,\bigl|\,X=(\cdot)\bigr]-F
\end{equation}
in the $L^2(\mu)$-sense. This just means that the exchangeable pair 
\begin{equation*}
 (W,W'):=\bigl(F(X),F(X')\bigr)
\end{equation*}
satisfies Stein's famous \textit{linear regression property} \eqref{clinreg}:
\begin{equation*}
 \E\bigl[W'-W\,\bigl|\,W\bigr]=\E\bigl[W'-W\,\bigl|\,X\bigr]=-\lambda W\quad \Prob\text{-a.s.} 
\end{equation*}
In this situation, both bounds \eqref{gb11} and \eqref{gb21} reduce to the classical bound \eqref{cb1}:
\begin{align*}
 d_\W(\Prob\circ W^{-1},\gamma)&\leq\Bigl(\Var\Bigl(\frac{1}{2\lambda}\E\bigl[(W'-W)^2\,\bigl|\,X\bigr]\Bigr)\Bigr)^{1/2}+\frac{1}{2\lambda}\E\babs{W'-W}^3\,.
\end{align*}
Similarly, the bounds in Theorem \ref{genbound3} can be related to the bound \eqref{cb2} of Rinott and Rotar \cite{RiRo97}, that is 
\begin{align*}
d_\W(W,Z)&\leq \sqrt{\frac{2}{\pi}}\E\Babs{1-\frac{1}{2\lambda}\E\bigl[(W'-W)^2\,\bigl|\,W\bigr]} +\frac{1}{3\lambda}\E\babs{W'-W}^3+\frac{1}{\lambda}\E\abs{R}\,,
\end{align*}
where, for $\lambda>0$, the remainder term $R$ is defined by \eqref{linreg}:
\[\E\bigl[W'-W\,\bigl|\,X\bigr]=-\lambda W+R\quad \Prob\text{-a.s.}\]
Indeed, in the proof of Theorem \ref{genbound2} and with $\lambda:=E_\mu[\Gamma(F)]$, we use 
\begin{equation*}
 \tilde{R}:=LF+\lambda F
\end{equation*}
and then we can let $R=\tilde{R}(X)$. 
Observe that the terms in the bounds \eqref{gb31} and \eqref{cb2} that involve the terms $\tilde{R}$ and $R$ respectively, can only vanish in the limit if these quantities are negligible as compared to $\lambda$. This respectively means that $F$ is "almost" an eigenfunction of $L$ to the eigenvalue $\lambda$ and $(W,W')$ "almost" satisfies \eqref{linreg}.

Note that the constants in the second terms of the respective bounds \eqref{gb31} and \eqref{cb2} are different. This is due to the fact that, for the bound \eqref{cb2}, we did not impose the condition 
\[\lambda =\frac{1}{2}\E\bigl[\E\bigl[(W'-W)^2\,\bigl|\,X\bigr]\bigr]=\frac12\E\bigl[(W'-W)^2\bigr]\]
which is automatically satisfied by the choice of $\lambda$ in Theorem \ref{genbound3}. However, since $\lambda$ always is 
asymptotically unique, as has been observed in the introduction of \cite{ReiRoe09}, the two bounds are nearly the same.\\

Finally, we remark that an alternative way of obtaining a (quantitative) CLT for $F$ in the situation of Theorem \ref{genbound2} is to introduce the exchangeable pair $(V,V')$ of $m$-dimensional vectors $V:=(F_1(X),\dotsc,F_m(X))$ and $V':=(F_1(X'),\dotsc,F_m(X'))$ which satisfies the multivariate linear regression property 
\begin{align*}
\E\bigl[V'-V\,\bigl|\, X\bigr]=-\Lambda V\,,\quad\text{where } \Lambda=\diag(\lambda_1,\dotsc,\lambda_p)\,.
\end{align*}
Then, multivariate analogs of the bound \eqref{cb1} (see e.g. \cite{ChaMe, ReiRoe09, Me09}) yield bounds on the error of approximating  $V$ by a multivariate normal distribution with the same (diagonal) covariance matrix as $V$. Hence, by applying the linear form 
$(x_1,\dotsc,x_m)^T\mapsto\sum_{p=1}^m x_p$ to $V$ as well as to the limiting distribution, one also obtains a quantitative CLT for $F=\sum_{p=1}^mF_p$. This route was taken in the recent paper \cite{DP19} in order to obtain analytic bounds on the normal approximation of symmetric, non-degenerate $U$-statistics of i.i.d. random variables. 
There are, however, two main drawbacks to this approach. Firstly, the multivariate detour necessarily imposes stronger smoothness conditions 
on the class of test functions. This is why the bounds in \cite{DP19} are stated in terms of $C^3$ test functions with bounded derivatives. In contrast, the bounds in the present paper are on the Wasserstein distance, which only involves Lipschitz-continuity. We mention in this context that, very recently, in \cite{FaKo} a general bound on the multivariate normal approximation in Wasserstein distance via (approximately) linear exchangeable pairs has been given. In principle, one could apply such a bound in the same way as in \cite{DP19} to the pair $(V,V')$ from above in order to obtain univariate bounds for non-linear exchangeable pairs. However, the Wasserstein bound in \cite{FaKo} is valid only for a non-singular limiting covariance matrix. This implies that one would at least have to be able to identify those components that do not vanish in the limit in order to make this strategy work. This might however not always be feasible, for example because these components might change according to varying parameter regimes. This is for instance the case in the situation of subgraph counts in geometric random graphs, considered in Section \ref{apps}. On th contrary, the bound in 
Corollary \ref{genexpair2} are completely rigid with respect to such phenomena.

Secondly, there might be cases in which a multivariate CLT for the vector $V$ would not even hold, but, due to some cancellation effects, 
the sum $F$ is still asymptotically normal. In such a situation the approach of the present paper might still yield the CLT.

\subsection{Approximation by other distributions}\label{ssgamma}
The methods developed in this paper are by no means restricted to normal approximation but may be easily extended to the approximation by any absolutely continuous distribution, for which a version of Stein's method has been established. We exemplify this by considering the family of centered Gamma distributions (see e.g. \cite{DP18b, NouPec09a}).
Thus, we denote by $\Gammabar(\nu)$ the so-called \textit{centered Gamma distribution} with parameter $\nu>0$, which by definition is the distribution of 
\[Z_\nu:=2X_{\nu/2,1}-\nu\,,\]
where,  $X_{\nu/2,1}$ has distribution $\Gamma(\nu/2,1)$. Here, for $r,\lambda\in (0,\infty)$, we let $\Gamma(r,\lambda)$ be the usual \textit{Gamma distribution} with \textit{shape parameter} $r$ and \textit{rate} $\lambda$ which has probability density function (p.d.f.)
\begin{equation*}
p_{r,\lambda}(x)=\begin{cases}
                  \frac{\lambda^r}{\Gamma(r)}x^{r-1}e^{-\lambda x}\,,&\text{if }x>0\\
                  0\,,&\text{otherwise,}
                 \end{cases}
\end{equation*}
where 
\begin{equation*}
 \Gamma(t):=\int_0^\infty x^{t-1}e^{-x}dx\,.
\end{equation*}

 Notice that, if $\nu$ is an integer, then $\Gammabar(\nu)$ has a centered $\chi^2$ distribution with $\nu$ degrees of freedom. It is obvious that one has
\begin{equation*}
 \E[Z_\nu]=0\quad\text{and}\quad \Var(Z_\nu)=\E[Z_\nu^2]=2\nu\,.
\end{equation*}

A suitable version of Stein's method for $\Gammabar(\nu)$ that is valid on the whole real line has been developed in the recent paper \cite{DP18b}. There, the authors considered the Stein equation 
\begin{equation}\label{steineqgamma}
 2(x+\nu)\psi'(x)-x\psi(x)=h(x)-\E\bigl[h(Z_\nu)\bigr]\,,
\end{equation}
which is defined for all $x\in\R$ and which, as opposed to previous work \cite{NouPec09a, PTh} crucially features a linear coefficient of $\psi'$. We refer to \cite{DP18b} for a discussion of the importance of a linear coefficient of $\psi'$ in the Stein equation. 

In \cite{DP18b}, the exchangeable pairs approach for centered Gamma approximation has also been established. We review their main approximation result in this context:
Suppose that $W,W'$ are identically distributed real-valued random variables on the probability space $(\Omega,\F,\Prob)$ such that $\E[W^2]<\infty$. Assume that $\G$ is a sub-$\sigma$-field of $\F$ such 
that $\sigma(W)\subseteq\G$. Given a real number $\lambda>0$, define the random variables $R$ and $S$ via the \textit{regression equations}
\begin{align}
 \frac{1}{\lambda}\E\bigl[W'-W\,\bigl|\,\G\bigr]&=-W+R\quad\text{and}\label{linreg2}\\
 \frac{1}{2\lambda}\E\bigl[(W'-W)^2\,\bigl|\,\G\bigr]&=2(W+\nu)+S\label{regprop2}\,.
\end{align}
The following result is taken from \cite{DP18b}.

\begin{prop}[Theorem 2.5 of \cite{DP18b}]\label{gammaDP}
Let $W$ and $W'$ be as above and assume that $h$ is continuously differentiable on $\R$ and that both $h$ and $h'$ are Lipschitz-continuous. Then, we have the bound
\begin{align*}
 \Babs{\E\bigl[h(W)\bigr]-\E\bigl[h(Z_\nu)\bigr]}&\leq\fnorm{h'}\bigl(\max(1,2\nu^{-1})\E\abs{S}+\E\abs{R}\bigr)\\
 &\;+\frac{\max\bigl(1,\frac{2}{\nu}\bigr)\fnorm{h'}+\fnorm{h''}}{6\lambda}\E\babs{W'-W}^3\,.
\end{align*}
If, moreover, $\E[W^2]=2\nu$ and \eqref{linreg2} holds with $R=0$, then, since $\E[S]=0$ in this case, we also have the bound
\begin{align*}
 \Babs{\E\bigl[h(W)\bigr]-\E\bigl[h(Z_\nu)\bigr]}&\leq\max\Bigl(1,\frac{2}{\nu}\Bigr)\fnorm{h'}\sqrt{\Var(S)}\notag\\
 &\;+\frac{\max\bigl(1,\frac{2}{\nu}\bigr)\fnorm{h'}+\fnorm{h''}}{6\lambda}\E\babs{W'-W}^3\,.
\end{align*}
\end{prop}

Proposition \ref{gammaDP} relies on the following Proposition also proved in \cite{DP18b}, which provides suitable bounds on the solution to \eqref{steineqgamma} under smoothness assumptions on $h$.
It is also key to proving Theorem \ref{genboundgamma} below.

\begin{prop}[Theorem 2.3 of \cite{DP18b}]\label{gammabounds}
 \begin{enumerate}[{\normalfont (a)}]
 \item Let $h$ be Lipschitz-continuous on $\R$. 
 Then, there exists a Lipschitz-continuous solution $\psi_h$ of \eqref{steineqgamma} on $\R$ which satisfies the bounds
 { \[\fnorm{\psi_h}\leq\fnorm{h'}\quad\text{and}\quad \fnorm{\psi_h'}\leq\max\Bigl(1,\frac{2}{\nu}\Bigr) \fnorm{h'}\,.\]}
 \item Suppose that $h$ is continuously differentiable on $\R$ and that both $h$ and $h'$ are Lipschitz-continuous. Then, there is a continuously differentiable solution $\psi_h$ of \eqref{steineqgamma} on $\R$ whose 
 derivative $\psi_h'$ is Lipschitz-continuous with minimum Lipschitz constant
{ \[\fnorm{\psi_h''}\leq \max\Bigl(1,\frac{2}{\nu}\Bigr)\fnorm{h'}+\fnorm{h''}\,.\]}
\end{enumerate}
\end{prop}

Using the bounds on the solution $\psi_h'$ given in Proposition \ref{gammabounds}, the following version of Theorem \ref{genbound1} for centered Gamma approximation, which does not depend on \eqref{linreg2} and \eqref{regprop2}, can easily be proved. Gamma versions of the other approximation results in Subsection \ref{ssnormal} can be derived similarly.
\begin{theorem}\label{genboundgamma}
 Assume that Assumption \ref{specgap} holds. Suppose that $h:\R\rightarrow\R$ is continuously differentiable with a Lipschitz derivative $h'$.
 Then, for $F\in L^3(\mu)$ which satisfies $E_\mu[F]=0$ and $\Var_\mu(F)=2\nu$ we have the bound 
 \begin{align*}
 &\Babs{E_\mu\bigl[h(F)\bigr]-\Gammabar(\nu)(h)}  \leq\max\Bigl(1,\frac{2}{\nu}\Bigr) \fnorm{h'}\,E_\mu\babs{2(F+\nu)-\Gamma(F,-L^{-1}F)}\notag\\
 &\;+\Bigl(\max\Bigl(\frac14,\frac{1}{2\nu}\Bigr)\fnorm{h'}+\frac14\fnorm{h''}\Bigr)\,\E\Bigl[\babs{L^{-1}F(X')-L^{-1}F(X)}\bigl(F(X')-F(X)\bigr)^2\Bigr]\,.
\end{align*}
If, in fact $F\in L^4(\mu)$, then one has the further bound 
\begin{align*}
&\Babs{E_\mu\bigl[h(F)\bigr]-\Gammabar(\nu)(h)}  \leq \max\Bigl(1,\frac{2}{\nu}\Bigr) \fnorm{h'}\sqrt{\Var_\mu\bigl(2F-\Gamma(F,-L^{-1}F)\bigr)}\notag\\
&\;+\Bigl(\max\Bigl(\frac14,\frac{1}{2\nu}\Bigr)\fnorm{h'}+\frac14\fnorm{h''}\Bigr)\sqrt{-E_\mu\bigl[FL^{-1}F\bigr]}\sqrt{E_\mu\bigl[F^3 LF\bigr]+3E_\mu\bigl[F^2\Gamma(F,F)\bigr]}\,.
 \end{align*}
 \end{theorem}

We finish this subsection with the remark that, similarly, results analogous to Theorem \ref{genboundgamma} can be proved for many other absolutely continuous distributions like for example the Beta distribution. This is done by combining the methods of the present paper with bounds on the solution to the Beta Stein equation as provided in \cite{GRbeta, DoeBeta, Doe12c}. 

In general, roughly, the term $2(F+\nu)$ then must be replaced 
by a term $\tau(F)$, where $\tau$ is the celebrated \textit{Stein factor} or \textit{Stein kernel} (see e.g. \cite{LRS}), defined by 
\[\tau(x)=\frac{-1}{p(x)}\int_a^x tp(t)dt=\frac{1}{p(x)}\int_x^b tp(t)dt\,,\]
 where $p$ is the density function of the respective limiting distribution. Here, we assume that $p$ is supported on the interval $\overline{(a,b)}$, $-\infty\leq a<0<b\leq+\infty$, as well as that it is centered.

Finally, we mention that the approach developed in the present paper can be naturally extended to the approximation by a centered, multivariate normal distribution. This is dealt with 
in a separate paper in progress.

\subsection{Solution to a problem posed by Stein}\label{stproblem}
Our setup is related to the abstract framework provided in Stein's classic monograph \cite{St86}, where the following simple but important observation was made: 
Whenever $(X,X')$ is an exchangeable pair of $E$-valued random elements, i.e. 
\begin{equation*}
 (X,X')\stackrel{\D}{=}(X',X)\,,
\end{equation*}
and  $S:E\times E\rightarrow\R$ is antisymmetric in the sense that $S(y,x)=-S(x,y)$ for all $x,y\in E$ and, if furthermore $\E\babs{S(X,X')}<\infty$, then the function $G:E\rightarrow\R$ defined by 
\begin{equation}\label{antisym}
 G(x):=\E\bigl[S(X,X')\,\bigl|\, X=x\bigr]\,,\quad x\in E\,,
\end{equation}
is in the kernel $\ker E_\mu$ of the expectation operator $E_\mu$ associated with the measure $\mu$, i.e. 
\begin{equation}\label{meanzero}
 E_\mu[G]=0\,.
\end{equation}
Since the appearance of \cite{St86}, the method of exchangeable pairs has been intensively studied in the context of uni- and multivariate normal approximation as well as for approximation by other absolutely continuous distributions. 
The emphasis, however, has always been placed on the distributional behaviour of a single random variable $W\in\R$ or random vector $W\in\R^d$ which is member of an exchangeable pair $(W,W')$, even though it might have been constructed as a functional 
$(W,W')=(F(X),F(X'))$ of some other exchangeable pair $(X,X')$ on a more complicated state space $E$. 
As a consequence, in most papers about Stein's method of exchangeable pairs, Property \eqref{meanzero} is only used for the antisymmetric function $S$ of two real variables given by 
\begin{equation*}
 S(w,w'):=(w'-w)(f(w')+f(w))\,,
\end{equation*}
or for some multivariate analogue of it, where $f=f_h$ is the solution of the Stein equation associated to the limiting distribution and the real test function $h$ and \eqref{antisym} and \eqref{meanzero} are thus applied to $(W,W')$ directly rather than to $(X,X')$. \\

However, Stein \cite{St86} also raises the natural question, under what additional conditions on the space $(E,\mathcal{E})$ and on the exchangeable pair $(X,X')$ all random variables $G\in\ker E_\mu$ are of the form \eqref{antisym} for some antisymmetric function $S$. 
He gives a positive answer only in the case of a finite space $E$ that additionally satisfies the following \textit{connectedness condition} related to the pair $(X,X')$: 
For all $x,y\in E$ there is a finite sequence $x=x_0,x_1,\dotsc,x_k=y$ in $E$ such that for all $0\leq j\leq k-1$ one has 
\[\P(X=x_j,X'=x_{j+1})>0\,.\]
Translated into the language of the present paper this just means that for all $0\leq j\leq k-1$ one has
\[0<\P(X=x_j,X'=x_{j+1})=\mu(x_j)K(x_j,x_{j+1})\,.\]
It is easy to see that this connectedness condition is in fact equivalent to the irreducibility of the stochastic matrix $K$. Moreover, Stein \cite{St86} suspects that an
\begin{center}
\textit{``analogous result must be true in infinite cases.''} 
\end{center}
Indeed, on page 32 of \cite{St86} he writes
\begin{center}
\textit{``Although this result is never needed when applying this method,
one may ask, for the sake of completeness, whether it also holds in the
countable case, and whether an appropriate analogue can be formulated for the
general case.''}
\end{center}
As we have seen above, in the general case this result in particular holds, when the operator $L$ associated to the pair $(X,X')$ satisfies a Poincar\'e inequality. Indeed, under this assumption, for a given $G\in\ker E_\mu$, we even have the explicit expression 
\begin{equation}\label{Sform}
S(x,x')= L\bigl(L^{-1}G\bigr)(x')-L\bigl(L^{-1}G\bigr)(x)
\end{equation}
for $S$. 
Note that an irreducible and reversible continuous time Markov chain on a finite state space is always ergodic, has a spectral gap and thus 
satisfies a Poincar\'{e} inequality. Hence, Assumption \ref{specgap} reduces to Stein's connectedness condition in the case of a finite state space and thus, the following result 
is an extension of Stein's result for finite spaces.

\begin{theorem}\label{steintheo}
Suppose that $(E,\mathcal{E},\mu)$ is a probabilty space and that, one some other probability space $(\Omega,\F,\P)$, there is an exchangeable pair $(X,X')$ of 
$E$-valued random variables with marginal distribution $\mu$. Then, in order that for every $G\in L^2(\mu)$ with $E_\mu[G]=0$ there exist a measurable and antisymmetric function 
$S:E\times E\rightarrow\R$ such that $G=\E[S(X,X')|X=(\cdot)]$ it is sufficient that the operator $L$ defined by \eqref{defL} is ergodic and satisfies a Poincar\'{e} inequality. In this case, for a 
given $G\in\ker E_\mu$, one can choose $S$ as given in \eqref{Sform}.
\end{theorem}

\section{Concrete spaces with exchangeable pair structure}\label{spaces}

\subsection{Functionals of independent random variables}\label{independent}
In this subsection we assume that the probability space $(E,\mathcal{E},\mu)$ is a (finite) product space, i.e. there is an $n\in\N$ as well as probability spaces $(S_1,\mathcal{S}_1,\nu_1),\dotsc,(S_n,\mathcal{S}_n,\nu_n)$ such that 
\begin{equation*}
 (E,\mathcal{E},\mu)=\Biggl(\prod_{j=1}^n S_j,\bigotimes_{j=1}^n\mathcal{S}_j,\bigotimes_{j=1}^n \nu_j\Biggr)\,.
\end{equation*}
In the special situation that all the spaces $(S_j,\mathcal{S}_j,\nu_j)$, $1\leq j\leq n$, are instances of the same space $(S,\mathcal{S},\nu)$ we simply write 
\begin{equation}\label{iidcase}
 (E,\mathcal{E},\mu)=(S^n,\mathcal{S}^{\otimes n},\nu^n)\,.
\end{equation}
The following construction of an exchangeable pair $(X,X')$ of $E$-valued random variables with marginal distribution $\mu$ from \cite{St86} is by now fairly standard. Let $X_1,\dotsc,X_n,Y_1,\dotsc,Y_n$ and $\alpha$ be independent random variables 
on some probability space $(\Omega,\F,\P)$ such that, for $1\leq j\leq n$, $X_j$ and $Y_j$ both have distribution $\nu_j$ and such that $\alpha$ is uniformly distributed on $[n]:=\{1,\dotsc,n\}$. Then, we let $X:=(X_1,\dotsc,X_n)$ and define 
$X':=(X_1',\dotsc,X_n')$ coordinate-wise as follows: For $1\leq j\leq n$ we let 
\begin{equation*}
 X_j':=\begin{cases}
        Y_j\,,&\text{if }\alpha=j\,,\\
        X_j&\text{if }\alpha\not=j\,.
       \end{cases}
\end{equation*}
Then, it is easy to check that $(X,X')$ is indeed exchangeable. Note also that $X$ and $X'$ differ only in one (random) coordinate. In a Markov Chain Monte Carlo context, this procedure is also known as the \textit{Gibbs sampler} or as 
\textit{Glauber dynamics}.
The operator $L$ on $L^2(\mu)$ from \eqref{defL} that is associated with the pair $(X,X')$ is in this case given by 
\begin{align}\label{Lind}
 LF(x)&=\E\bigl[F(X')-F(X)\,\bigl|\, X=x\bigr]\notag\\
 &=\frac{1}{n}\sum_{j=1}^n\Bigl(\E\bigl[F(x_1,\dotsc,x_{j-1},Y_j,x_{j+1},\dotsc,x_n)\bigr]-F(x)\Bigr)\,,
\end{align}
where $F\in L^2(\mu)$ and $x=(x_1,\dotsc,x_n)\in E$. Moreover, for $F,G\in L^2(\mu)$, by Theorem \ref{abstheo} the corresponding carr\'e du champ operator $\Gamma$ is given by 
\begin{align}\label{Gammaind}
 \Gamma(F,G)(x)&=\frac12\E\bigl[\bigl(F(X')-F(X)\bigr)\bigl(G(X')-G(X)\bigr)\,\bigl|\,X=x\bigl]\notag\\
 &=\frac{1}{2n}\sum_{j=1}^n\E\Bigl[\Bigl(F(x_1,\dotsc,x_{j-1},Y_j,x_{j+1},\dotsc,x_n)-F(x)\Bigr)\notag\\
 &\hspace{3cm}\cdot\Bigl(G(x_1,\dotsc,x_{j-1},Y_j,x_{j+1},\dotsc,x_n)-G(x)\Bigr)\Bigr]\,.
\end{align}
We next make sure that Assumption \ref{specgap} is satisfied by $L$: 
Indeed, by the \textit{Efron-Stein inequality} (see e.g \cite[Theorem 3.1]{BLM}) we have that 
\begin{align*}
 E_\mu[\Gamma(F,F)]&=\frac{1}{2}\E\bigl[(F(X')-F(X)\bigr)^2\bigr]\notag\\
 &=\frac{1}{2n}\sum_{j=1}^n \E\Bigl[\bigl(F(X_1,\dotsc,X_{j-1},Y_j,X_{j+1},\dotsc,X_n)-F(X_1,\dotsc,X_n)\bigr)^2\Bigr]\notag\\
 &\geq \frac{1}{n}\Var\bigl(F(X)\bigr)\,.
 \end{align*}
Hence, \eqref{pcex} holds with the constant $c=n$.

In order to give a crucial alternative expression of $\Gamma(F,G)$, knowledge of the so-called \textit{Hoeffding decomposition} (see e.g. \cite{Hoeffding, vitale}) of functionals of $X$ is very useful.
 For $J\subseteq[n]$ let $\F_J:=\sigma(X_j,j\in J)$.
If $F:E\rightarrow\R$ is in $L^1(\mu)$ and, hence, $Y:=F(X_1,\dotsc,X_n)\in L^1(\P)$, then 
$Y$ has a $\P$-a.s. unique representation of the form 
\begin{equation}\label{HD}
 Y=\sum_{M\subseteq[n]} Y_M\,,
\end{equation}
where, for $M\subseteq[n]$, $Y_M\in L^1(\P)$ is a random variable such that 
\begin{enumerate}[(i)]
 \item $Y_M$ is $\F_M$-measurable and 
 \item for all $J\subseteq[n]$ one has $\E[Y_M\,|\, \F_J]=0$ unless $M\subseteq J$.
\end{enumerate}
Note that (ii) implies that $\E[Y_M]=0$ whenever $M\not=\emptyset$. Hence, from (i) it follows that $Y_\emptyset=\E[Y]$ $\P$-a.s. 
More generally, one has the following explicit expression for $Y_M$, $M\subseteq[n]$:
\begin{equation}\label{HDexp}
 Y_M=\sum_{L\subseteq M}(-1)^{|M|-|L|}\E\bigl[Y\,\bigl|\,\F_L\bigr]\,.
\end{equation}
From $\eqref{HDexp}$ it is straighforward to see that $Y\in L^p(\P)$ for some $p\in[1,\infty]$ implies that $Y_M\in L^p(\P)$ for all $M\subseteq[n]$. Moreover, if $Y\in L^2(\P)$, then (i) and (ii) together imply that 
$Y_M$ and $Y_N$ are orthogonal in $L^2(\P)$, whenever $N\not=M$. It will be useful to reorder the sum in \eqref{HD} according to the size of the subset $|M|$. For $p=0,1,\dotsc,n$ we thus define 
\begin{equation}\label{Yp}
 Y^{(p)}:=\sum_{\substack{M\subseteq[n]:\\ |M|=p}} Y_M=\sum_{M\in\D_p(n)}{Y_M}\,,
\end{equation}
where we write $\D_p(n)$ for the collection of $p$-subsets of the set $[n]=\{1,\dotsc,n\}$.
Then, for $1\leq p\leq n$, $Y^{(p)}$ is called a \textit{degenerate (non-symmetric) $U$-statistic of order $p$} (or \textit{$p$-homogeneous sum} \cite{deJo90}) based on $X=(X_1,\dotsc,X_n)$. 
The representation 
\begin{equation}\label{HD2}
 Y= \sum_{p=0}^n Y^{(p)}=\E[Y]+\sum_{p=1}^n Y^{(p)}
\end{equation}
of $Y$ as a sum of its expectation and a sum of degenerate non-symmetric $U$-statistics is also sometimes referred to as the Hoeffding decomposition of $Y$. 
For $p=0,1,\dotsc,n$, we define the subspace $\mathcal{H}_p(X)$ consisting of all random variables $U\in L^2(\P)$ of the form 
\begin{align*}
 U=\sum_{\substack{M\subseteq[n]:\\ |M|=p}} U_M\,,
\end{align*}
where the summands $U_M$ are supposed to satisfy the conditions (i) and (ii) above, i.e. $\mathcal{H}_p(X)$ equals the subspace of all square-integrable, degenerate (non-symmetric) $U$-statistics of order $p$ based on $X$. Then, the spaces $\mathcal{H}_p(X)$ are pairwise orthogonal. Moreover, one has 
\begin{equation*}
 L^2(\Omega,\sigma(X),\P)=\bigoplus_{p=0}^n \mathcal{H}_p(X)
\end{equation*}
and $Y^{(p)}$ is equal to the orthogonal projection of $Y$ on $\mathcal{H}_p(X)$. The space $\mathcal{H}_p(X)$, $p=0,\dotsc,n$, is known as the \textit{Hoeffding space of order $p$} based on $X=(X_1,\dotsc,X_n)$ 
in the literature. Note that, by the factorization lemma, each $U\in \mathcal{H}_p(X)$ has a representation of the form 
\begin{equation}\label{representative}
 U=u(X_1,\dotsc,X_n)\,,
\end{equation}
where $u\in L^2(\mu)$ is a $\mu$-a.e. unique \textit{representative function}. 
Thus, defining for $p=0,1,\dotsc,n$ the subspace $\mathcal{H}_p$ to be the collection of all functions $u$ such that $U=u(X_1,\dotsc,X_n)\in\mathcal{H}_p(X)$, one obtains the orthogonal decomposition
\begin{align}\label{orthoind}
 L^2(\mu)=\bigoplus_{p=0}^n \mathcal{H}_p\,.
\end{align}
It turns out that the $\mathcal{H}_p$ are precisely the eigenspaces of the operator $L$ from \eqref{Lind}.

\begin{prop}\label{eigenind}
 For $p=0,1,\dotsc,n$, one has that $\mathcal{H}_p=\ker(L+\frac{p}{n}\Id)$, i.e. $\mathcal{H}_p$ is the eigenspace of $L$ corresponding to the eigenvalue $-\frac{p}{n}$. 
 In particular, $L$ is diagonalizable.
\end{prop}

\begin{remark}\label{remhdeigen}
We would like to emphasize the fact that the canonical Hoeffding spaces associated to the product measure $\mu$ are precisely the eigenspaces of the operator $L$ given by \eqref{Lind}. 
This might be a good reason for considering $L$ and the associated $\Gamma$ an equally canonical tool for analyzing the structure of product spaces. 
\end{remark}

\begin{proof}[Proof of Proposition \ref{eigenind}]
 It is clear that $Lu=0$ for all $u\in \mathcal{H}_0$, since $u(X)$ is $\P$-a.s. constant. Now suppose that $1\leq p\leq n$ and let $u\in\mathcal{H}_p$. Then, $U:=u(X)\in\mathcal{H}_p(X)$ is a degenerate, non-symmetric $U$-statistic of order $p$. Thus, writing $U':=u(X')$ and using \cite[Lemma 2.3]{DP16} for the third equality, it follows that 
 \begin{align*}
  Lu(X)&=\E\bigl[u(X')-u(X)\,\bigl|\, X\bigr]=\E\bigl[U'-U\,\bigl|\, X\bigr]=-\frac{p}{n}U\\
  &=-\frac{p}{n}u(X)\quad\P\text{-a.s.}
 \end{align*}

This implies that $Lu=-\frac{p}{n}u$ in $L^2(\mu)$. Hence, we have shown that $\mathcal{H}_p\subseteq\ker\bigl(L+\frac{p}{n}\Id)$ holds for $p=0,1,\dotsc,n$. 
Conversely, suppose that $u\in\ker\bigl(L+\frac{p}{n}\Id)$ for some $0\leq p\leq n$. Then, by \eqref{orthoind} there are $u_j\in\mathcal{H}_j$, $j=0,\dotsc,n$, such that $u=\sum_{j=0}^n u_j$. Now, since $L$ is a self-adjoint operator on 
$L^2(\mu)$ we know that its eigenspaces are mutually orthogonal and, as we already know that $\mathcal{H}_j\subseteq\ker\bigl(L+\frac{j}{n}\Id)$ for $j=0,1,\dotsc,n$, we conclude that
\[0=\int_E u u_j\,d\mu=\int_E u_j^2\,d\mu\quad\text{for all }j\not=p\,.\]
Hence, $u=u_p\in\mathcal{H}_p$. Thus, we have $\mathcal{H}_p=\ker(L+\frac{p}{n}\Id)$ for all $p=0,1,\dotsc,n$.
\end{proof}

\begin{prop}\label{Prodind}
 Suppose that $1\leq p,q\leq n$ as well as that $F\in\mathcal{H}_p$ and $G\in\mathcal{H}_q$. Suppose that the Hoeffding decomposition of $U:=F(X)G(X)$ is given by 
 \begin{equation*}
  U=\sum_{\substack{M\subseteq[n]:\\ |M|\leq p+q}} U_M=\sum_{\substack{M\subseteq[n]:\\ |M|\leq p+q}}u_M(X)\,,
 \end{equation*}
for certain representative functions $u_M\in L^1(\mu)$, $M\subseteq[n]$, $|M|\leq p+q$.
Then, we have 
\begin{align*}
 \Gamma(F,G)&=\frac{1}{2n}\sum_{\substack{M\subseteq[n]:\\ |M|\leq p+q-1}}(p+q-|M|)u_M\,.
\end{align*}
Moreover, if $FG\in L^2(\mu)$, then the functions $u_M\in L^2(\mu)$, $M\subseteq[n]$, $|M|\leq p+q$ are orthogonal.
\end{prop}

\begin{proof}
Since $F(X)G(X)\in L^1(\P)$ its Hoeffding decomposition exists and then the existence of the functions $u_M\in L^1(\mu)$ such that $U_M=u_M(X)$ follows from the factorization lemma. Moreover, by Proposition \ref{eigenind} and the definition of $\Gamma$ we have 
\begin{align*}
2\Gamma(F,G)&=L(FG)-GLF-FLG=L(FG)+\frac{p}{n} FG+\frac{q}{n}FG\\
&=-\sum_{\substack{M\subseteq[n]:\\ |M|\leq p+q}} \frac{\abs{M}}{n}u_M+\frac{p+q}{n}\sum_{\substack{M\subseteq[n]:\\ |M|\leq p+q}}u_M\\
&=\frac{1}{n}\sum_{\substack{M\subseteq[n]:\\ |M|\leq p+q-1}}(p+q-|M|)u_M\,.
\end{align*}
If $FG\in L^2(\mu)$, then the Hoeffding components $U_M$ are mutually orthogonal in $L^2(\P)$. Hence, so are the functions $u_M$ in $ L^2(\mu)$ in this case.
\end{proof}

\begin{remark}
\begin{enumerate}[(a)]
\item Alternatively, we could have invoked \cite[Lemma 3.3]{DP16} to obtain the Hoeffding decomposition 
\begin{align*}
 n\E\bigl[\bigl(F(X')-F(X)\bigr)\bigl(G(X')-G(X)\bigr)\,\bigl|\, X\bigr]&=\sum_{\substack{M\subseteq[n]:\\ |M|\leq p+q-1}}(p+q-|M|) U_M\\
 &=\sum_{\substack{M\subseteq[n]:\\ |M|\leq p+q-1}}(p+q-|M|)u_M(X)
\end{align*}
and then apply Theorem \ref{abstheo}. Conversely, \cite[Lemma 3.3]{DP16} is a simple consequence of Proposition \ref{Prodind} and Theorem \ref{abstheo}, the overall argumentation being much shorter than the original proof given in \cite{DP16}. This observation is another piece of evidence that our abstract viewpoint pays off.
\item Note that since the random variable $U_M$ in the above proposition is $\F_M$-measurable, it does only depend on the random variables $X_j$, $j\in M$. Hence, there is even a measurable function $v_M: \prod_{j\in M} E_j\rightarrow \R$ such that $U_M=v_M(X_j,j\in M)$. Since we want all the functions $u_M$ to be defined on the same (bigger) space $E$, we did not make use of this fact explicitly in the statement of Proposition \ref{Prodind}. However, it is implicit in the orthogonality of the $u_M$.
\end{enumerate}
\end{remark}

We now present a general bound on the Wasserstein distance between a functional $F\in L^4(\mu)$ and the standard normal random variable $Z$. 
From now on until the end of this Subsection, in order to facilitate notation, we will work on the canonical probability space 
\[(\Omega,\F,\P):=(E,\mathcal{E},\mu)=\Biggl(\prod_{j=1}^n S_j,\bigotimes_{j=1}^n\mathcal{S}_j,\bigotimes_{j=1}^n \nu_j\Biggr)\]
and we denote by $X_i:\prod_{j=1}^n E_j\rightarrow E_i$ the canonical projection on $E_i$, $i=1,\dotsc,n$. Thus, $X=(X_1,\dotsc,X_n)$ is just the identity map on $\Omega=E$. Since we are only interested in distributional properties and since, with the operators $L$ and $\Gamma$ in place, we no longer need the auxiliary randomization inherent in the coupling $(X,X')$, this causes no troubles. 
 We will make explicit use of the Hoeffding decomposition of $F=F(X)$. Suppose that \eqref{stand} holds for $F$. Then, by \eqref{orthoind} we can write 
\[F=\sum_{p=1}^n F_p\,,\]
where we know from Proposition \ref{eigenind} that $F_p\in\mathcal{H}_p$ is an eigenfunction of $L$ corresponding to the eigenvalue $-p/n$, $1\leq p\leq n$. 
Moreover, each $F_p=F_p(X)$ is a (non-symmetric) degenerate $U$-statistic of order $p$. Hence, there are $\F_J$-measurable random variables $W_J(p)$, $J\in\D_p(n)$, such that 
\[F_p=\sum_{J\in\D_p(n)} W_J(p)\]
is the Hoeffding decomposition of $F_p$ and such that also the variables $W_J(p)\in\mathcal{H}_p$.
Due to the orthogonality of different Hoeffding spaces and by our assumption we further know that 
\[1=\Var(F)=\sum_{p=1}^n \E[F_p^2]\,.\] 
 We further define (following \cite{deJo90, DP16}) the quantities 
\[\rho_p^2:=\rho_p^2(n):=\max_{1\leq i\leq n}\sum_{\substack{J\in\D_p(n):\\ i\in J}}\E[W_J(p)^2]\,,\]
which measure the maximal influence that a single variable $X_i$ can have on the total variance of $F_p$.
Since 
\begin{align*}
n\rho_p^2(n)&\geq \sum_{i=1}^n\sum_{\substack{J\in\D_p(n):\\ i\in J}}\E[W_J(p)^2]=p \sum_{J\in\D_p(n)}\E[W_J(p)^2]=p\E[F_p^2]\,,
\end{align*}
we have 
\[\rho_p^2(n)\geq\frac{p\E[F_p^2]}{n}\]
for all $1\leq p\leq n$.

In order to apply Theorem \ref{genbound2} we need expressions for the quantities $\Gamma(F_p,F_q)$, $1\leq p,q\leq n$. By Proposition \ref{Prodind} this reduces to finding the Hoeffding decomposition 
\begin{equation}\label{HDFpFq}
F_p F_q=\sum_{\substack{M\subseteq[n]:\\ |M|\leq p+q}}U_{M}(p,q)
\end{equation}
of the products $F_p F_q$. 
With this decomposition at hand we have that 
\begin{equation*}
\Gamma(F_p,F_q)=\frac{1}{2n}\sum_{\substack{M\subseteq[n]:\\ |M|\leq p+q-1}}(p+q-|M|)U_{M}(p,q).
\end{equation*}
Then, from Theorem \ref{genbound2} we obtain the following result, whose proof is given in Section \ref{proofs}.

\begin{theorem}\label{genboundind}
Under the above assumptions, we have the following general bounds:
\begin{align*}
&d_\W(F,Z)\leq \sqrt{\frac{2}{\pi}}\Biggl(\sum_{l=1}^{2n-1}\sum_{M\in\D_l(n)}\Var\biggl(\sum_{\substack{1\leq p,q\leq n:\\ p+q\geq l+1}}\frac{p+q-l}{2p}U_M(p,q)\biggr)\Biggr)^{1/2}            \\
&\; +\sqrt{2}\sum_{p=1}^n\frac{\sqrt{\E[F_p^2]}}{\sqrt{p}}\cdot\biggl(\sum_{q=1}^n q^{1/4}\Bigl(\kappa_q\E[F_q^2]\rho_q^2 
+\sum_{\substack{M\subseteq[n]:\\1\leq |M|\leq 2q-1}}\frac{2q-|M|}{q}\Var\bigl(U_M(q,q)\bigr)\Bigr)^{1/4}\biggr)^2\\
&\leq 
\sqrt{\frac{2}{\pi}} \sum_{ p,q=1}^n \frac{p+q-1}{2p}\biggl(\sum_{\substack{M\subseteq[n]:\\ 1\leq |M|\leq p+q-1}}\Var\bigl(U_M(p,q\bigr)\biggr)^{1/2}    \\
&\; +\sqrt{2}\sum_{p=1}^n\frac{\sqrt{\E[F_p^2]}}{\sqrt{p}}\cdot\biggl(\sum_{q=1}^n q^{1/4}\Bigl(\kappa_q\E[F_q^2]\rho_q^2 
+\frac{2q-1}{q}\sum_{\substack{M\subseteq[n]:\\1\leq |M|\leq 2q-1}}\Var\bigl(U_M(q,q)\bigr)\Bigr)^{1/4}\biggr)^2\,.
\end{align*}
Here, the $\kappa_q$, $1\leq q\leq n$, are finite, positive constants that only depend on $q$.
\end{theorem}

\begin{remark}\label{genboundindrem}
\begin{enumerate}[(a)]
\item The above bounds involve similar quantities as the multivarate bounds given in Lemma 4.1 of \cite{DP19}. However, in contrast to those bounds, which are for $C^3$ or at least $C^2$ test functions, they bound the Wasserstein distance between $F$ and $Z$.
\item Note that the bounds in Theorem \ref{genboundind} still involve the implicit constants $\kappa_q$. Since we do not give any explicit formula or general bound on these constants, 
at this level of generality the bounds seem only to be useful, if there is a fixed $m\in\N$ (independent of $n$) such that $F_p=0$ $\P$-a.s. for all $p>m$. The combinatorial constants $\kappa_q$ stem from \cite[Proposition 2.9, Lemma 2.10]{DP16} and indeed grow faster than exponentially in general. However, we will see below via a new observation that 
in the case of symmetric functionals of $n$ i.i.d. random variables, we can make sure that $\kappa_q\leq 2q$. Hence, in such a situation we can allow 
$m:=\max\{1\leq p\leq n\,:\, \Var(F_p)>0\}$ to be an increasing function of $n$ and still hope to obtain vanishing bounds.
\item By an application of Theorem \ref{genboundgamma}, a similar bound for the centered Gamma approximation of such a functional $F$ can easily be derived.
\item In \cite{Chatind} a different method of deriving Wasserstein bounds on the normal approximation of a functional of independent random variables has been developed. This approach has further been 
extended to also yield Kolmogorov bounds in \cite{LRP3}. As opposed to Theorem \ref{genboundind}, whose application necessitates dealing with the Hoeffding decomposition of products, the approach 
in \cite{Chatind} relies on the control of a certain random variable $T$, which is defined in terms of products of first order differences, averaged over subsets, and whose construction also involves an 
independent copy of $X$. Although both approaches can in principle be applied to arbitrary functionals of independent random variables, it seems that they are suitable for rather different sorts of problems. Indeed, whereas the approach of \cite{Chatind} can be carried out, if local perturbations to the functional $F(X_1,\dotsc,X_n)$ resulting from replacing $X_j$ with an independent copy $Y_j$ can be controlled, our method can be effectively used, whenever Hoeffding decompositions (of products) can be computed. In the important situation of symmetric functionals of i.i.d. random variables dealt with in the next subsection, this is in particular possible by means of product formulae for degenerate $U$-statistics, which are the building blocks of symmetric functionals.
\end{enumerate}
\end{remark}

\subsubsection{Symmetric functionals of i.i.d. random variables}\label{symfunct}
From now on we will deal with $\mathcal{E}$-measurable, symmetric functionals $F=F(X_1,\dotsc,X_n)$ of i.i.d. random variables $X_1,\dotsc,X_n$ with a common distribution $\nu$ on a measurable space $(S,\mathcal{S})$. In particular, we are dealing with the situation \eqref{iidcase} here. By symmetric we mean that $F(x_{\sigma(1)},\dotsc,x_{\sigma(n)})=F(x_1,\dotsc,x_n)$ holds for all $x=(x_1,\dotsc,x_n)\in E=S^n$ and all 
$\sigma$ in the symmetric group $\mathbb{S}_n$ acting on $[n]$. From \eqref{orthoind} we know that there are $F_j\in\mathcal{H}_j$, $j=0,1,\dotsc,n$, such that 
\begin{equation}\label{symstat}
F=\sum_{j=0}^n F_j\,.
\end{equation}
Moreover, it follows from \eqref{HDexp} that 
$F_0\equiv \E[F]$ and, for $p=1,\dotsc,n$, we have
\begin{equation*}
 F_p=F_p(X_1,\dotsc,X_n)=\sum_{1\leq i_1<\ldots<i_p\leq n}\phi_p(X_{i_1},\dotsc,X_{i_p})\,,
\end{equation*}
where
\begin{equation}\label{defphip}
 \phi_p(y_1,\dotsc,y_p)=\sum_{k=0}^p (-1)^{p-k}\sum_{1\leq j_1<\ldots< j_k\leq p}h_k(y_{j_1},\dotsc,y_{j_k})
\end{equation}
and 
\begin{equation*}
 h_k(t_1,\dotsc,t_k):=\int_{S^{n-k}} F(t_1,\dotsc,t_k,s_1,\dotsc,s_{n-k})\,d\nu^{n-k}(s_1,\dotsc,s_{n-k})
\end{equation*}
for $k=0,\dotsc,n$ (see e.g. \cite{ser-book, Major, DP19} for details). Moreover, the functions $\phi_p$, $p=1,\dotsc,n$, are \textit{canonical} or \textit{completely degenerate} of order $p$ with respect to $\nu$. Recall that this means that
\begin{equation*}\label{canonical}
\int_{S} \phi_p(x_1,\dots,x_{p-1},y)\,d\nu(y)=0\quad\text{for }\nu^{p-1}\text{-a.a. } (x_1,\dotsc,y_{p-1})\in S^{p-1}\,. 
\end{equation*}
One customarily refers to $F_p$ as a \textit{degenerate, symmetric $U$-statistic of order $p$} based on $X_1,\dotsc,X_n$ and to $\phi_p$ as its \textit{kernel}. 
Hence, \eqref{symstat} says that a symmetric function $F$ can be written as an orthogonal sum of its expectation $\E[F]$ and a sum of degenerate, symmetric $U$-statistics of respective orders $p=1,\dotsc,n$. By the orthogonality of Hoeffding components it further holds that 
\[\Var(F)=\sum_{p=1}^n\Var(F_p)\,,\]
where 
\begin{equation}\label{varfpsym}
\Var(F_p)=\binom{n}{p}\norm{\phi_p}_{L^2(\nu^p)}^2\,,\quad 1\leq p\leq n.
\end{equation}
In order to apply Theorem \ref{genboundind} we need to make the Hoeffding decomposition \eqref{HDFpFq} of $F_pF_q$ more explicit. To this end, the \textit{product formula for degenerate, symmetric $U$-statistics} from \cite{DP19} will be very useful. Since it involves the notion of contraction kernels, as in the recent papers \cite{DP19, DKP}, our bounds on the normal approximation of $F$ will be expressed in terms of norms of contraction kernels. Hence, we briefly recall their definition.

Given integers $p,q\geq1$, $0\leq l\leq r\leq p\wedge q$ and two symmetric kernels $\psi\in L^2(\nu^{ p})$ and $\phi\in L^2(\nu^{ q})$, we define the \textit{contraction kernel} $\psi\star_r^l \phi$ on $S^{p+q-r-l}$ by {the relation}
\begin{align}
 &(\psi\star_r^l \phi)(y_1,\dotsc,y_{r-l},t_1,\dotsc,t_{p-r},s_1,\dotsc,s_{q-r})\notag\\
&:=\int_{S^l}\Bigl(\psi\bigl(x_1,\dotsc,x_l, y_1,\dotsc,y_{r-l},t_1,\dotsc,t_{p-r}\bigr)\notag\\
&\hspace{3cm}\cdot\phi\bigl(x_1,\dotsc,x_l, y_1,\dotsc,y_{r-l},s_1,\dotsc,s_{q-r}\bigr)\Bigr)d\nu^{ l}(x_1,\dotsc,x_l)\label{defcontr1}\\
&=\E\Bigl[\psi\bigl(X_1,\dotsc,X_l, y_1,\dotsc,y_{r-l},t_1,\dotsc,t_{p-r}\bigr)\notag\\
&\hspace{3cm}\cdot\phi\bigl(X_1,\dotsc,X_l, y_1,\dotsc,y_{r-l},s_1,\dotsc,s_{q-r}\bigr)\Bigr]\notag\,.
\end{align}
By \cite[Lemma 2.4 (i)]{DP19} $\psi\star_r^l \phi$ is a $\nu^{ p+q-r-l}$-a.e. well-defined function on $S^{p+q-r-l}$ and, as usual, we set it equal to zero on the remaining set of $\nu^{ p+q-r-l}$-measure $0$.

\medskip

If $l=0$, then \eqref{defcontr1} has to be understood as follows: 
\begin{align*}
 &(\psi\star_r^0 \phi)(y_1,\dotsc,y_{r},t_1,\dotsc,t_{p-r},s_1,\dotsc,s_{q-r})\notag\\
 &=\psi(y_1,\dotsc,y_{r},t_1,\dotsc,t_{p-r})\phi( y_1,\dotsc,y_{r},s_1,\dotsc,s_{q-r})\,.
\end{align*}
In particular, if $l=r=0$, then $\psi\star_r^l \phi$ reduces to the \textit{tensor product}
\[\psi\otimes \phi:S^{p+q}\rightarrow\R\]
 of $\psi$ and $\phi$, given by 
\begin{align*}
 (\psi \otimes\phi)(x_1,\dotsc,x_{p+q})&=\psi(x_1,\dotsc,x_p)\cdot\phi(x_{p+1},\dotsc,x_{p+q})\,.
\end{align*}
We observe that $\psi\star_p^0\psi=\psi^2$ is square-integrable if and only if $\psi\in L^4(\nu^{ p})$. As a consequence, $\psi\star_r^l\phi$ might not be in $L^2(\nu^{ p+q-r-l})$ even though $\psi\in L^2(\nu^{ p})$ and $\phi\in L^2(\nu^{ q})$. Moreover, if $l=r=p$, then $\psi\star_p^p\psi=\|\psi\|_{L^2(\nu^{ p})}^2$ is constant.
\medskip

For positive integers $p,q,t,r$ such that $1\leq t\leq 2(p\wedge q)\wedge (p+q-1)$, $t\leq 2r\leq 2(t\wedge p\wedge q) $ we define the constants
\begin{equation}\label{defalpha}
\alpha(p,q,t,r):=\frac{\sqrt{(p+q-t)!}}{(t-r)!(p-r)!(q-r)!(2r-t)!}\,.
\end{equation}
Moreover, for a function $f\in L^2(\nu^q)$ for some $q\in\N$ we will constantly write $\norm{f}_2$ for $\norm{f}_{L^2(\nu^q)}$ in order to facilitate notation.

\begin{theorem}\label{genboundsymstat}
For $F$ as in \eqref{symstat} satisfing \eqref{stand} and with the kernels $\phi_p$ defined in \eqref{defphip} we have the general bound:
\begin{align*}
&d_\W(F,Z)\leq \sqrt{\frac{2}{\pi}} \sum_{ p,q=1}^n \frac{p+q-1}{2p}
\sum_{t=1}^{2(p\wedge q)\wedge(p+q-1)}\sum_{r=\ceil{\frac{t}{2}}}^{t\wedge p\wedge q}\alpha(p,q,t,r) \norm{\phi_p\star_r^{t-r}\phi_q}_{2}\,n^{\frac{p+q+t}{2}-r}\\
&\;+ \sqrt{2}\sum_{p=1}^n\frac{\sqrt{\E[F_p^2]}}{\sqrt{p}}\cdot\Biggl(\sum_{q=1}^n q^{1/4}\biggl[ \frac{2^{1/4}\sqrt{q}}{n^{1/4}}\sqrt{\E[F_q^2] }\\
&\hspace{3cm}+\Bigl(\frac{2q-1}{q}\Bigr)^{1/4}  \sum_{t=1}^{2q-1}\sum_{r=\ceil{\frac{t}{2}}}^{t\wedge q}\sqrt{\alpha(q,q,t,r)} \norm{\phi_q\star_r^{t-r}\phi_q}_{2}^{1/2} n^{\frac{2q+t-2r}{4}} \biggr]\Biggr)^2\,.
\end{align*}
\end{theorem}

\begin{remark}\label{genboundsymstatrem}
\begin{enumerate}[(a)]
\item The above bound gives an estimate for the Wasserstein distance between the distribution of a completely general symmetric functional of $n$ i.i.d. random variables and the standard normal distribution. 
Note that, in view of \eqref{varfpsym} and \eqref{defalpha}, all quantities and constants involved in this bound are at least in principle completely explicit.
\item However, the expressions \eqref{defphip} defining the kernels $\phi_p$ are quite complicated. Hence, starting from $F$ it might be unfeasible to compute these kernels and a fortiori to 
bound the corresponding contraction norms appearing in the bound. We will demonstrate below for the more special class of symmetric $U$-statistics (see Theorem \ref{symustat2}), how one can deal with this issue. Basically, the $\phi_p$ may be replaced with kernels $h_k$ in these norms. This important computational trick, which has already been applied in the recent papers \cite{DP19,DKP}, would also be applicable in the present, more general situation of Theorem \ref{genboundsymstat}. 
\end{enumerate}
\end{remark}

\subsubsection{Symmetric $U$-statistics}\label{symustat}
We now specialize the situation to the setting of symmetric $U$-statistics. Let $m\geq 1$ and suppose that $\psi=\psi(n)\in L^m(\nu^m)$ is a symmetric (not necessarily degenerate) kernel which might depend on the sample size $n$. We denote by 
\begin{equation}\label{ustat}
G:=J_m(\psi):=J_{m,X}(\psi):=\sum_{1\leq i_1<\ldots<i_m\leq n}\psi(X_{i_1},\dotsc,X_{i_m})=\sum_{J\in\D_m(n)}\psi(X_j,j\in J)
\end{equation} 
the \textit{symmetric (not necessarily degenerate) $U$-statistic of order $m$} based on $X=(X_1,\dotsc,X_n)$ and with kernel $\psi$. If $\psi$ is in fact canonical, then $J_m(\psi)$ is called a \textit{degenerate, symmetric  $U$-statistic of order $m$}. 
We will assume that 
\[\sigma^2:=\sigma_n^2:=\Var(G)>0\,.\] 
Moreover, we let 
\[F:=\frac{G-\E[G]}{\sigma}\]
denote the normalized version of $G$. Then, $F$ is a symmetric functional of the i.i.d. variables $X_1,\dotsc,X_n$ and it is well-known (see e.g. \cite{ser-book, DP19}), or otherwise follows from \eqref{defphip} via a combinatorial exercise, that the kernels in \eqref{defphip} are in this case given by 
\[\phi_p(x_1,\dotsc,x_p)=\binom{n-p}{m-p}\frac{\psi_p(x_1,\dotsc,x_p)}{\sigma}\]
for $p=1,\dotsc,m$ and $\phi_p\equiv0$ for $p>m+1$. Here, the kernels $\psi_p$, $1\leq p\leq m$, are defined by 
\begin{align}\label{defpsip}
\psi_p(x_1,\dotsc,x_p)&=\sum_{k=0}^p(-1)^{p-k}\sum_{1\leq i_1<\dotsc<i_k\leq p} g_k(x_{i_1},\dotsc,x_{i_k})\,,
 \end{align}
where 
\begin{equation}\label{gk}
 g_k(y_1,\dotsc,y_k):=\E\bigl[\psi(y_1,\dotsc,y_k,X_1,\dotsc,X_{m-k})\bigr]\,.
\end{equation}
Hence, one has that $F=\sum_{p=1}^m F_p$, where the orthogonal Hoeffding components $F_p$, $1\leq p\leq m$, of $F$ are given by 
\[F_p=\frac{\binom{n-p}{m-p}}{\sigma}J_p(\psi_p)\,.\]
The following result is an immediate consequence of Theorem \ref{genboundsymstat}. 
Since 
\begin{align}\label{vardecustat}
1&=\Var(F)=\sum_{p=1}^m\Var(F_p)=\frac{\binom{n-p}{m-p}^2\binom{n}{p}}{\sigma^2}\norm{\psi_p}_2^2\,,
\end{align}
one always has the bounds 
\begin{align}
\E[F_p^2]&\leq 1\,,\label{vb1}\\
\E[F_p^2]&\leq \frac{n^{2m-p}\norm{\psi_p}_2^2}{\sigma^2(m-p)!(m-p)! p!}\label{vb2}\,,
\end{align}
for $p=1,\dotsc,m$.

\begin{cor}\label{symustat1}
With the above definitions and notation, we have the bound:
\begin{align*}
&d_\W(F,Z)\\
&\leq\sqrt{\frac{2}{\pi}} \sum_{ p,q=1}^m \frac{p+q-1}{2p}
\sum_{t=1}^{2(p\wedge q)\wedge(p+q-1)}\sum_{r=\ceil{\frac{t}{2}}}^{t\wedge p\wedge q}\alpha(p,q,t,r) \frac{\norm{\psi_p\star_r^{t-r}\psi_q}_{2}}{\sigma^2(m-p)!(m-q)!}
n^{2m-\frac{p+q+2r-t}{2}}\\
&\;+ \sqrt{2}\sum_{p=1}^m\frac{\sqrt{\E[F_p^2]}}{\sqrt{p}}\cdot\Biggl(\sum_{q=1}^m q^{1/4}\biggl[ \frac{2^{1/4}\sqrt{q}}{n^{1/4}}\sqrt{\E[F_q^2] }\\
&\hspace{3cm}+\Bigl(\frac{2q-1}{q}\Bigr)^{1/4}  \sum_{t=1}^{2q-1}\sum_{r=\ceil{\frac{t}{2}}}^{t\wedge q}\sqrt{\alpha(q,q,t,r)} \frac{\norm{\psi_q\star_r^{t-r}\psi_q}_{2}^{1/2}}{\sigma}
 n^{m-\frac{2q+2r-t}{4}} \biggr]\Biggr)^2\\
 &\leq\sqrt{\frac{2}{\pi}} \sum_{ p,q=1}^m \frac{p+q-1}{2p}
\sum_{t=1}^{2(p\wedge q)\wedge(p+q-1)}\sum_{r=\ceil{\frac{t}{2}}}^{t\wedge p\wedge q}\alpha(p,q,t,r) \frac{\norm{\psi_p\star_r^{t-r}\psi_q}_{2}}{\sigma^2(m-p)!(m-q)!}
n^{2m-\frac{p+q+2r-t}{2}}\\
&\;+ \sqrt{2}\sum_{p=1}^m\frac{n^{m-p/2}\norm{\psi_p}_2 }{\sqrt{p}\sigma(m-p)!\sqrt{p!}}\cdot\Biggl(\sum_{q=1}^m q^{1/4}\biggl[ 2^{1/4}\frac{n^{m-q/2-1/4}\norm{\psi_q}_2 }{\sigma(m-q)!\sqrt{(q-1)!}}\\
&\hspace{3cm}+\Bigl(\frac{2q-1}{q}\Bigr)^{1/4}  \sum_{t=1}^{2q-1}\sum_{r=\ceil{\frac{t}{2}}}^{t\wedge q}\sqrt{\alpha(q,q,t,r)} \frac{\norm{\psi_q\star_r^{t-r}\psi_q}_{2}^{1/2}}{\sigma}
 n^{m-\frac{2q+2r-t}{4}} \biggr]\Biggr)^2\,.
\end{align*}
\end{cor}

As already mentioned in Remark \ref{genboundsymstatrem}, the expressions for the degenerate kernels $\psi_p$ are quite complicated and, hence, bounding the contraction norms in Corollary \ref{symustat1} 
may become intractable. In contrast, the functions $g_k$ from \eqref{gk} are usually much simpler to handle. The next theorem is a version of Corollary \ref{symustat1} in which the kernels $\psi_p$ are systematically replaced with the $g_k$. In order to give the statement, me need to introduce some more notation that has been introduced in the recent paper \cite{DKP}.

For integers $r, p, q \geq 1$ and $ l\geq 0$ such that $0\leq l\leq r\leq p\wedge q$, we let 
$Q(p,q,r,l)$ be the set of quadruples $(j,k,a,b)$ of nonnegative integers such that the following hold:
\begin{enumerate}[(1)]
\item $j\leq p$ and $k\leq q$.
\item $b\leq a\leq r$.
\item $b\leq l$.
\item $a-b\leq r-l$.
\item $j+k-a-b\leq p+q-r-l\leq p+q-1$
\item $a\leq j\wedge k$.
\end{enumerate}

Moreover, for $1\leq k\leq m$ we define 
\[\hat{g}_k:=g_k-\E[\psi(X_1,\dotsc,X_m)]\]
such that, in particular,
\[\norm{\hat{g}_k}_2^2=\Var\bigl(g_k(X_1,\dotsc,X_k)\bigr)\,.\]

\begin{theorem}\label{symustat2}
Under the same assumptions as in Corollary \ref{symustat1} it holds that:
\begin{align*}
&d_\W(F,Z)\leq\sqrt{\frac{2}{\pi}} \sum_{ p,q=1}^m \frac{p+q-1}{2p}
\sum_{t=1}^{2(p\wedge q)\wedge(p+q-1)}\sum_{r=\ceil{\frac{t}{2}}}^{t\wedge p\wedge q}\alpha(p,q,t,r) K(p,q,r,t-r)\notag\\
&\hspace{2cm}\cdot\max_{(j,k,a,b)\in Q(p,q,r,t-r)}\frac{\|g_j\star_a^b g_k\|_{2}}{\sigma^2(m-p)!(m-q)!}
n^{2m-\frac{p+q+2r-t}{2}}\\
&\;+ \sum_{p=1}^m\frac{\sqrt{2\E[F_p^2]}}{\sqrt{p}}\cdot\Biggl(\sum_{q=1}^m q^{1/4}\biggl[ \frac{2^{1/4}\sqrt{q}}{n^{1/4}}\sqrt{\E[F_q^2] }
+\Bigl(\frac{2q-1}{q}\Bigr)^{1/4}  \sum_{t=1}^{2q-1}\sum_{r=\ceil{\frac{t}{2}}}^{t\wedge q}\sqrt{\alpha(q,q,t,r)}\\
&\hspace{2cm}\cdot\sqrt{K(p,q,r,t-r)  } \max_{(j,m,a,b)\in Q(p,q,r,t-r)}
\frac{\|g_j\star_a^b g_m\|_{2}^{1/2}}{\sigma}
 n^{m-\frac{2q+2r-t}{4}} \biggr]\Biggr)^2\,.
\end{align*}
Moreover, the right hand side can be further bounded to give
\begin{align*}
 &d_\W(F,Z)\leq\sqrt{\frac{2}{\pi}} \sum_{ p,q=1}^m \frac{p+q-1}{2p}
\sum_{t=1}^{2(p\wedge q)\wedge(p+q-1)}\sum_{r=\ceil{\frac{t}{2}}}^{t\wedge p\wedge q}\alpha(p,q,t,r) K(p,q,r,t-r)\notag\\
&\hspace{3cm}\cdot\max_{(j,k,a,b)\in Q(p,q,r,t-r)}\frac{\|g_j\star_a^b g_k\|_{2}}{\sigma^2(m-p)!(m-q)!}
n^{2m-\frac{p+q+2r-t}{2}}\\
&\;+ \sqrt{2}\sum_{p=1}^m\frac{n^{m-p/2}\norm{\hat{g}_p}_2}{\sqrt{p}\sigma(m-p)!\sqrt{p!}}\cdot\Biggl(\sum_{q=1}^m q^{1/4}\biggl[ 2^{1/4}\frac{n^{m-q/2-1/4}
\norm{\hat{g}_q}_2}{\sigma(m-q)!\sqrt{(q-1)!}}\\
&\hspace{3cm}+\Bigl(\frac{2q-1}{q}\Bigr)^{1/4}  \sum_{t=1}^{2q-1}\sum_{r=\ceil{\frac{t}{2}}}^{t\wedge q}\sqrt{\alpha(q,q,t,r)}\cdot\sqrt{K(p,q,r,t-r)  } \\
&\hspace{4cm}\cdot\max_{(j,m,a,b)\in Q(p,q,r,t-r)}
\frac{\|g_j\star_a^b g_m\|_{2}^{1/2}}{\sigma}
 n^{m-\frac{2q+2r-t}{4}} \biggr]\Biggr)^2\,.
\end{align*}
Here, the constants $K(p,q,r,t-r)\in(0,\infty)$ are defined in Lemma \ref{genulemma} below.
\end{theorem}

\begin{remark}\label{symustatrem}
\begin{enumerate}[(a)]
\item In particular, if $m$ does not depend on $n$ and $\psi=\psi(n)\in L^2(\nu^m)$ is a sequence of symmetric kernels of order $m$, then Theorem \ref{symustat2} implies that 
$F_n=F$ converges in distribution to the standard normal random variable $Z$, if for all $1\leq p,q\leq m$ and all pairs $(t,r)$ and quadruples $(j,k,a,b)$ of integers such that 
$1\leq t\leq 2(p\wedge q)\wedge (p+q-1)$, $t\leq 2r\leq 2(t\wedge p\wedge q)$ and $(j,k,a,b)\in Q(p,q,r,t-r)$
\[\lim_{n\to\infty}\frac{n^{2m-\frac{p+q+2r-t}{2}}}{\sigma_n^2}\|g_j\star_a^b g_k\|_{2}=0\]
and one also has a bound on the rate of convergence in the Wasserstein distance.
In \cite[Theorem 3.1]{DKP} it was shown that (under an additional minor variance condition) one can also prove functional convergence of the whole corresponding \textit{$U$-process} to a linear combination of time-changed Brownian motions, if the slightly stronger condition that 
\[\frac{n^{2m-\frac{p+q+2r-t}{2}}}{\sigma_n^2}\|g_j\star_a^b g_k\|_{2}=O(n^{-\epsilon})\]
for some $\epsilon>0$ is satisfied, for the same collection of indices as above.
\item Of course, similar conditions for convergence may be phrased in terms of the quantities appearing in the bounds given in Corollary \ref{symustat1}.
\item If the order $m=m_n$ of the $U$-statistic $G$ is bounded, then it typically suffices to apply the respective first bounds in Theorem \ref{symustat2} and Corollary \ref{symustat1} and to use
\eqref{vb1}. However, if $m_n$ is unbounded, then one might profit from the more careful second bounds.
\item The constants $K(p,q,r,t-r)$ appearing in the above Theorem are not given explicitly. However, by analyzing the proof of \cite[Lemma 5.7]{DKP} it would be easily possible to derive upper bounds on them as well. In this way, it is possible to use Theorem \ref{symustat2} also in situations where $m=m_n$ is unbounded.
\item In the paper \cite{RubVi80}, the notion of a \textit{symmetric statistic of order $m$} has been introduced. In our terminology this is just a finite sum of $U$-statistics of orders 
at most $m$. Using the methods from the present paper and starting again from Theorem \ref{genboundsymstat}, it would be straightforward to obtain similar bounds on the Wasserstein distance to normality for this bigger class as well. This is of some importance to applications because in concrete situations, the symmetric statistic at hand is sometimes given as a sum of non-degenerate $U$-statistics and, hence, it is desirable to have bounds on the normal approximation that are phrased in terms of their (non-canonical) kernels.
We leave the details to the interested reader.
\end{enumerate} 
\end{remark}

\subsection{Finite population statistics}\label{srs}
In this subsection we fix two integers $1\leq n<N$ as well as a set $\mathcal{X}=\{a_1,\dotsc,a_N\}$ of cardinality $N$. We then let $E$ be the collection of all $\binom{N}{n}$ $n$-subsets of $\mathcal{X}$, equipped with its power set 
$\mathcal{E}=\Pot(E)$ and the uniform distribution $\mu$ on $(E,\mathcal{E})$. A functional 
\[F:E\rightarrow\R\]
is customarily called a \textit{finite population statistic}. Note that such a finite population statistic may be identified with a symmetric function $F:\mathcal{X}_{\not=}^n\rightarrow\R$, where 
\[\mathcal{X}_{\not=}^n=\{(x_1,\dotsc,x_n)\in\mathcal{X}^n\,:\, x_i\not=x_j\text{ for all }i\not=j\}.\]
This will be tacitly done in what follows. Elements $(X_1,\dotsc,X_n)$ of $\mathcal{X}_{\not=}^n$ that are chosen uniformly at random are customarily called a \textit{simple random sample} of size $n$ drawn from the population $\mathcal{X}$.\\

 We may construct an exchangeable pair $(X,X')$ with marginal distribution $\mu$ as follows: On a suitable probability space $(\Omega,\F,\P)$, construct a 
simple random sample $X_1,\dotsc,X_{n+1}$ of size $n+1$ from $\mathcal{X}$ and let $I$ be independent of $S$ and uniformly distributed on $[n]$. Then, for $1\leq j\leq n$, let $X_j':=X_j$ if 
$j\not=I$ and let $X_I':=X_{n+1}$. Finally, set $X:=\{X_1,\dotsc,X_n\}$ and 
$X':=\{X_1',\dotsc,X_n'\}=(X\setminus\{X_I\})\cup \{X_{n+1}\}$. Then, $(X,X')$ is exchangeable and the operator $L$ from \eqref{defL} on $L^2(\mu)=\R^{E}$ is given by 
\begin{align}\label{Lpop}
LF(A)&=\frac{1}{n(N-n)}\sum_{a\in A}\sum_{b\in\mathcal{X}\setminus A}\Bigl(F\bigl((A\setminus\{a\})\cup \{b\}\bigr)-F(A)\Bigr)\,.
\end{align}
Recall that $-L$ is a positive, self-adjoint operator on the $\binom{N}{n}$-dimensional Hilbert space $L^2(\mu)$ and, hence, there is a basis of $L^2(\mu)$ consisting of eigenvectors of $-L$ that correspond 
to the non-negative eigenvalues $0=\lambda_0<\lambda_1<\ldots<\lambda_r$. We will show next that, as in the independent case considered in Subsection \ref{independent}, the Hoeffding decomposition of such functionals $F$ is key to understanding the diagonalization of $L$. In particular, we will see that $r=\min\{n,N-n\}$ and identify the eigenvalues $\lambda_j$ and the dimensions of the corresponding eigenspaces.\\

The Hoeffding decomposition of symmetric statistics of a simple random sample has been investigated and applied in several papers like e.g. \cite{ZhaoCh, BlozGo2001,BlozGo2002}. We will rely here mostly on the results and presentation from the paper \cite{BlozGo2001}. For $n,N$ as above define $n_*:=\min\{n,N-n\}\leq N/2$. Then, according to \cite[Section 2]{BlozGo2001}, the functional $F$ from above has the follwing (unique) \textit{Hoeffding decomposition} 
\begin{align}\label{HDsrs}
F(X_1,\dotsc,X_n)&=E_\mu[F]+\sum_{k=1}^{n_*} \sum_{1\leq i_1<\ldots<i_k\leq n} g_k(X_{i_1},\dotsc,X_{i_k})\notag\\
&=E_\mu[F]+\sum_{k=1}^{n_*} \sum_{J\in\D_k(n)} g_k(X_i,i\in J)\notag\\
&=:E_\mu[F]+\sum_{k=1}^{n_*} U_k\,,
\end{align}
where 
\[E_\mu[F]=\E[F(X_1,\dotsc,X_n)]=\binom{N}{n}^{-1}\sum_{1\leq i_1<\ldots<i_n\leq N} F(a_{i_1},\dotsc,a_{i_N})\]
and where, for $1\leq k\leq n_*$, $g_k:\mathcal{X}_{\not=}^k\rightarrow\R$ is a symmetric function such that 
\begin{align*}
0&=\E\bigl[g_k(X_1,\dotsc,X_k)\,\bigl|\,X_1,\dotsc,X_{k-1}\bigr]\\
&=  \sum_{y\in\mathcal{X}\setminus\{X_1,\dotsc,X_{k-1}\}} g(y,X_1,\dotsc,X_{k-1})\P\bigl(X_k=y\,\bigl|\, X_1,\dotsc,X_{k-1}\bigr)\\
&=\1_{\bigl\{(X_1,\dotsc,X_{k-1})\in\mathcal{X}_{\not=}^{k-1}\bigr\}}\frac{1}{N-k+1}\sum_{y\in\mathcal{X}\setminus\{X_1,\dotsc,X_{k-1}\}} g(y,X_1,\dotsc,X_{k-1})\,.
\end{align*}
Such a function $g_k$ is then, as in the case of independent observations, called a \textit{degenerate, symmetric kernel of order $k$}. From this the following seemingly stronger property of $g_k$ follows:
For all $0\leq r\leq k-1$, all $1\leq i_1<\ldots<i_k\leq n$ and all $1\leq j_1<\ldots<j_r\leq n$ one has that
\begin{equation}\label{degsrs}
\E\bigl[g_k(X_{i_1},\dotsc,X_{i_k})\,\bigl|\, X_{j_1},\dotsc,X_{j_r}\bigr]=0\,.
\end{equation}
It follows from \eqref{degsrs} that, as in the independent case, the \textit{Hoeffding components} $U_k$, $1\leq k\leq n_*$, are orthogonal in $L^2(\P)$. 
We remark that in \cite{BlozGo2001} completely explicit formulas similar to \eqref{defpsip} are given for the degenerate kernels $g_k$.\\

Let $\mathcal{S}_0:=\R$ be the space of real constant functions and, for $1\leq p\leq n$, let $\mathcal{S}_p$ be the linear space of all functionals $F_p$ of the form 
\begin{equation}\label{Fpsrs}
F_p(x_1,\dotsc,x_n)=\sum_{1\leq i_1<\ldots<i_p\leq n} g(x_{i_1},\dotsc,x_{i_p})\,,\quad (x_1,\dotsc,x_n)\in\mathcal{X}_{\not=}^n\,,
\end{equation}
where $g$ is a degenerate, symmetric kernel of order $p$. Such a function is called a \textit{degenerate, symmetric $U$-statistic} with kernel $g$ and based on $X_1,\dotsc,X_n$.
Then, from the above facts about Hoeffding decompositions it follows that 
\begin{equation}\label{L2decsrs}
L^2(\mu)=\bigoplus_{p=0}^{n_*}\mathcal{S}_p\,.
\end{equation}
In particular, one has $\mathcal{S}_p=\{0\}$ for $n_*+1\leq p\leq N$, i.e. for such a $p$ there is no non-zero degenerate, symmetric kernel of order $p$. Again, we refer to $\mathcal{S}_p$ as the \textit{Hoeffding space of order $p$}.
We have the following analogue of Proposition \ref{eigenind}.

\begin{prop}\label{eigensrs}
For $p=0,1,\dotsc,n_*$, the space $\mathcal{S}_p$ equals the eigenspace of $-L$ corresponding to the eigenvalue $\lambda_p:=\frac{p(N-p+1)}{n(N-n)}$. In particular, $L$ is diagonalizable and has a spectral gap of $\frac{N}{n(N-n)}$. 
Moreover, one has $\dim\mathcal{S}_p=\binom{N}{p}-\binom{N}{p-1}$ for $1\leq p\leq n_*$ and $\dim\mathcal{S}_0=1$.
\end{prop}

\begin{proof}
We first prove the claim about the dimension of the $\mathcal{S}_p$. Clearly, $\dim\mathcal{S}_0=1$. Let $1\leq p\leq n_*$. By the uniqueness of the Hoeffding decomposition it follows that 
$\dim\mathcal{S}_p=\dim V_p$, where, for $1\leq k\leq N/2$, $V_k$ is the space of all symmetric functions $g:\mathcal{X}_{\not=}^k\rightarrow \R$ such that for all $(x_1,\dotsc,x_{k-1})\in\mathcal{X}_{\not=}^{k-1}$ one has 
\[\sum_{y\not=x_1,\dotsc,x_{k-1}}g(x_1,\dotsc,x_{k-1},y)=0\,.\]
In particular, $V_k$ does not depend on $n$. Taking dimensions on both sides of \eqref{L2decsrs} we obtain that 
\[\binom{N}{n}=1+\sum_{p=1}^{n_*} \dim \mathcal{S}_p=1+\sum_{p=1}^{n_*} \dim V_p\]
for all $1\leq n\leq \lfloor N/2\rfloor$. Thus, (replacing $n$ with $p$) for  $1\leq p\leq \lfloor N/2\rfloor$ we obtain that 
\[\dim V_p=\binom{N}{p}-\binom{N}{p-1}.\]
Next, we prove that 
\[\mathcal{S}_p\subseteq \ker\Bigl(L+\lambda_p \Id\Bigr)\,,\quad 0\leq p\leq n_*\,.\]
 This is clear for $p=0$. If $1\leq p\leq n_*$, take $F_p\in \mathcal{S}_p$ as in \eqref{Fpsrs}
and define 
\begin{align*}
U_p&:=F_p(X_1,\dotsc,X_n)=\sum_{1\leq i_1<\ldots<i_p\leq n} g(X_{i_1},\dotsc,X_{i_p})\quad\text{as well as}\\
 U_p'&:=F_p(X_1',\dotsc,X_n')=\sum_{1\leq i_1<\ldots<i_p\leq n} g(X_{i_1}',\dotsc,X_{i_p}')\,.
\end{align*}
Then, with obvious notation we have 
\[U_p'-U_p=\sum_{i=1}^n\1_{\{I=i\}}\sum_{\substack{J\in\D_p(n):\\ i\in J}} \Bigl(g\bigl(X_j, j\in J\bigr)-g\bigl(X_{n+1},X_j,j\in J\setminus\{i\}\bigr)\Bigr) \]
and, hence,
\begin{align*}
LF_p(X)&=\E\bigl[U_p'-U_p\,\bigl|\, X\bigr]\\
&=\frac{1}{n}\sum_{i=1}^n\sum_{\substack{J\in\D_p(n):\\ i\in J}} \Bigl(\E\bigl[g\bigl(X_{n+1},X_j,j\in J\setminus\{i\}\bigr)\bigl|\, X_1,\dotsc,X_n\bigr]
-g\bigl(X_j, j\in J\bigr)  \Bigr)\\
&=\frac{1}{n}\sum_{i=1}^n\sum_{\substack{J\in\D_p(n):\\ i\in J}} \Bigl(-g\bigl(X_j, j\in J\bigr)-\frac{1}{N-n}\sum_{l\in([n]\setminus J)\cup\{i\}} g\bigl(X_l,X_j,j\in J\setminus\{i\}\bigr)\Bigr)\\
&=-\frac{p}{n}U_p-\frac{1}{n(N-n)}\sum_{i=1}^n\sum_{\substack{J\in\D_p(n):\\ i\in J}}\sum_{l\in([n]\setminus J)\cup\{i\}} g\bigl(X_l,X_j,j\in J\setminus\{i\}\bigr)\Bigr)\\
&=-\frac{p}{n}U_p-\frac{  p(n-p+1)}{n(N-n)}U_p=-\frac{p(N-p+1)}{n(N-n)}U_p=-\frac{p(N-p+1)}{n(N-n)}F_p(X)\,.
\end{align*}
For the second identity we used the fact that
\[\E\bigl[g\bigl(X_{n+1},X_j,j\in J\setminus\{i\}\bigr)\bigl|\, X_1,\dotsc,X_n\bigr]=-\frac{1}{N-n}\sum_{l\in([n]\setminus J)\cup\{i\}} g\bigl(X_l,X_j,j\in J\setminus\{i\}\bigr)\,,\]
which follows from equation (2.11) of \cite{ZhaoCh} (or else from a direct computation). Hence, we have shown that 
\[L F_p=-\lambda_p F_p\,.\]
It remains to prove the inclusions 
\[\mathcal{S}_p\supseteq \ker\Bigl(L+\lambda_p \Id\Bigr)\,,\quad 0\leq p\leq n_*\,.\]
Let $0\leq p\leq n_*$ and $F\in \ker\Bigl(L+\lambda_p \Id\Bigr)$. Then, by \eqref{L2decsrs} there are $F_j\in\mathcal{S}_j$, $0\leq j\leq n_*$, such that 
 \[F=\sum_{j=0}^{n_*} F_j\,.\]
By the inclusions just proved and by the orthogonality of the eigenspaces of the self-adjoint operator $L$, this implies that, for $j\not=p$, we have 
\[0=E_\mu\bigl[F F_j\bigr]=\sum_{k=0}^{n_*}  E_\mu[ F_kF_j]=E_\mu[F_j^2]\,.\]
Hence, $F_j=0$ and $F=F_p\in \mathcal{S}_p$.
\end{proof}

In order to cope with the carr\'e du champ operator $\Gamma$ corresponding to $L$, we prove the following analogue of Proposition \ref{Prodind}.
\begin{prop}\label{prodsrs}
Let $1\leq p,q\leq n_*$ be integers and suppose that $F\in\mathcal{S}_p$ and $G\in \mathcal{S}_q$ are given. Then, $F(X)G(X)$ has a Hoeffding decomposition of the form 
\begin{align*}
F(X)G(X)&=E_\mu[FG]+\sum_{k=1}^{(p+q)\wedge n_*}\sum_{J\in\D_k(n)} h_k(X_j,j\in J)
=E_\mu[FG]+\sum_{k=1}^{(p+q)\wedge n_*} H_k(X)\,,
\end{align*}
where, for $1\leq k\leq (p+q)\wedge n_*$, $h_k$ is a degenerate, symmetric kernel of order $k$ and we define $H_k\in\mathcal{S}_k$ by
\[H_k(x_1,\dotsc,x_n):=\sum_{J\in\D_k(n)} h_k(x_j,j\in J)\,.\]
Moreover, in terms of this Hoeffding decomposition of the product $F(X)G(X)$ one has that 
\begin{align*}
\Gamma(F,G)&=\frac{1}{2}\sum_{k=1}^{(p+q)\wedge n_*}\bigl(\lambda_p+\lambda_q-\lambda_k\bigr) H_k\\
&=\frac{1}{2n(N-n)}\sum_{k=1}^{(p+q)\wedge n_*}\bigl((N+1)(p+q-k)-p^2-q^2+k^2\bigr) H_k\,.
\end{align*}
\end{prop}

\begin{proof}
The claim about the Hoeffding decomposition is clear by the definition of the spaces $\mathcal{S}_p$ and $\mathcal{S}_q$ and since $FG$ is a symmetric functional on $E$. The representation of $\Gamma$ then follows from the definition $2\Gamma(F,G)=L(FG)-FLG-GLF$ and from Proposition \ref{eigensrs}.
\end{proof}

We have now provided the main ingredients to prove analogues of Theorems \ref{genboundind} and \ref{genboundsymstat} about the normal approximation of functionals of a simple random sample, again by an application of the abstract Theorem \ref{genbound2}. 
There are, however, some important differences to keep in mind. For instance, in contrast to the situation of Subsection \ref{independent}, we here have $\lambda_{p+q}\not= \lambda_p+\lambda_q$. This introduces 
an additional (but controllable) term when dealing with the term $\Var_\mu(\Gamma(F_p,F_q))$. Moreover, a substitute for \cite[Lemma 2.10]{DP16}, which is an essential tool in the proof of Theorem\ref{genboundind}, is needed. Since dealing with these issues necessarily involves some technicalities, we consider the objective of giving explicit bounds on the normal approximation of a functional $F\in L^2(\mu)$ as a separate problem for future work. We just mention here that the asymptotic normality with or without providing error bounds for such functionals has been dealt with in several papers, like e.g. \cite{NanSen2, ZhaoCh, KokWeb, BlozGo2001, BlozGo2002}. In all these papers however, the functional $F$ is assumed to be non-degenerate, meaning that $\Var(g_1(X_1))>0$. On the contrary, as in the independent case, it is to be expected that our approach can deal with all possible scenarios of degeneracy and non-degeneracy at the same time and is thus more flexible.

\subsection{Functionals on finite groups}\label{groups}

In this subsection we let $(G,\cdot)$ be a finite group of order $|G|=m$ and denote by $\mathcal{P}(G)$ its power set. Moreover, we denote by $\mu$ the uniform distribution on $(G,\mathcal{P}(G))$. We will write 
$x_1=e,x_2,\dotsc,x_m$ for the pairwise distinct elements of $G$, with $e$ denoting its identity element.
We may and will identify $\mu$ with the row vector $\mu=(\mu(x_1),\dotsc,\mu(x_m))$ in what follows.
For a fixed probability measure $T$ on $(G,\mathcal{P}(G))$ we define the matrix $K=(K(x,y))_{x,y\in G}\in\R^{G\times G}$ by $K(x,y):=T(yx^{-1})$. Then, as is well-known, $K$ is always a doubly stochastic matrix and the 
corresponding Markov chain is customarily called the \textit{random walk on $G$} induced by the probability measure $T$. 
Since $K$ is doubly stochastic, the uniform distribution $\mu$ is always stationary for $K$, i.e. $\mu K=\mu$. Moreover, $\mu$ is the only stationary distribution for $K$ if and only if $K$ is irreducible, which holds if and only if the support 
\[\mathcal{S}(T):=\{g\in G\,:\, T(g)>0\}\]
of $T$ generates $G$ as a group. Moreover, $\mu$ is even reversible with respect to $K$ if and only if $K=K^T$ is symmetric, which is equivalent to the condition that 
\begin{equation}\label{gr1}
 T(x)=T(x^{-1})\quad\text{for all } x\in G\,.
\end{equation}
In what follows, we will always assume that \eqref{gr1} holds as well as that $\mathcal{S}(T)$ generates $G$, so that $K$ is irreducible and the unique stationary distribution $\mu$ is also reversible with respect to $K$.  

Denote by $\Irr(G)=(V_i,\rho_i)$ a complete system of pairwise non-equivalent irreducible representations of $G$ and, for $i\in I$, let $d_i$ denote the dimension and $\chi_i$ the character of the 
representation $\rho_i$, respectively. 
Note that in this situation, the space $L^2(\mu)$ coincides with the space $ \C^G$ of all complex-valued functions on $G$ and that the inner product on $L^2(\mu)$ is given by 
\begin{equation*}
 \langle f,g\rangle_\mu =\frac{1}{m}\sum_{j=1}^m f(x_j)\overline{g(x_j)}\,.
\end{equation*}
By $\rho_{reg}$ we denote the \textit{left-regular representation} on $G$ which acts on $L^2(\mu)$ as follows:
\begin{equation*}
 (\rho_{reg}(x) f)(y):=f(x^{-1}y)
\end{equation*}
It is well-known (see e.g. \cite{Dia87, Serre}) that, for $i\in I$, $\rho_{reg} $ contains exactly $d_i$ copies of $\rho_i$. More precisely, this means that 
\begin{align}\label{decreg}
 L^2(\mu)=\bigoplus_{i\in I}\bigoplus_{j=1}^{d_i} W_{i,j}\,,
\end{align}
where, for all $i\in I$ and $1\leq j\leq d_j$, $W_{i,j}$ is an \textit{invariant} subspace of $L^2(\mu)$ with respect to $\rho_{reg}$ on which $\rho_{reg}$ acts like an equivalent copy $\tilde{\rho}_i$ of $\rho_i$, i.e.
\begin{equation*}
 \rho_{reg}(x)f=\tilde{\rho}_i(x)f\in W_{i,j}\quad\text{for all } f\in W_{i,j}\text{ and } x\in G\,.
\end{equation*}
The Fourier transform $\hat{f}$ of $f\in L^2(\mu)$ evaluated at a representation $\rho$ is defined by 
\begin{align*}
 \hat{f}(\rho):=\sum_{x\in G} f(x)\rho(x)\,.
\end{align*}
Hence, we conclude that each space $W_{i,j}$ in the above decomposition of $L^2(\mu)$ is also invariant with respect to the linear operator 
$\hat{f}(\rho_{reg})$.
We call $f\in L^2(\mu)$ a \textit{class function} if $f$ is constant on the conjugacy classes of $G$, i.e. if 
\[f(yxy^{-1})=f(x)\quad\text{for all }x,y\in G\,.\]
It is well-known that, if $f$ is a class function and $(\rho,V)$ is an irreducible representation of dimension $d$ and with character $\chi$, then 
\begin{align*}
 \hat{f}(\rho)=\lambda_\rho id_{V}\,,\quad\text{where } \lambda_\rho =\frac{1}{d}\sum_{x\in G} f(x)\chi(x)=\frac{|G|}{d}\langle f,\overline{\chi}\rangle_{L^2(\mu)}\,.
\end{align*}

Now let us consider the operator $L$ that is associated to $K$ or to the exchangeable pair $(X,X'):=(X_0,X_1)$, where $(X_n)_{n\in\N_0}$ is a Markov chain with transition matrix $K$ and initial distribution $\mu$.
In this context, since $K=K^T$, it follows from \cite[Lemma 15]{DiaSh81} that $Kf=\hat{T}(\rho_{reg})f$ for all $f\in L^2(\mu)$ (see also 
\cite[Theorem 3.6]{Dia87}). Hence, we have that 
\[L=\hat{T}(\rho_{reg})-\Id\,.\]
 In particular, the spaces $W_{i,j}$ appearing in \eqref{decreg} are also invariant with respect to $L$. 
From now on denote by $C_1,\dotsc,C_K$ the pairwise distinct conjugacy classes of the group $G$ and by $m_j:=|C_j|$, $j=1,\dotsc,K$, their cardinalities. Note that we have that $K=|I|$ is also the number of pairwise non-equivalent irreducible representations of $G$. Moreover, we will assume that the probability measure $T$ is constant on conjugacy classes and denote its value on the class 
$C_j$ by $p_j$, $j=1,\dotsc,K$. We write $\chi_i(K_j)$ for the value of $\chi_i$ on an arbitrary element of $K_j$. Now, consider the numbers 
\begin{equation}\label{eigDS}
c_i:=\frac{1}{d_i}\sum_{j=1}^K m_j p_j \chi_i(K_j)=\frac{m}{d_i}\langle \chi_i,T\rangle_{\mu}\,,\quad i\in I,
\end{equation}
and define the set  
\[S:=\{ 1-c_i\,:\, i\in I\}\,.\]
In \cite[Theorem 3]{DiaSh81} Diaconis and Shahshahani proved that the $c_i$, $i\in I$, are precisely the eigenvalues of $K$ and that each has multiplicity $d_i^2$.
Since $K=K^T$ these are all real. We denote by $0=\lambda_0<\lambda_1\dotsc<\lambda_r\leq1$ the pairwise distinct elements of $S$, i.e. the pairwise distinct eigenvalues of the positive 
operator $-L$. W.l.o.g. from now on we will assume that the 
irreducible representations $\rho_i$, $i\in I$, are given as unitary matrix representations, i.e. for each $x\in G$, 
$\rho_i(x)=(\rho_i(x,k,j))_{1\leq k,j\leq d_i}$ is a $d_i\times d_i$ unitary matrix. Then, the collection of functions
\[\Bigl\{\phi_{i,l,j}:=\frac{d_i}{m}\rho_i(\cdot,l,j)\,:\, i\in I,\;1\leq l,j\leq d_i\Bigr\}\]
is an orthonormal basis of the Hilbert space $L^2(\mu)$ (see e.g. \cite[Corollary 2.3]{Dia87}).

\begin{prop}
Under the above assumptions, the self-adjoint operator $-L$ is diagonalizable and \eqref{decreg} gives the orthogonal decomposition of $L^2(\mu)$ into eigenspaces of $-L$. More precisely, the set $S$ coincides with the set of eigenvalues of $-L$ and, for each $k=1,\dotsc,r$, 
the multiplicity of $\lambda_k$ in the operator $-L$ is given by $\sum_{i\in I: 1-c_i=\lambda_k}d_i^2$. Furthermore, an orthonormal basis 
of the eigenspace $E_k$ corresponding to $\lambda_k$ is given by the collection of functions   
\[\Bigl\{
 \phi_{i,l,j}\,\Bigl|\, i\in I:\, 1-c_i=\lambda_k,\;1\leq l,j\leq d_i\Bigr\}\]
and we have 
\[E_k=\bigoplus_{\substack{i\in I:\\1- c_i=\lambda_k}}\bigoplus_{j=1}^{d_i} W_{i,j}\,.\]
For $k=0$, the multiplicity of the eigenvalue $\lambda_0=0$ in $-L$ is equal to $1$ and the collection of constant functions coincides with the eigenspace corresponding to $0$.
\end{prop}

\begin{proof}
Since $L=K-\Id$, this is an immediate consequence of \cite[Corollary 3]{DiaSh81} and \cite[Theorem 3.6]{Dia87}. The fact that $0$ is a simple eigenvalue of $-L$ follows easily from Perron-Frobenius theory, since the Markov chain corresponding to $K$ is obviously irreducible and aperiodic.
\end{proof}

Now, if $f\in L^2(\mu)$ is an arbitrarily given function, then we have 
\begin{align*}
f&=\sum_{i\in I}\sum_{l,j=1}^{d_i} \langle f, \phi_{i,l,j}\rangle \phi_{i,l,j}\quad\text{and}\\ 
Lf&=-\sum_{i\in I}(1-c_i)\sum_{l,j=1}^{d_i} \langle f, \phi_{i,l,j}\rangle \phi_{i,l,j}=-\sum_{k=1}^r\lambda_k\sum_{\substack{i\in I:\\ 1-c_i=\lambda_k}}\sum_{l,j=1}^{d_i} \langle f, \phi_{i,l,j}\rangle \phi_{i,l,j}\,.
\end{align*}
In particular, we see that $\ker(L)$ is the space of constant functions and that $\im(L)=\{1\}^{\perp}$. Hence, $L$ is ergodic and its pseudo-inverse $L^{-1}$ is given on $f$ as above by 
\begin{equation*}
L^{-1}f=-\sum_{k=1}^r\lambda_k^{-1}\sum_{\substack{i\in I:\\ 1-c_i=\lambda_k}}\sum_{l,j=1}^{d_i} \langle f, \phi_{i,l,j}\rangle \phi_{i,l,j}\,.
\end{equation*}

In order to make the theory from Section \ref{setup} work in this situation, one needs expressions for the the quantities $\Gamma(f,g)$ analogous to those given in Subsections \ref{independent} and 
\ref{srs}. Of course, if $f\in E_k$ and $g\in E_l$, then it always holds that 
\begin{align*}
\Gamma(f,g)&= \frac12\Bigl(\sum_{s=1}^{r}\bigl(\lambda_k+\lambda_l-\lambda_s\bigr)\sum_{\substack{i\in I:\\ 1-c_i=\lambda_s}}\sum_{l,j=1}^{d_i} \langle fg, \phi_{i,l,j}\rangle \phi_{i,l,j}\Bigr)\,,
\end{align*}
but to obtain manageable expressions, it would be beneficial to at least be able to express products $\phi_{i,l,j}\cdot \phi_{i',l',j'}$ of two basis functions again in terms of this basis. In other words, one would profit from the knowledge of the so-called \textit{Clebsch-Gordan coefficients} associated with this group and basis. In the analogous problem for class functions $f$ and $g$, the task becomes simpler because these are linear combinations of the irreducible characters of the group and the Clebsch-Gordan coefficients can be replaced with the decomposition of the \textit{tensor product} of two irreducible representations into irreducibles. This path might be viable in more concrete situations like symmetric groups, for instance, and will be pursued elsewhere.

\section{Applications}\label{apps}
\subsection{Subgraph counting in geometric random graphs}\label{counts}

Geometric random graphs are graphs whose vertices are random points that are scattered in some metric space and where two vertices are connected by an edge, if and only if their distance is less than a prescribed deterministic radius. These graphs play an important role e.g. in statistical test theory as well as in the modeling of telecommunication networks and, hence, represent a fundamental alternative to the merely combinatorial Erd\"os-R\'enyi random graphs. We refer to the monographs \cite{Penrose} and \cite{PecRei16} for an introduction to random geometric graphs on Euclidean domains and to some of their applications.

 Even more recently, it has been discovered that random geometric graphs on hyperbolic spaces reflect several fundamental properties of large real world networks like \textit{scale freenes}, \textit{the small world phenomenon} and \textit{strong clustering} \cite{HG1, HG2, HG3, HG4} and have thus become a very popular model for investigating further properties of such networks.\\

We will focus here on the Euclidean case and use our Theorem \ref{symustat2} in order to prove rates of convergence in the Wasserstein distance for the Gaussian fluctuations of subgraph counts. 
At a qualitative level, the asymptotic (jointly) Gaussian behaviour of these counts is well understood both in the binomial \cite{BhaGo92, JJ, Penrose} and in the Poisson point process 
\cite{Penrose, LRP2} situation, but until the very recent paper \cite{DP19}, the quantitative aspect of this distributional convergence has been neglected for the binomial point process situation. On the contrary, in the situation of a Poisson point process, rates of convergence for these counts in the Wasserstein distance have been provided in \cite{LRP2}. As already hinted at, in \cite{DP19} the problem of finding quantitative bounds also in the binomial was addressed but, due to a multivariate detour, these bounds were only expressed for $C^3$ test functions with bounded derivatives. The reason for this detour was precisely the fact that the authors had to rely on a result for exchangeable pairs satisfying a (multivariate) linear regression property. Hence, we will already see by way of this example how the new methods from this paper pay off. Moreover, the derivation of Theorem 6.2 in \cite{DP19} involved a slightly incorrect application of Lemma 5.1 of the same paper and, hence, deserves some rectification. 

\medskip

To describe our model, we fix a dimension $d\geq1$ as well as a bounded and Lebesgue almost everywhere continuous probability density function $f$ on $\R^d$. Let $\nu(dx):=f(x)dx$ be the corresponding probability measure on $(\R^d,\B(\R^d))$ and suppose that $X_1,X_2,\dotsc$ are i.i.d. with distribution $\nu$. Let $X:=(X_j)_{j\in\N}$. We denote by $(t_n)_{n\in\N}$ a sequence of radii in $(0,\infty)$ such that $\lim_{n\to\infty}t_n=0$ and, for each $n\in\N$, we denote by $G(X;t_n)$ the 
\textit{random geometric graph} obtained as follows: The vertices of $G(X;t_n)$ are given by the points of the set $V_n:=\{X_1,\dotsc,X_n\}$, which $\Prob$-a.s. has cardinality $n$, 
and two vertices $X_i,X_j\in V_n$ are connected, if and only if $0<\Enorm{X_i-X_j}<t_n$. Furthermore, let $m\geq2$ be a fixed integer and suppose that 
$\Delta$ is a fixed connected graph on $m$ vertices. For each $n$ we denote by $G_n(\Delta)$ the number of induced subgraphs of $G(X;t_n)$ that are isomorphic to $\Delta$. 
For technical reasons, we have to assume that $\Delta$ is \textit{feasible} for every 
$n\geq m$. This means that the probability that the restriction of $G(X;t_n)$ to $X_1,\dotsc,X_m$ is isomorphic to $\Delta$ is strictly positive for $n\geq m$. Note that feasibility depends on the common distribution $\nu$ of the points.
The quantity $G_n(\Delta)$ is a symmetric $U$-statistic of order $m$ based on $X_1,\dotsc,X_n$ since 
\begin{equation*}
 G_n(\Delta)=\sum_{1\leq i_1<\ldots<i_m\leq n} \psi_{\Delta,t_n}(X_{i_1},\dotsc,X_{i_m})\,,
\end{equation*}
where $\psi_{\Delta,t_n}(x_1,\dotsc,x_m)$ equals $1$ if the graph with vertices $x_1,\dotsc,x_m$ and edge set $\{\{x_i,x_j\}\,:\, 0<\Enorm{x_i-x_j}< t_n\}$ is isomorphic to $\Delta$ and $0$, otherwise. 

In what follows, for $a_n,b_n>0$, $n\in\N$, we write $a_n\sim b_n$ if $\lim_{n\to\infty} a_n/b_n=1$, $a_n\ll b_n$, if $\lim_{n\to\infty} a_n/b_n=0$, $a_n\lesssim b_n$, if $\limsup_{n\to\infty} a_n/b_n<\infty$ ).

For obtaining asymptotic normality one distinguishes between three different asymptotic regimes:
\begin{enumerate}
 \item[\textbf{(R1)}] $nt_n^d\to0$ and $n^m t_n^{d(m-1)}\to\infty$ as $n\to\infty$ (\textit{sparse regime})
 \item[\textbf{(R2)}] $n t_n^d\to\infty$ as $n\to\infty$ (\textit{dense regime})
 \item[\textbf{(R3)}] $nt_n^d\to\rho\in(0,\infty)$ as $n\to\infty$ (\textit{thermodynamic regime})
\end{enumerate}
Note that we could rephrase the regimes  $\textbf{(R1)}$ and $\textbf{(R2)}$ as follows:
\begin{enumerate}
 \item[\textbf{(R1)}] $\displaystyle \Bigl(\frac{1}{n}\Bigr)^{\frac{m}{m-1}}\ll t_n^d\ll\frac{1}{n}$  
\item[\textbf{(R2)}] $\displaystyle \frac{1}{n}\ll t_n^d$,
\end{enumerate}
It turns out that, contrary to the Poisson case, under regime \textbf{(R2)} one also has to take into account whether the common distribution $\nu$ of the $X_j$ is the uniform distribution $\mathcal{U}(M)$ on some Borel subset $M\subseteq\R^d$, $0<\lambda^d(M) <\infty$ with 
density $f(x)=\lambda^d(M)^{-1}\,\1_M(x)$, or not. 
Thus, we will distinguish between the following four cases:
\begin{enumerate}
 \item[\textbf{(C1)}] $nt_n^d\to0$ and $n^m t_n^{d(m-1)}\to\infty$ as $n\to\infty$. 
 \item[\textbf{(C2)}] $n t_n^d\to\infty$ as $n\to\infty$ and $\nu= \mathcal{U}(M)$ for some Borel subset $M\subseteq\R^d$ s.t. $0<\lambda^d(M) <\infty$.
 \item[\textbf{(C3)}] $n t_n^d\to\infty$ as $n\to\infty$, and $\nu$ is not a uniform distribution.
 \item[\textbf{(C4)}] $nt_n^d\to\rho\in(0,\infty)$ as $n\to\infty$.
 \end{enumerate}

The following important variance estimates have been established.  
\begin{prop}[see \cite{Penrose}, Theorems 3.12 and 3.13, \cite{DKP}, Proposition 4.1]\label{regprop}
 Under all regimes {\normalfont \textbf{(R1)}}, {\normalfont\textbf{(R2)}} and {\normalfont\textbf{(R3)}} it holds that\\
 $\E[G_n(\Delta)]\sim c n^{m}t_n^{d(m-1)}$ for some constant $c\in(0,\infty)$. Moreover, there exist constants $c_1,c_2,c_3,c_4\in(0,\infty)$ such that, as $n\to\infty$,  
   \begin{enumerate}
  \item[{\normalfont \textbf{(C1)}}] $\Var(G_n(\Delta))\sim c_1\cdot n^m t_n^{d(m-1)}$,
  \item[{\normalfont \textbf{(C2)}}] $\Var(G_n(\Delta))\gg c_2\cdot n^{2m-2} t_n^{d(2m-3)}$,
  \item[{\normalfont \textbf{(C3)}}] $\Var(G_n(\Delta))\sim c_3\cdot n^{2m-1} t_n^{d(2m-2)}$,
  \item[{\normalfont \textbf{(C4)}}] $\Var(G_n(\Delta))\sim c_4\cdot n$.
  \end{enumerate}
\end{prop}

We denote by 
\begin{equation*}
 F:=F_n:=\frac{G_n(\Delta)-\E[G_n(\Delta)]}{\sqrt{\Var(G_n(\Delta))}}
\end{equation*}
the normalized version of $G_n(\Delta)$. The following statement follows directly from Theorem \ref{symustat2}.

\begin{theorem}\label{subcounts}
 Let $Z$ be a standard normal random variable. Then, with the above definitions and notation, there exists a finite constant $C>0$ which is 
 independent of $n$ such that for all $n\geq m$,
 \begin{align*}
  d_\W(F,Z)&\leq C\cdot\bigl(n^m t_n^{d(m-1)}\bigr)^{-1/2}\quad\text{in case {\normalfont\textbf{(C1)}}},\\
  d_\W(F,Z)&\leq C\cdot n^{-1/2} \quad\text{in cases {\normalfont\textbf{(C3)}} and {\normalfont\textbf{(C4)}, }}\\
  d_\W(F,Z)&\leq C\cdot \bigl(nt_n^{2d}\bigr)^{1/2} \quad\text{in case {\normalfont\textbf{(C2)}}}.
 \end{align*}
In particular, we have that $F_n$ always converges in distribution to $Z$ as $n\to\infty$ in the cases {\normalfont \textbf{(C1)}}, {\normalfont \textbf{(C3)}} and {\normalfont\textbf{(C4)}}. 
In case {\normalfont\textbf{(C2)}}, it follows that $F_n$ converges in distribution to $Z$ under the additional assumption that $\lim_{n\to\infty} nt_n^{2d}=0$.
 \end{theorem}

\begin{proof}
We apply the first bound in Theorem \ref{symustat2} by making use of several computations and estimates from the papers \cite{LRP2, DKP, DP19}. Let $\sigma_n^2:=\Var(G_n(\Delta))$. 
In the proof of \cite[Theorem 4.2]{DKP}, the following estimates have been proved (with the change of variable $l=t-r$): For all $1\leq p,q\leq m$, all $1\leq t\leq 2(p\wedge q)\wedge p+q-1$, 
all $\lceil t/2\rceil\leq r\leq t\wedge p\wedge q$ and all $(j,k,a,b)\in Q(p,q,r,t-r)$ it holds that
\begin{align*}
\frac{n^{4m-(p+q+2r-t)}}{\sigma_n^4}\,\norm{g_{j}\star_a^{b}g_{k}}_{2}^2 &\lesssim
\begin{cases}
\Bigl(n^mt_n^{d(m-1)}\Bigr)^{-1}\,,&\text{in case \textbf{(C1)}}\\
nt_n^{2d}\,,&\text{in case \textbf{(C2)}}\\
n^{-1}\,,&\text{in cases \textbf{(C3)} and \textbf{(C4)}}.\\
\end{cases}
\end{align*}
Then, the result follows from Theorem \ref{symustat2} by observing that $n^mt_n^{d(m-1)}=O(n)$ in case \textbf{(C1)} and $nt_n^{2d}=O(n)$ in case \textbf{(C2)} so that the term of order 
$n^{-1/2}$ appearing in the bound does not affect the rate of convergence.
\end{proof}

\begin{remark}\label{remcounts} 
\begin{enumerate}[(a)]
 \item The fact that, with the same level of technicality as in the Poisson situation of \cite{LRP2}, we obtain comparable bounds on the Wasserstein distance in Theorem \ref{subcounts} 
demonstrates the utility of Theorem \ref{symustat2}. In the existing literature, proofs of asymptotic normality in the binomial setting often relied on a de-Poissonization technique due to \cite{DynMan83}. It might be possible to use this method in order to obtain quantitative results as well but complicated expressions involving conditional variances would make the computations very cumbersome.
 \item Note that in all three regimes \textbf{(R1)}, \textbf{(R2)} and \textbf{(R3)} considered in Theorem \ref{subcounts}, one has that $\lim_{n\to\infty} n^mt_n^{d(m-1)}=+\infty$. Indeed, it is shown in \cite[Section 3.2]{Penrose} that  $F_n$ converges weakly to a Poisson distribution if $\lim_{n\to\infty} n^mt_n^{d(m-1)}=\alpha\in (0,\infty)$ whereas it converges to $0$, if $\lim_{n\to\infty} n^mt_n^{d(m-1)}=0$. Hence, $\lim_{n\to\infty} n^mt_n^{d(m-1)}=+\infty$ is a necessary condition for the asymptotic normality of $G_n(\Delta)$.
 \item We mention that the distinction between uniform and non-uniform distributions is not necessary for the analogous problem on Poisson space considered in \cite{Penrose,LRP2}. The reason is that, in the Poisson case, the formulae for the respective limiting variances are slightly different, see \cite[Section 3.2]{Penrose}. Moreover, the phenomenon that, in the case of a uniform distribution on a set $M$, the asymptotic order of the variance is different in the dense regime \textbf{(R2)} has already been observed in the references \cite[Section 4]{JJ} and in \cite[Theorem 2.1,Theorem 3.1]{BhaGo92}.
Moreover, it is remarked on page 1357 of \cite{JJ} for the very special case $m=2$ of edge counting, that the asymptotic order of $\Var(G_n(\Delta))$ in fact depends on the boundary structure of the set $M$. For smooth enough boundaries, it is claimed there that $ \Var(G_n(\Delta))\sim c n^2 t_n^d$ for some constant $c\in(0,\infty)$, whenever $t_n=o(n^{-1/(d+1)})$. In particular, there are cases where our lower bound for $\Var(G_n(\Delta))$ in case \textbf{(C2)} given in Proposition \ref{regprop} is sharp.  
 \item We also stress that, in the case of a uniform $\nu$, the condition {\bf($\mathbf{D_3'}$)}, which is assumed in Theorem 3.1 of \cite{BhaGo92} to guarantee asymptotic normality for the number of $m$ clusters {(that is, subgraphs of $m$ vertices that are isomorphic to the complete graph $K_m$ on $m$ vertices)}, is stronger than our additional condition that
 $\lim_{n\to\infty} nt_n^{2d}=0$ if $m\geq3$, whereas for $m=2$ it is exactly the same. Hence, even for the qualitative result on the special case of $m$ clusters, our Theorem \ref{subcounts} improves on what has been known so far.
 \item More generally, in the situation of case \textbf{(C2)}, even the qualitative CLT for $G_n(\Delta)$ given in Theorem \ref{subcounts} (which is an improvement and rectification of \cite[Theorem 6.2]{DP19}) seems to be new. Indeed, as already mentioned in \cite[Remark 6.3 (e)]{DP19}, the scaling used in Theorem 3.12 of \cite{Penrose} leads to a degenerate limit. However, we mention that, in the special case of edge counting ($m=2$) considered in \cite[Section 4]{JJ}, the authors prove that asymptotic normality holds even without the additional assumption that $\lim_{n\to\infty} n t_n^{2d}=0$. 
The main problem in case \textbf{(C2)} is that, in \cite{DP19, DKP}, we only managed to prove lower bounds on the variance $\sigma_n^2$ of $G_n(\Delta)$, whereas in all other cased, the asymptotic 
behaviour of $\sigma_n^2$ is precisely known from \cite{Penrose} (see Proposition \ref{regprop}). Moreover, it seems to be infeasible to generalize the computations from \cite{JJ} to determine $\sigma_n^2$ asymptotically for general $m$ and general connected subgraphs $\Delta$.
\end{enumerate}
\end{remark}

\subsection{Normal approximation of Pearson's statistic}
In this subsection we apply Theorem \ref{genboundsymstat} in order derive bounds on the normal approximation of Pearson's chi-square statistic $\chi^2_n$ that is well-known from Pearson's chi-square goodness-of-fit test \cite{Pearson}. Although it is a classical fact that, for a fixed number of classes, Pearson's statistic is asymptotically chi-square distributed, when the number $m=m_n$ of classes is allowed to diverge to infinity with the sample size $n$, it may have an asymptotic normal distribution. Such double asymptotics of Pearson's statistic have been investigated by many authors. We refer to the recent article \cite{RemWes} for a discussion of the relevant literature. For bounds on the chi-square approximation of $\chi_n^2$, we refer to the recent article \cite{GPR} and the references mentioned therein. \\

Let $m,n$ be positive integers and let $p_n$ be a probability mass function on $\{1,\dotsc,m\}$ such that $p_n(i)>0$ for all $1\leq i\leq m$. Moreover, we let $X_1,\dotsc,X_n$ be independent random variables 
on some probability space $(\Omega,\F,\P)$, each distributed according to $p_n$. Since we are dealing with bounds for fixed (large) $n\in\N$, we do not need to consider triangular arrays of ranom variables, here. However, we keep the dependence of $p_n$ on $n$ in our notation. 
Recall that $\chi_n^2$ is defined by 
\begin{equation*}
\chi_n^2:= n\sum_{i=1}^m\frac{\bigl(\hat{p}_n(i)-p_n(i)\bigr)^2}{p_n(i)}\,,
\end{equation*}
where the
\[\hat{p}_n(i):=\frac{1}{n}\sum_{j=1}^n\1_{\{X_j=i\}}\,,\quad 1\leq i\leq m\,,\]
denote the empirical relative frequencies. Moreover, from \cite{RemWes} we have that
\begin{equation}\label{sigchi}
\sigma_n^2:=\Var\bigl(\chi_n^2\bigr)=\frac{1}{n}\Bigl(\Var\bigl(p_n(X_1)^{-1}\bigr)+2(n-1)(m-1)\Bigr)\,,
\end{equation}
where 
\begin{align}\label{varpn}
\Var\bigl(p_n(X_1)^{-1}\bigr)&=\norm{1/p_n-m}_2^2=\sum_{i=1}^m \frac{1}{p_n(i)} -m^2\,.
\end{align}
Denote by 
\[F:=F_n:=\frac{\chi_n^2-\E\bigl[\chi_n^2\bigr]}{\sqrt{\Var\bigl(\chi_n^2\bigr)}}=\frac{\chi_n^2-(m-1)}{\sigma_n}\]
the normalized Pearson statistic. Then, from \cite{RemWes} we have the Hoeffding decomposition
\begin{align}\label{HDchin}
F&=\frac{1}{n\sigma_n}\sum_{j=1}^n\bigl(p_n(X_j)^{-1}-m\bigr)+\frac{2}{n\sigma_n}\sum_{1\leq j<k\leq n}\biggl(\frac{\1_{\{X_j=X_k\}}}{p_n(X_j)}-1\biggr)\,.
\end{align}
Thus, the kernels $\phi_p$ from \eqref{defphip} are given by 
\begin{align*}
\phi_1(j)&=\frac{p_n(j)^{-1}-m}{n\sigma_n}\,,\\
\phi_2(j,k)&=\frac{2}{n\sigma_n}\biggl(\frac{\1_{\{j=k\}}}{p_n(j)}-1\biggr)\,,
\end{align*}
and $\phi_p\equiv0$ for $p=3,\dotsc,n$.

The following Theorem gives a general bound for the Wasserstein distance between the distribution of $F$ and the standard normal distribution. Its proof is deferred to Section \ref{proofs}.
\begin{theorem}\label{chitheo}
There exist absolute constants $C,C'\in(0,\infty)$ such that the following bounds hold:
\begin{align*}
d_\W(F,Z)&\leq C\max\Biggl(\frac{1}{\sqrt{n}},\,\frac{1}{\sqrt{m}},\,\frac{1}{\sqrt{n}}\frac{\norm{1/p_n -m}_4^2}{n\sigma_n^2},\, \\
&\hspace{2cm}\biggl(\frac{1}{n^{1/4}(m-1)^{1/4}}+\frac{1}{\sqrt{n}}\biggr)\frac{\norm{1/p_n-m}_4}{\sqrt{n}\sigma_n}, \, \frac{\sqrt{\sum_{i=1}^m\frac{1}{p_n(i)^2}}}{n\sigma_n^2}\;\Biggr)\text{ and}\\
d_\W(F,Z)&\leq C'\max\Biggl(\frac{1}{\sqrt{n}},\,\frac{1}{\sqrt{m}},\,\frac{\sqrt{\sum_{i=1}^m\frac{1}{p_n(i)^3}-m^4}}{n^{3/2}\sigma_n^2},\, \frac{\sqrt{\sum_{i=1}^m\frac{1}{p_n(i)^2}}}{n\sigma_n^2}\;\Biggr)\,.
\end{align*}
\end{theorem}

Let us define $p^*:=p_n^*:=\min_{1\leq i\leq m} p_n(i)$. In the classical situation of a fixed number $m$ of classes it is a classical result that $\chi_n^2$ converges in distribution to the chi-square 
distribution with $m-1$ degrees of freedom, as long as $\lim_{n\to\infty}n p_n^*=\infty$ (see e.g. the discussion in \cite[Section 4]{GPR}). Since, for $k\in\N$ we have 
\[\sum_{i=1}^m\frac{1}{p_n(i)^k}\leq \frac{m}{(p_n^*)^k}\,,\]
the following result is a direct consequence of Theorem \ref{chitheo} and of equation \eqref{sigchi}.
\begin{cor}\label{chicor1}
There is an absolute constant $D\in(0,\infty)$ such that 
\begin{align*}
d_\W(F,Z)&\leq D\max\Biggl(\frac{1}{\sqrt{n}},\,\frac{1}{\sqrt{m}},\,  \frac{1}{\sqrt{m}(np_n^*)^{3/2}},\, \frac{1}{(np_n^*)\sqrt{m}}\;\Biggr)\,.
\end{align*}
\end{cor}

We look at some examples in order to illustrate the bounds:
\begin{example}\label{chiex1}
Let $p_n$ be the uniform distribution, i.e. $p_n(i)=m^{-1}$ for $i=1,\dotsc,m$. Then, the above bounds reduce to 
\[d_\W(F,Z)\leq C\max\biggl(\frac{1}{\sqrt{n}},\,\frac{1}{\sqrt{m}},\,\frac{\sqrt{m}}{n}\biggr)\,.\]
In particular, in a sequential situation, we have that $F_n$ converges in distribution to $Z$ if $\lim_{n\to\infty}m_n=+\infty$ in such a way that $m_n=o(n^2)$. Indeed, it was shown in \cite{RemWes} that these conditions are also necessary for the asymptotic normality of $F_n$. Note that, in this situation, $\phi_1\equiv0$ so that $F_n$ is a degenerate, symmetric $U$-statistic of order $2$. Hence, this special result could have also been derived from the findings of the paper \cite{DP19}.
\end{example}

The next example has been dealt with in \cite{RemWes}.
\begin{example}\label{chiex2}
For a fixed $\alpha\in (0,1)$ consider the distribution given by $p_n(i)=c_\alpha i^{-\alpha}$, where 
\[c_\alpha^{-1}=\sum_{i=1}^m i^{-\alpha}\sim \frac{m^{1-\alpha}}{1-\alpha}\,.\]
In this case, we have 
\[\sum_{i=1}^m\frac{1}{p_n(i)^k}=c_\alpha^{-k} \sum_{i=1}^m i^{k\alpha}\sim \frac{m^{k(1-\alpha)}}{(1-\alpha)^k} m^{\alpha k+1}=\frac{m^{k+1}}{(1-\alpha)^k}\,.\]
In this case, our Theorem gives the bound 
\begin{align*}
d_\W(F,Z)&\leq C\max\biggl(\frac{1}{\sqrt{n}},\,\frac{1}{\sqrt{m}},\,\frac{1}{\sqrt{n}\bigl((1-\alpha)^{-1}-1 +\frac{2 n}{m}\bigr)},\\
&\hspace{4cm}\,\frac{1}{\sqrt{m}\bigl( (1-\alpha)^{-1}-1\bigr) + \frac{2n}{\sqrt{m}}}\;\biggr)
\end{align*}
that always converges to zero as long as $\lim_{n\to\infty}m_n=+\infty$. Note that this is an improvement as compared to the result in \cite{RemWes} for this example, who had to assume that $m_n/n\rightarrow 0$. This example also illustrates that in the previous example of a uniform distribution (which is the case $\alpha=0$) a quite unfortunate phenomenon happens to the denominator of the last term of the bound.
\end{example}

\begin{remark}\label{chirem}
\begin{enumerate}[(a)]
\item The absolute constants appearing in Theorem \ref{chitheo} and Corollary \ref{chicor1} could be computed explicitly by carefully evaluating the coefficients in the bound from Theorem \ref{genboundsymstat}. For the sake of conciseness we however refrained from doing so.
\item We briefly compare our results to the existing literature on asymptotic normality for the chi-square statistic. To the best of our knowledge, ours are the first bounds on the normal approximation of Pearson's statistic in the literature. In \cite{Tum} it was shown that $\chi_n^2$ is asymptotically normal as long as $np_n^*\to\infty$ and $m_n\to\infty$. Our Corollary \ref{chicor1} improves on this not only by giving a rate of convergence but also in that the condition may be weakened to $m_n\to\infty$ and $m_n (np_n^*)^3\to\infty$. In \cite{RemWes} the asymptotic distribution problem for $\chi_n^2$ was considered under the additional assumption that the second Hoeffding component of $F_n$ is asymptotically dominant, i.e. that $F_n$ may be written as a degenerate, symmetric $U$-statistic of order $2$ plus a negligible remainder term. In particular they prove that, under this assumption, $\chi_n^2$ is asymptotically normal, if $m_n$ diverges to infinity in such a way that $m_n=o(n^2)$ and the condition $\sup_{n\in\N} m_n^{-(1+\delta)}\E[p_n(X_1)^{-(1+\delta)}]<\infty$ is satisfied for some $\delta>0$. Moreover, they show that $\chi_n^2$ has an asymptotic shifted and scaled Poisson distribution if $m_n/n^2$ converges to some positive real number. In \cite[Remark 3.3]{RemWes} the authors briefly address the case when the first Hoeffding component of $F_n$ is dominant. In this case, of course, the classical theory of sums of independent random variables may be used in order to prove (quantitative) CLTs for $\chi_n^2$. They authors of \cite{RemWes} also mention that for the remaining case, where the variances of both Hoeffding components  of $F_n$ asymptotically have the same order, one would need a different approach. From Theorem \ref{chitheo} and Corollary \ref{chicor1} we see that, using our bound, such a distinction is not necessary. Indeed, in Example \ref{chiex2}, if $m_n$ is of the same order as $n$, we see that both Hoeffding components have variances of the same order $m_n^2$ and our bound still proves 
asymptotic normality. Hence, we conclude that our results also significantly extend the literature on (qualitiative) CLTs for Pearson's statistic.
\item Finally, we mention that one could use Theorem \ref{genboundgamma} in order to derive an analogue of Theorem \ref{genboundsymstat} for centered Gamma approximation. Then, it would be possible to use the methods from this subsection in order to derive bounds for the chi-square approximation of Pearson's statistic. This problem will be addressed in a separate project.
\end{enumerate}
\end{remark}

\section{Proofs}\label{proofs}

\subsection{Proofs from Section \ref{setup}}

\begin{proof}[Proof of Lemma \ref{remlemma}]
Using exchangeability, we obtain that 
\begin{align}\label{rl1}
&\frac{1}{2}\E\bigl[\bigl(F(X')-F(X)^4\bigr)\bigr]=\E\bigl[F(X)\bigl(F(X)-F(X')\bigr)^3\bigr]\notag\\
&=\E\bigl[F(X)^4-3F(X)^3F(X')+3F(X)^2F(X')^2-F(X)F(X')^3\bigr]\notag\\
&=\E\bigl[F(X)^4\bigr]-4\E\bigl[F(X)^3F(X')\bigr]+3\E\bigl[F(X)^2F(X')^2\bigr]\,.
\end{align}
By the definition of $L$ we can write 
\begin{align}\label{rl2}
\E\bigl[F(X)^3F(X')\bigr]&=\E\Bigl[F(X)^3\E\bigl[F(X')\,\bigl|\,X\bigr]\Bigr]
=\E\bigl[F(X)^3\bigl(LF(X)+F(X)\bigr)\bigr]\notag\\
&=\E\bigl[F(X)^4\bigr]+\E\bigl[F(X)^3LF(X)\bigr]\notag\\
&=E_\mu[F^4]+E_\mu\bigl[F^3 LF\bigr]\,,
\end{align}
whereas, by means of Theorem \ref{abstheo} we can write
\begin{align}\label{rl3}
&\E\bigl[F(X)^2F(X')^2\bigr]=\E\Bigl[F(X)^2\E\bigl[\bigl(F(X')-F(X)+F(X)\bigr)^2\,\bigl|\,X\bigr]\Bigr]\notag\\
&=\E\Bigl[F(X)^2\E\bigl[\bigl(F(X')-F(X)\bigr)^2 +2F(X)\bigl(F(X')-F(X)\bigr)+F(X)^2\,\bigl|\,X\bigr]\Bigr]\notag\\
&=\E\bigl[F(X)^4\bigr]+2\E\bigl[F(X)^2\Gamma(F,F)(X)\bigr] +2 \E\bigl[F(X)^3 LF(X)\bigr]\notag\\
&=E_\mu\bigl[F^4\bigr]+2E_\mu\bigl[F^2\Gamma(F,F)\bigr] +2 E_\mu\bigl[F^3 LF\bigr]\,.
\end{align}
Hence, from \eqref{rl1}-\eqref{rl3} we obtain that 
\begin{align*}
&\frac{1}{2}\E\bigl[\bigl(F(X')-F(X)^4\bigr)\bigr]\notag\\
&=\bigl(1-4+3)E_\mu\bigl[F^4\bigr] +6E_\mu\bigl[F^2\Gamma(F,F)\bigr] +(6-4)E_\mu\bigl[F^3 LF\bigr]\notag\\
&=6E_\mu\bigl[F^2\Gamma(F,F)\bigr] +2E_\mu\bigl[F^3 LF\bigr]\,,
\end{align*}
proving the first identity. The second one follows from it by applying \eqref{intparts}.
\end{proof}

\begin{proof}[Proof of Theorem \ref{genbound1}]
 Fix $h\in\Lip(1)$ and denote by $\psi=\psi_h$ the solution \eqref{steinsol} to the Stein equation \eqref{steineq}. Then, using \eqref{intparts} we see that 
 \begin{align}\label{bounds1}
 E_\mu\bigl[F\psi(F)\bigr]&= E_\mu\bigl[LL^{-1}F\psi(F)\bigr]=E_\mu\bigl[\Gamma\bigl(\psi(F),-L^{-1}F\bigr)\bigr]\notag\\
 &=E_\mu\bigl[\psi'(F)\Gamma(F,-L^{-1}F\bigr)\bigr]+E_\mu\bigl[R_\psi(F,-L^{-1}F)\bigr]\,.
 \end{align}
Hence, 
\begin{align}\label{bounds2}
 \babs{E_\mu[h(F)]-\gamma(h)}&=\babs{E_\mu\bigl[\psi'(F)-F\psi(F)\bigr]}\notag\\
 &=\babs{E_\mu\bigl[\psi'(F)\bigl(1-\Gamma(F,-L^{-1}F\bigr)\bigr)-R_\psi(F,-L^{-1}F)\bigr]}\notag\\
& \leq\fnorm{\psi'}\,E_\mu\babs{1-\Gamma(F,-L^{-1}F\bigr)}+\babs{E_\mu\bigl[R_\psi(F,-L^{-1}F)\bigr]}\notag\\
&\leq  \sqrt{\frac{2}{\pi}}E_\mu\babs{1-\Gamma(F,-L^{-1}F\bigr)}+\babs{E_\mu\bigl[R_\psi(F,-L^{-1}F)\bigr]}\,,
\end{align}
where we have used \eqref{bounds} for the final inequality. Now, from \eqref{boundRpsiFG} and again \eqref{bounds} we have that 
\begin{align}\label{bounds3}
 \babs{E_\mu\bigl[R_\psi(F,-L^{-1}F)\bigr]}&\leq\frac{\fnorm{\psi''}}{4}\E\bigl[\bigl(F(X')-F(X)\bigr)^2\babs{L^{-1}F(X')-L^{-1}F(X)}\bigr]\notag\\
 &\leq\frac{1}{2}\E\bigl[\bigl(F(X')-F(X)\bigr)^2\babs{L^{-1}F(X')-L^{-1}F(X)}\bigr]
\end{align}
and an application of the Cauchy-Schwarz inequality, \eqref{intparts} and Lemma \ref{remlemma} together imply that
\begin{align}\label{bounds4}
 &\babs{E_\mu\bigl[R_\psi(F,-L^{-1}F)\bigr]}\notag\\
&\leq\frac12\Bigl(\E\bigl[\bigl(F(X')-F(X)\bigr)^4\bigr]\Bigr)^{1/2} \Bigl(\E\bigl[\bigl(L^{-1}F(X')-L^{-1}F(X)\bigr)^2\bigr]\Bigr)^{1/2}\notag\\
 &\Bigl(E_\mu\bigl[F^3 LF\bigr]+3E_\mu\bigl[F^2\Gamma(F,F)\bigr]\Bigr)^{1/2} \Bigl(E_\mu\bigl[\Gamma(L^{-1}F,L^{-1}F)\bigr]\Bigr)^{1/2}\notag\\
 &=\Bigl(E_\mu\bigl[F^3 LF\bigr]+3E_\mu\bigl[F^2\Gamma(F,F)\bigr]\Bigr)^{1/2}\Bigl(-E_\mu\bigl[FL^{-1}F\bigr]\Bigr)^{1/2}\,.
\end{align}
\end{proof}

\begin{proof}[Proof of Theorem \ref{genbound2}]
Note that for $F$ as in the statement we can define $L^{-1}F$ by
\[L^{-1}F:=-\sum_{p=1}^m\lambda_p^{-1} F_p\,.\]
Then, both $LL^{-1}F=F$ and $L^{-1}LF=F$ hold. Hence, Assumption \eqref{specgap} can be dispensed with and as in the proof of Theorem \ref{genbound1} we arrive at
\begin{align}\label{bounds7}
\babs{E_\mu[h(F)]-\gamma(h)}&\leq \sqrt{\frac{2}{\pi}}E_\mu\babs{1-\Gamma(F,-L^{-1}F\bigr)}\notag\\
&\;+\frac{1}{2}\E\bigl[\bigl(F(X')-F(X)\bigr)^2\babs{L^{-1}F(X')-L^{-1}F(X)}\bigr]\,.
\end{align}
Note that under the assumption of the theorem we have 
\begin{align*}
\Gamma(F,-L^{-1}F\bigr)=\sum_{p,q=1}^m\lambda_p^{-1}\Gamma(F_q,F_p)\,,
\end{align*}
so we only have to further bound the second term on the right hand side of \eqref{bounds7}. We have 
\begin{align}\label{bounds5}
&\E\bigl[\bigl(F(X')-F(X)\bigr)^2\babs{L^{-1}F(X')-L^{-1}F(X)}\bigr]\notag\\
&\leq \sum_{p=1}^m\lambda_p^{-1}\E\bigl[\bigl(F(X')-F(X)\bigr)^2\babs{F_p(X')-F_p(X)}\bigr]\notag\\
&=\sum_{p,q,r=1}^m\lambda_p^{-1}\E\bigl[\bigl(F_q(X')-F_q(X)\bigr)\bigl(F_r(X')-F_r(X)\bigr)\babs{F_p(X')-F_p(X)}\bigr]\,,
\end{align}
already implying \eqref{gb21}. Now, for all fixed integers $p,q,r\in\{1,\dotsc,m\}$, by H\"older's inequality, 
\begin{align}\label{bounds6}
&\babs{\E\bigl[\bigl(F_q(X')-F_q(X)\bigr)\bigl(F_r(X')-F_r(X)\bigr)\babs{F_p(X')-F_p(X)}\bigr]}\notag\\
&\leq\Bigl(\E\babs{F_p(X')-F_p(X)}^2\Bigr)^{1/2} \Bigl(\E\babs{F_q(X')-F_q(X)}^4\Bigr)^{1/4} \Bigl(\E\babs{F_r(X')-F_r(X)}^4\Bigr)^{1/4}\notag\\
&=2\Bigl(2\lambda_p E_\mu\bigl[F_p^2\bigr]\Bigr)^{1/2}\Bigl(3E_\mu\bigl[F_q^2\Gamma(F_q,F_q)\bigr]-\lambda_q E_\mu\bigl[F_q^4\bigr]\Bigr)^{1/4}\notag\\
&\hspace{4cm}\cdot\Bigl(3E_\mu\bigl[F_r^2\Gamma(F_r,F_r)\bigr]-\lambda_r E_\mu\bigl[F_r^4\bigr]\Bigr)^{1/4}\notag\\
&=2\sqrt{2}\sqrt{\lambda_p}\sqrt{E_\mu[F_p^2]}\lambda_q^{1/4}\lambda_r^{1/4}\Bigl(3E_\mu\bigl[F_q^2\lambda_q^{-1}\Gamma(F_q,F_q)\bigr]- E_\mu\bigl[F_q^4\bigr]\Bigr)^{1/4}\notag\\
&\hspace{4cm}\cdot\Bigl(3E_\mu\bigl[F_r^2\lambda_r^{-1}\Gamma(F_r,F_r)\bigr]- E_\mu\bigl[F_r^4\bigr]\Bigr)^{1/4}\,,
\end{align}
where we have used Lemma \ref{remlemma}, the fact that $F_q$ and $F_r$ are eigenfunctions of $L$ as well as that 
\begin{equation*}
\E\babs{F_p(X')-F_p(X)}^2=2\lambda_p\E\bigl[F_p(X)^2\bigr]=2\lambda_p E_\mu\bigl[F_p^2\bigr]\,.
\end{equation*}
Now plugging \eqref{bounds6} into \eqref{bounds5} yields
\begin{align*}
&\frac12\E\bigl[\bigl(F(X')-F(X)\bigr)^2\babs{L^{-1}F(X')-L^{-1}F(X)}\bigr]\notag\\
&\leq \sqrt{2}\sum_{p,q,r=1}^m\lambda_p^{-1/2}\sqrt{E_\mu[F_p^2]}\lambda_q^{1/4}\lambda_r^{1/4}\Bigl(3E_\mu\bigl[F_q^2\lambda_q^{-1}\Gamma(F_q,F_q)\bigr]- E_\mu\bigl[F_q^4\bigr]\Bigr)^{1/4}\notag\\
&\hspace{4cm}\cdot\Bigl(3E_\mu\bigl[F_r^2\lambda_r^{-1}\Gamma(F_r,F_r)\bigr]- E_\mu\bigl[F_r^4\bigr]\Bigr)^{1/4}\notag\\
&=\sqrt{2}\sum_{p=1}^m\lambda_p^{-1/2}\sqrt{E_\mu[F_p^2]} \Biggl(\sum_{q=1}^m \lambda_q^{1/4}\Bigl(3E_\mu\bigl[F_q^2\lambda_q^{-1}\Gamma(F_q,F_q)\bigr]- E_\mu\bigl[F_q^4\bigr]\Bigr)^{1/4}\Biggr)^2\notag\,.
\end{align*}
The bound \eqref{gb22} 
now follows from \eqref{gb11}.
\end{proof}

\begin{proof}[Proof of Theorem \ref{genbound3}]
By the assumption that $\lambda>0$, for $h\in\Lip(1)$ and with $\psi=\psi_h$ we can still write
\begin{align}\label{steincomp2}
 \int_Eh(F)d\mu-\int_\R hd\gamma&=E_\mu\bigl[\psi'(F)-F\psi(F)\bigr]\notag\\
 &=E_\mu\Bigl[\psi'(F)+\lambda^{-1}(LF)\psi(F)\Bigr]-E_\mu\bigl[\psi(F)\bigl(\lambda^{-1}LF+F\bigr)\bigr]    \notag\\
 &=\lambda^{-1}E_\mu\Bigl[ \psi'(F)\bigl(\lambda-\Gamma(F)\bigr)\Bigr] -E_\mu\bigl[\psi(F)\bigl(\lambda^{-1}LF+F\bigr)\bigr]\notag\\
 &\;+\lambda^{-1}E_\mu\bigl[R_{\psi}(F)\bigr]\,,
\end{align}
where we have applied \eqref{intparts}. Hence, we obtain the bound
\begin{align*}
 &\Babs{\int_Eh(F)d\mu-\int_\R hd\gamma}\leq\frac{\fnorm{\psi'}}{\lambda}E_\mu\babs{\Gamma(F)-\lambda}+\fnorm{\psi}E_\mu\babs{\lambda^{-1}LF+F}+ \frac{1}{\lambda} E_\mu\babs{R_{\psi}(F)}\notag\\
 &\leq\frac{\fnorm{\psi'}}{\lambda}\Bigl(\Var_\mu\bigl(\Gamma(F)\bigr)\Bigr)^{1/2}+\fnorm{\psi}\Bigl(E_\mu\bigl[\bigl(\lambda^{-1}LF+F\bigr)^2\bigr]\Bigr)^{1/2}+ \frac{1}{\lambda} E_\mu\babs{R_{\psi}(F)}\notag\\
&=\frac{\fnorm{\psi'}}{\lambda}\Bigl(\Var_\mu\bigl(\Gamma(F)\bigr)\Bigr)^{1/2}+\fnorm{\psi}\Bigl(\lambda^{-2}E_\mu\bigl[(LF)^2\bigr]-1\Bigr)^{1/2}+ \frac{1}{\lambda} E_\mu\babs{R_{\psi}(F)}\notag\\
&=\frac{\fnorm{\psi'}}{\lambda}\Bigl(\Var_\mu\bigl(\Gamma(F)\bigr)\Bigr)^{1/2}+\fnorm{\psi}\Bigl(\lambda^{-2}\Var_\mu(LF)-1\Bigr)^{1/2}+ \frac{1}{\lambda} E_\mu\babs{R_{\psi}(F)}\,.
\end{align*}
The bounds \eqref{gb31} and \eqref{gb32} then follow from \eqref{boundRpsiF} and then \eqref{gb33} follows by an application of Lemma \ref{remlemma}.
\end{proof}

\subsection{Proofs from Section \ref{spaces}}

\begin{proof}[Proof of Theorem \ref{genboundind}]
We will apply bound \eqref{gb22} from Theorem \ref{genbound2} with $m=n$ and with $\lambda_p=p/n$ for $p=1,\dotsc,n$. By Proposition \ref{Prodind} we have 
\[\Gamma(F_p,F_q)=\frac{1}{2n}\sum_{\substack{M\subseteq[n]:\\|M|\leq p+q-1}}(p+q-|M|)U_M(p,q)\,,\quad 1\leq p,q\leq n\,,\]
where 
\[F_pF_q=\sum_{\substack{M\subseteq[n]:\\|M|\leq p+q}}U_M(p,q)\]
is the Hoeffding decomposition of $F_pF_q$. From this, using orthogonality, it follows immediately that 
\begin{align*}
\biggl(\Var_\mu\Bigl(\sum_{p,q=1}^m\lambda_p^{-1}\Gamma(F_p,F_q)\Bigr)\biggr)^{1/2}&\leq \Biggl(\sum_{l=1}^{2n-1}
\sum_{M\in\D_l(n)}\Var\biggl(\sum_{\substack{1\leq p,q\leq n:\\ p+q\geq l+1}}\frac{p+q-l}{2p}U_M(p,q)\biggr)\Biggr)^{1/2}\\
&\leq\sum_{ p,q=1}^n \frac{p+q-1}{2p}\biggl(\sum_{\substack{M\subseteq[n]:\\ 1\leq |M|\leq p+q-1}}\Var\bigl(U_M(p,q\bigr)\biggr)^{1/2} \,.
\end{align*}
To bound the second term, we argue again by way of orthogonality to obtain that 
\begin{align}\label{pind1}
&3\E\bigl[F_q^2\lambda_q^{-1}\Gamma(F_q,F_q)\bigr]- \E\bigl[F_q^4\bigr]\notag\\
&=3\E\Biggl[\sum_{{\substack{M\subseteq[n]:\\  |M|\leq 2q}}}U_M(q,q)\cdot
\frac{1}{2q}\sum_{{\substack{M\subseteq[n]:\\  |N|\leq 2q-1}}}(2q-|N|)U_N(q,q)\Biggr]-\E\bigl[F_q^4\bigr]\notag\\
&=3\E[U_\emptyset(q,q)^2]-\E\bigl[F_q^4\bigr]+3 \sum_{{\substack{M\subseteq[n]:\\ 1\leq|M|\leq 2q-1}}}\frac{2q-|M|}{2q}\E\bigl[U_M(q,q)^2\bigr]\notag\\
&\leq 3\E\bigl[F_q^2\bigr]^2-\E\bigl[F_q^4\bigr]+\frac{2q-1}{2q}\sum_{{\substack{M\subseteq[n]:\\ 1\leq|M|\leq 2q-1}}}\Var(U_M(q,q))\notag\\
&\hspace{3cm}+ 2\sum_{{\substack{M\subseteq[n]:\\ 1\leq|M|\leq 2q-1}}}\frac{2q-|M|}{2q}\Var(U_M(q,q))\notag\\
&\leq  3\E\bigl[F_q^2\bigr]^2-\E\bigl[F_q^4\bigr]+\sum_{{\substack{M\subseteq[n]:\\ 1\leq|M|\leq 2q-1}}}\Var(U_M(q,q))\notag\\
&\hspace{3cm}+ \sum_{{\substack{M\subseteq[n]:\\ 1\leq|M|\leq 2q-1}}}\frac{2q-|M|}{q}\Var(U_M(q,q))\,.
\end{align}
Here, for the second equality we used the fact that $U_\emptyset(q,q)=\E[F_q^2]$.
Next, from a slight generalization of \cite[Lemma 2.10]{DP16} to non-normalized random variables we obtain that 
\begin{align*}
\sum_{{\substack{M\subseteq[n]:\\ 1\leq|M|\leq 2q-1}}}\Var(U_M(q,q))&\leq \E\bigl[F_q^4\bigr]-3\E\bigl[F_q^2\bigr]^2 +\kappa_q\E[F_q^2]\rho_q^2\,.
\end{align*}
Plugging this into \eqref{pind1} we obtain that 
\begin{align}\label{pind2}
3\E\bigl[F_q^2\lambda_q^{-1}\Gamma(F_q,F_q)\bigr]- \E\bigl[F_q^4\bigr]&\leq \sum_{{\substack{M\subseteq[n]:\\ 1\leq|M|\leq 2q-1}}}\frac{2q-|M|}{q}\Var(U_M(q,q)) +\kappa_q\E[F_q^2]\rho_q^2 \notag\\
&\leq\frac{2q-1}{q}\sum_{{\substack{M\subseteq[n]:\\ 1\leq|M|\leq 2q-1}}}\Var(U_M(q,q))+\kappa_q\E[F_q^2]\rho_q^2\,,
\end{align}
yielding the respective second terms of the bounds.
\end{proof}

In what follows, for $p\in\N$ and a function $f:S^p\rightarrow\R$ we write $\tilde{f}$ for its \textit{symmetrization}, i.e. 
\begin{equation*}
 \tilde{f}(x_1,\dotsc,x_p):=\frac{1}{p!}\sum_{\sigma\in\mathbb{S}_p}f\bigl(x_{\sigma(1)},\dotsc,x_{\sigma(p)}\bigr)\,,\quad (x_1,\dotsc,x_p)\in S^p\,.
\end{equation*}
 Note that, if $f\in L^q(\nu^p)$, then $\|\tilde{f}\|_{L^q(\nu^p)}\leq \|{f}\|_{L^q(\nu^p)}$, by the triangle inequality. 

The following crucial product formula is Proposition 2.6 from \cite{DP19}.

\begin{prop}[Product formula for degenerate, symmetric $U$-statistics]\label{pform}
Let $p,q\geq1$ be positive integers and assume that $\psi\in L^2(\nu^{ p})$ and $\phi\in L^2(\nu^{ q})$ are canonical, symmetric kernels of orders $p$ and $q$ respectively. Then, whenever $n\geq p+q$ we have the Hoeffding decomposition:
\begin{equation}\label{prodform}
 J_p(\psi) J_q(\phi)=\sum_{t=0}^{2(p\wedge q)} J_{p+q-t}(\chi_{p+q-t})\,,
\end{equation}
where, for $t\in\{0,1,\dotsc,2(p\wedge q)\}$, the canonical, symmetric kernel \[\chi_{p+q-t}:S^{p+q-t}\rightarrow\R,\] of order $p+q-t$, is given by 
\begin{equation}\label{kerpf}
 \chi_{p+q-t}=\sum_{r=\ceil{\frac{t}{2}}}^{t\wedge p\wedge q}\binom{n-p-q+t}{t-r}\binom{p+q-t}{p-r,q-r,2r-t}\bigl(\widetilde{\psi\star_r^{t-r}\phi}\bigr)_{p+q-t}\,.
\end{equation}
{In the previous expression,} the kernels $\bigl(\widetilde{\psi\star_r^{t-r}\phi}\bigr)_{p+q-t}$ appearing in the definition of the Hoeffding components $J_{p+q-t}\bigl((\widetilde{\psi\star_r^{t-r}\phi})_{p+q-t}\bigl)$
are defined as in \eqref{defphip} and we have written $\ceil{x}$ to indicate the smallest integer greater or equal to the real number $x$.
\end{prop}

\begin{proof}[Proof of Theorem \ref{genboundsymstat}]
We start from the second bound in Theorem \ref{genboundind}.
By Proposition \ref{pform}, for $1\leq p,q\leq n$, and letting 
\[\chi_{p+q-t}=\sum_{r=\ceil{\frac{t}{2}}}^{t\wedge p\wedge q}\binom{n-p-q+t}{t-r}\binom{p+q-t}{p-r,q-r,2r-t}\bigl(\widetilde{\phi_p\star_r^{t-r}\phi_q}\bigr)_{p+q-t}\]
 we have 
\begin{align}\label{psym1}
&\biggl(\sum_{\substack{M\subseteq[n]:\\ 1\leq |M|\leq p+q-1}}\Var\bigl(U_M(p,q\bigr)\biggr)^{1/2}=\biggl(\;\sum_{t=1}^{2(p\wedge q)\wedge(p+q-1)} \Var\bigl(J_{p+q-t}(\chi_{p+q-t})\bigr)\biggr)^{1/2}\notag\\
&=\biggl(\sum_{t=1}^{2(p\wedge q)\wedge(p+q-1)}\binom{n}{p+q-t}\norm{(\chi_{p+q-t})_{p+q-t}}_2^2\biggr)^{1/2}\notag\\
&\leq\sum_{t=1}^{2(p\wedge q)\wedge(p+q-1)}\sqrt{\binom{n}{p+q-t}}\norm{\chi_{p+q-t}}_2\notag\\
&\leq \sum_{t=1}^{2(p\wedge q)\wedge(p+q-1)}\sqrt{\binom{n}{p+q-t}}\sum_{r=\ceil{\frac{t}{2}}}^{t\wedge p\wedge q}\binom{n-p-q+t}{t-r}\notag\\
&\hspace{4cm}\cdot\binom{p+q-t}{p-r,q-r,2r-t}
\norm{\phi_p\star_r^{t-r}\phi_q}_2\notag\\
&\leq\sum_{t=1}^{2(p\wedge q)\wedge(p+q-1)}\;\sum_{r=\ceil{\frac{t}{2}}}^{t\wedge p\wedge q} n^{\frac{p+q+t}{2}-r}\alpha(p,q,t,r)\norm{\phi_p\star_r^{t-r}\phi_q}_2\,,
\end{align}
where we have used the triangle inequality several times as well as the fact that 
\[\norm{(\chi_{p+q-t})_{p+q-t}}_2\leq\norm{\chi_{p+q-t}}_2\,.\]
Next, it is easy to see that in the symmetric case 
\begin{equation}\label{psym2}
\rho_q^2=\binom{n-1}{q-1}\norm{\phi_q}^2=\frac{q}{n}\E[F_q^2]\,.
\end{equation}
In order to derive a bound on $\kappa_q^2$ we have to recall its definition from \cite{DP16}. From the proof of \cite[Lemma 2.10]{DP16} it is evident that we may take $\kappa_q=2q$, whenever we can show that the quantity $S_0$ appearing there in formula (2.13) is nonnegative in the symmetric situation. This however is the content of (the slightly more general) Lemma \ref{s0pos} below. 
Hence, using \eqref{psym2}, we can conclude that 
\[\kappa_q\E[F_q^2]\rho_q^2=\frac{2q^2}{n}\E[F_q^2]^2\,,\]
finishing the proof of Theorem \ref{genboundsymstat}.
\end{proof}

For degenerate, symmetric kernels $\phi\in L^2(\nu^q)$ and $\psi \in L^2(\nu^p)$ consider the corresponding degenerate, symmetric $U$-statistics
\begin{align*}
V:=J_q(\phi)&=\sum_{I\in\D_q(n)}\phi(X_i,i\in I)=:\sum_{I\in\D_q(n)} V_I\quad\text{and}\\
W:=J_p(\psi)&=\sum_{J\in\D_p(n)}\phi(X_j,j\in I)=:\sum_{J\in\D_p(n)} W_J\,.
\end{align*}
Further, we let $\mathcal{S}_0:=\mathcal{S}_0(n,q,p)$ be the collection of all quadruples $(I,J,K,L)\in\D_q(n)^2\times\D_p(n)^2$ of indices such that
\begin{enumerate}[(i)]
 \item $I\cap K=J\cap L=\emptyset$,
 \item $\emptyset\not=I\cap J=I\setminus(I\cap L)\not=I$,
 \item $\emptyset\not=J\cap I=J\setminus(J\cap K)\not=J$,
 \item $\emptyset\not=K\cap J=K\setminus(L\cap K)\not=K$,
 \item $\emptyset\not=L\cap I=L\setminus(L\cap K)\not=L$,
\end{enumerate}
and we define 
\begin{align*}
S_0(V,W)&=\sum_{\substack{(I,J,K,L)\in\mathcal{S}_0}} \E\bigl[V_IV_JW_KW_L\bigr]=\sum_{\substack{I,J\in\D_q(n),\\ K,L\in\D_p(n):\\ I\cap K= J\cap L=\emptyset,\\ \emptyset\not=I\cap J=I\setminus(I\cap L)\not=I,\\ \emptyset\not=J\cap I=J\setminus(J\cap K)\not=J}} \E\bigl[V_IV_JW_KW_L\bigr]\,,
\end{align*}
where the last identity follows from degeneracy.

\begin{lemma}\label{s0pos}
 In the above situation of symmetric, degenerate $U$-statistics $V=J_q^{(m)}(\phi)$ and $W=J_p^{(n)}(\psi)$ we always have that $S_0(V,W)\geq0$.
\end{lemma}

\begin{proof}
 We slightly abuse notation to write $X_I$ for $(X_i)_{i\in I}$ and so on. Then, by definition of $S_0(V,W)$ we have 
 \begin{align}\label{s0p1}
 S_0(V,W)&=\sum_{\substack{(I,J,K,L)\in\mathcal{S}_0}}\E\bigl[\phi(X_I)\phi(X_J)\psi(X_K)\psi(X_L)\bigr]\,. 
 \end{align}
Fix $(I,J,K,L)\in\mathcal{S}_0$. Then, letting $r:=|I\cap J|$ it follows that $1\leq r\leq q-1$. Also, $|I\cap L|=|I|-|I\cap J|=q-r$ and $1\leq|K\cap L|=|L|-|L\cap I|=p-(q-r)=p-q+r$, which implies that $r\geq 1+q-p$.
Moreover, by Fubini's theorem, for such a quadruple we have 
\begin{align}\label{s0p2}
 &\E\bigl[\phi(X_I)\phi(X_J)\psi(X_K)\psi(X_L)\bigr]\notag\\
 &=\E\Bigl[\bigl(\phi\star_r^r\phi\bigr)\bigl(X_{I\cap L},X_{J\cap K}\bigr) \bigl(\psi\star_{p-q+r}^{p-q+r}\psi\bigr)\bigl(X_{J\cap K},X_{I\cap L}\bigr) \Bigr]\notag\\
&= \E\Bigl[\bigl(\phi\star_r^r\phi\bigr)\bigl(X_{I\cap L},X_{J\cap K}\bigr) \bigl(\psi\star_{p-q+r}^{p-q+r}\psi\bigr)\bigl(X_{I\cap L},X_{J\cap K}\bigr)\Bigr]\notag\\
&=\bigr\langle \bigl(\phi\star_r^r\phi\bigr), \bigl(\psi\star_{p-q+r}^{p-q+r}\psi\bigr)\bigl\rangle_{L^2(\nu^{2q-2r})}\,.
\end{align}
Now, letting $s=q-r$, we can conclude from items (iv) and (vi) of \cite[Lemma 2.4]{DP18b} that   
 \begin{align}\label{s0p3}
  &\bigr\langle \bigl(\phi\star_r^r\phi\bigr), \bigl(\psi\star_{p-q+r}^{p-q+r}\psi\bigr)\bigl\rangle_{L^2(\nu^{2q-2r})}=\bigr\langle \bigl(\phi\star_{q-s}^{q-s}\phi\bigr), \bigl(\psi\star_{p-s}^{p-s}\psi\bigr)\bigl\rangle_{L^2(\nu^{2q-2r})}\notag\\
&=\norm{\phi\star_s^s\psi}_{L^2(\nu^{p+q-2s})}^2=\norm{\phi\star_{q-r}^{q-r}\psi}^2_{L^2(\nu^{p-q+2r})}\geq0\,.
 \end{align}
Hence, \eqref{s0p1}, \eqref{s0p2} and \eqref{s0p3} yield that
\begin{align*}
 S_0(V,W)&=\sum_{r=1\vee(1+q-p)}^{q-1}\sum_{\substack{(I,J,K,L)\in\mathcal{S}_0:\\ |I\cap J|=r}}\norm{\phi\star_{q-r}^{q-r}\psi}^2_{L^2(\nu^{p-q+2r})}\geq0\,,
\end{align*}
as claimed.
\end{proof}

For the proof of Theorem \ref{symustat2}, the next result, which is Lemma 5.7 in \cite{DKP}, is needed. 

\begin{lemma}\label{genulemma}
For all integers $r, p, q \geq 1$ and $ l\geq 0$ such that $0\leq l\leq r\leq p\wedge q$, there exists a constant $K(p,q,r,l)\in(0,\infty)$ which only depends on $p,q,r$ and $l$ such that 
\begin{align*}
 \|\psi_{p}\star_r^l\psi_{q}\|_{2}&\leq K(p,q,r,l) \max_{(j,k,a,b)\in Q(p,q,r,l)}\|g_j\star_a^b g_k\|_{2}\,.
 \end{align*}
\end{lemma}  

\begin{proof}[Proofs of Corollary \ref{symustat1} and Theorem \ref{symustat2}]
Corollary \ref{symustat1} immediately follows from Theorem \ref{genboundsymstat} by noticing that, here, 
\[\phi_p=\frac{\binom{n-p}{m-p}}{\sigma}\psi_p\,,\quad p=1,\dotsc,m\,,\]
that $\phi_p=0$ for $p=m+1,\dotsc,n$ and by taking into account the bound \eqref{vb2}. Then, Theorem \ref{symustat2} follows from Corollary \ref{symustat1} by an application of Lemma \ref{genulemma}.
\end{proof}

\subsection{Proofs from Section \ref{apps}}

\begin{proof}[Proof of Theorem \ref{chitheo}]

The proof is an application of Theorem \ref{genboundsymstat}. Since $\phi_p\equiv0$ for $p\geq0$, we see from this Theorem that it suffices to prove the following estimates 
for the rellevant contraction norms:
\begin{enumerate}[(1)]
\item $\displaystyle n^{1/2} \norm{\phi_1\star_1^0\phi_1}_2=\frac{1}{\sqrt{n}}\frac{\norm{1/p_n -m}_4^2}{n\sigma_n^2}
\leq \frac{\sqrt{3}}{\sqrt{n}}\frac{\sqrt{\sum_{i=1}^m\frac{1}{p_n(i)^3}-m^4}}{n\sigma_n^2} $
\item $\displaystyle n\norm{\phi_1\star_1^0\phi_2}_2\leq \min\Biggl(  \frac{2\sqrt{\sum_{i=1}^m\frac{1}{p_n(i)^2}-m^3}}{n\sigma_n^2},\;  2\biggl(\frac{1}{n^{1/4}(m-1)^{1/4}}+\frac{1}{\sqrt{n}}\biggr)\frac{\norm{1/p_n-m}_4}{\sqrt{n}\sigma_n}\Biggr)$
\item $\displaystyle n^{3/2}\norm{\phi_1\star_1^1\phi_2}_2\leq   \frac{2}{\sqrt{m-1}}$
\item $\displaystyle n^{2}\norm{\phi_2\star_1^1\phi_2}_2  \leq  \frac{4}{\sqrt{m-1}}$
\item $\displaystyle n^{3/2}\norm{\phi_2\star_1^0\phi_2}_2 \leq\frac{4}{\sqrt{m-1}}+\frac{4}{\sqrt{n}}$
\item $\displaystyle n^{3/2}\norm{\phi_2\star_2^1\phi_2}_2 \leq\frac{4}{\sqrt{m-1}}+\frac{4}{\sqrt{n}}$
\item $\displaystyle n\norm{\phi_2\star_2^0\phi_2}_2  \leq  \frac{4\sqrt{\sum_{i=1}^m\frac{1}{p_n(i)^2}}}{n\sigma_n^2}$
\end{enumerate}
From \cite[Lemma 2.4 (vi)]{DP19} we see that (6) and (5) are the same. 
Furthermore, it follows for (1) from the definition of contractions that 
\begin{align*}
 &\norm{\phi_1\star_1^0\phi_1}_2^2=\norm{\phi_1}_4^4=\frac{1}{n^4\sigma_n^4}\sum_{i=1}^m\biggl(\frac{1}{p_n(i)}-m\biggr)^4p_n(i)\\
&=\frac{1}{n^4\sigma_n^4}\biggl(\sum_{i=1}^m\frac{1}{p_n(i)^3}-3m^4-4m\sum_{i=1}^m\frac{1}{p_n(i)^2}+6m^2\sum_{i=1}^m\frac{1}{p_n(i)}\biggr)\\
&\leq\frac{3}{n^4\sigma_n^4}\biggl(\sum_{i=1}^m\frac{1}{p_n(i)^3}-m^4\biggr)\,,
\end{align*}
where, for the last step, we have used the inequalities
\begin{align*}
\sum_{i=1}^m\frac{1}{p_n(i)^2}&\geq\biggl(\sum_{i=1}^m\frac{1}{p_n(i)}\biggr)^{3/2}\geq m\sum_{i=1}^m\frac{1}{p_n(i)}\,,\\
\sum_{i=1}^m\frac{1}{p_n(i)^3}&\geq \biggl(\sum_{i=1}^m\frac{1}{p_n(i)}\biggr)^2\geq m^2\sum_{i=1}^m\frac{1}{p_n(i)}\,,
\end{align*}
which all follow from the general inequality 
\begin{align*}
\sum_{i=1}^m\frac{1}{p_n(i)^k}&\geq \biggl(\sum_{i=1}^m\frac{1}{p_n(i)^{k-1}}\biggr)^{\frac{k+1}{k}}\,,\quad k\geq 1\,,
\end{align*}
which may be proved by means of Jensen's inequality. Hence, for (1) we conclude that 
\begin{align*}
 n^{1/2} \norm{\phi_1\star_1^0\phi_1}_2&=\frac{1}{\sqrt{n}}\frac{\norm{1/p_n -m}_4^2}{n\sigma_n^2}
\leq \frac{\sqrt{3}}{\sqrt{n}}\frac{\sqrt{\sum_{i=1}^m\frac{1}{p_n(i)^3}-m^4}}{n\sigma_n^2}
\end{align*}

Next, we deal with (5). Writing $\delta_{i,j}=\1_{\{i=j\}}$ we have that 
\begin{align*}
&\norm{\phi_2\star_1^0\phi_2}_2^2=\frac{16}{n^4\sigma_n^4}\sum_{i,j,k=1}^m\biggl(\frac{\delta_{i,j}}{p_n(j)}-1\biggr)^2\biggl(\frac{\delta_{j,k}}{p_n(j)}-1\biggr)^2p_n(i)p_n(j)p_n(k)\\
&=\frac{16}{n^4\sigma_n^4}\sum_{i,j,k=1}^m\frac{p_n(i)p_n(k)}{p_n(j)^3}\Bigl[\bigl(\delta_{i,j}-p_n(j)\bigr)^2\bigl(\delta_{k,j}-p_n(j)\bigr)^2\Bigr]\\
&=\frac{16}{n^4\sigma_n^4}\sum_{i,j,k=1}^m\frac{p_n(i)p_n(k)}{p_n(j)^3}\Bigl[\1_{\{i=j=k\}}-4\1_{\{i=j=k\}}p_n(j)+\delta_{i,j}p_n(j)^2+\delta_{k,j}p_n(j)^2\\
&\hspace{2cm}+4\1_{\{i=j=k\}}p_n(j)^2-2\delta_{i,j}p_n(j)^3-2\delta_{k,j}p_n(j)^3+p_n(j)^4\Bigr]\\
&=\frac{16}{n^4\sigma_n^4}\biggl(\sum_{j=1}^n\frac{1}{p_n(j)} -4m +m+m+4-2-2+1\biggr)\\
&=\frac{16}{n^4\sigma_n^4}\biggl(\sum_{j=1}^n\frac{1}{p_n(j)} -2m+1\biggr)=\frac{16}{n^4\sigma_n^4}\Bigl(\norm{1/p_n-m}_2^2+(m-1)^2\Bigr)\,.
\end{align*}
Hence, we conclude for (5) that 
\begin{align*}
&n^{3/2}\norm{\phi_2\star_1^0\phi_2}_2=\frac{4\sqrt{n}}{n\sigma_n^2}\Bigl(\norm{1/p_n-m}_2^2+(m-1)^2\Bigr)^{1/2}\\
&\leq\frac{4\sqrt{n}\norm{1/p_n-m}_2}{\norm{1/p_n-m}_2^2+2(n-1)(m-1)}
+\frac{4\sqrt{n}(m-1)}{\norm{1/p_n-m}_2^2+2(n-1)(m-1)}\\
&\leq \frac{4}{\sqrt{m-1}}+\frac{4}{\sqrt{n}}\,,
\end{align*}
where the second inequality follows from distinguishing cases as to whether \\$\norm{1/p_n-m}_2^2$ is greater or less than $2(n-1)(m-1)$.

Similarly, one computes that 
\begin{align*}
&\norm{\phi_1\star_1^0\phi_2}_2^2=\frac{4}{n^4\sigma_n^4}\biggl(\sum_{i=1}^m\frac{1}{p_n(i)^2}-(2m+1)\sum_{i=1}^m\frac{1}{p_n(i)}+m^3+m^2\biggr)\\
&\leq\frac{4}{n^4\sigma_n^4}\biggl(\sum_{i=1}^m\frac{1}{p_n(i)^2}-m^3\biggr)\,,
\end{align*}
so that, for (2), we obtain that 
\begin{align*}
n\norm{\phi_1\star_1^0\phi_2}_2&\leq \frac{2\sqrt{\sum_{i=1}^m\frac{1}{p_n(i)^2}-m^3}}{\norm{1/p_n-m}_2^2+2(n-1)(m-1)}\,.
\end{align*}
Alternatively, \cite[Lemma 2.4 (iii)]{DP19} implies for (2) that 
\begin{align*}
\norm{\phi_1\star_1^0\phi_2}_2^2&\leq \norm{\phi_1\star_1^0\phi_1}_2\cdot \norm{\phi_2\star_1^0\phi_2}_2
\end{align*}
and using the results for (1) and (5) yields that 
\begin{align*}
n\norm{\phi_1\star_1^0\phi_2}_2&\leq 2\biggl(\frac{1}{\sqrt{n(m-1)}}+\frac{1}{n}\biggr)^{1/2}\frac{\norm{1/p_n-m}_4}{\sqrt{n}\sigma_n}\\
&\leq 2\biggl(\frac{1}{n^{1/4}(m-1)^{1/4}}+\frac{1}{\sqrt{n}}\biggr)\frac{\norm{1/p_n-m}_4}{\sqrt{n}\sigma_n}
\end{align*}

Next, we turn to (4). Note that 
\begin{align*}
(\phi_2\star_1^1\phi_2)(i,j)&=\frac{4}{n^2\sigma_n^2}\sum_{k=1}^m\biggl(\frac{\delta_{i,k}}{p_n(i)}-1\biggr)\biggl(\frac{\delta_{j,k}}{p_n(j)}-1\biggr)p_n(k)\,.
\end{align*}
Hence, we have that 
\begin{align*}
&\norm{\phi_2\star_1^1\phi_2}_2^2=\frac{16}{n^4\sigma_n^4}\sum_{i,j,k,l=1}^m p_n(i)p_n(j)p_n(k)p_n(l)\\
&\;\cdot\biggl(\frac{\delta_{i,k}}{p_n(i)}-1\biggr)\biggl(\frac{\delta_{j,k}}{p_n(j)}-1\biggr)\biggl(\frac{\delta_{i,l}}{p_n(i)}-1\biggr)\biggl(\frac{\delta_{j,l}}{p_n(j)}-1\biggr)\\
&=\frac{16}{n^4\sigma_n^4}\sum_{i,j,k,l=1}^m\frac{p_n(k)}{p_n(j)}\bigl(\delta_{i,k}-p_n(i)\bigr)\bigl(\delta_{j,k}-p_n(j)\bigr)\bigl(\delta_{i,l}-p_n(l)\bigr)\bigl(\delta_{j,l}-p_n(j)\bigr)\\
&=\frac{16}{n^4\sigma_n^4}\bigl(m-4+6-4+1\bigr)=\frac{16(m-1)}{n^4\sigma_n^4}\,,
\end{align*}
where we have omitted some details of the computation. Hence, for (4) we conclude that 
\begin{align*}
n^2\norm{\phi_2\star_1^1\phi_2}_2&=\frac{4\sqrt{m-1}}{\sigma_n^2}=\frac{4n\sqrt{m-1}}{\norm{1/p_n-m}_2^2+2(n-1)(m-1)}\\
&\leq\frac{4}{\sqrt{m-1}}\,.
\end{align*}

Next, for (3), observe that 
\begin{align*}
(\phi_1\star_1^1\phi_2)(i)=\frac{2}{n^2\sigma_n^2}\sum_{j=1}^m\biggl(\frac{1}{p_n(j)}-m\biggr)\biggl(\frac{\delta_{i,j}}{p_n(j)}-1\biggr)p_n(j)
\end{align*}
and, hence,
\begin{align*}
&\norm{\phi_1\star_1^1\phi_2}_2^2=\frac{4}{n^4\sigma_n^4}\sum_{i,j,k=1}^m\biggl(\frac{1}{p_n(j)}-m\biggr)\biggl(\frac{\delta_{i,j}}{p_n(j)}-1\biggr)
\biggl(\frac{1}{p_n(k)}-m\biggr)\biggl(\frac{\delta_{i,k}}{p_n(k)}-1\biggr)\\
&\hspace{5cm}\cdot p_n(j)p_n(k)p_n(i)\\
&=\norm{\phi_1\star_1^1\phi_2}_2^2=\frac{4}{n^4\sigma_n^4}\sum_{i,j,k=1}^m\frac{1}{p_n(j)}\bigl(1-mp_n(j)\bigr)\bigl(1-mp_n(k)\bigr)\\
&\hspace{5cm}\cdot\bigl(\delta_{i,j}-p_n(j)\bigr)\bigl(\delta_{i,k}-p_n(i)\bigr)\\
&=\frac{4}{n^4\sigma_n^4}\biggl(\sum_{j=1}^m\frac{1}{p_n(j)}-m^2\biggr)=\frac{4\norm{1/p_n-m}_2^2}{n^4\sigma_n^4}\,,
\end{align*}
where we have again omitted the details. Hence, for (3) we obtain that 
\begin{align*}
n^{3/2}\norm{\phi_1\star_1^1\phi_2}_2&=\frac{2\sqrt{n}\norm{1/p_n-m}_2}{n\sigma_n^2}
=\frac{2\sqrt{n}\norm{1/p_n-m}_2}{\norm{1/p_n-m}_2^2+2(n-1)(m-1)}\\
&\leq \frac{2}{\sqrt{m-1}}\,,
\end{align*}
where the inequality again follows from distinguishing cases as above.
As to (7), oberve that we have 
\begin{align*}
&\norm{\phi_2\star_2^0\phi_2}_2^2=\norm{\phi_2}_4^4=\frac{16}{n^4\sigma_n^4}\sum_{i,j=1}^m\biggl(\frac{\delta_{i,j}}{p_n(i)}-1\biggr)^4p_n(i)p_n(j)\\
&=\frac{16}{n^4\sigma_n^4}\sum_{i=1}^m p_n(i)\biggl(\sum_{j\not=i}p_n(j)+ \biggl(\frac{1}{p_n(i)}-1\biggr)^4p_n(i)\biggr)\\
&=\frac{16}{n^4\sigma_n^4}\biggl(\sum_{i=1}^m p_n(i)(1-p_n(i))+\sum_{i=1}^m\biggl(\frac{1}{p_n(i)}-1\biggr)^4p_n(i)^2\biggr)\\
&=\frac{16}{n^4\sigma_n^4}\biggl(1-\sum_{i=1}^m p_n(i)^2+\sum_{i=1}^m\frac{1}{p_n(i)^2}-4\sum_{i=1}^mp_n(i)-4\sum_{i=1}^m\frac{1}{p_n(i)}\\
&\hspace{4cm}+6m+\sum_{i=1}^m p_n(i)^2\biggr)\\
&=\frac{16}{n^4\sigma_n^4}\biggl(6m-3+\sum_{i=1}^m\frac{1}{p_n(i)^2}-4\sum_{i=1}^m\frac{1}{p_n(i)}\biggr)\\
&\leq \frac{16}{n^4\sigma_n^4}\biggl(\sum_{i=1}^m\frac{1}{p_n(i)^2}+6m-3-4m^2\biggr)\leq\frac{16}{n^4\sigma_n^4}\sum_{i=1}^m\frac{1}{p_n(i)^2}\,,
\end{align*}
where we have used the inequality 
\[\sum_{i=1}^m\frac{1}{p_n(i)}\geq m^2\,,\]
which follows from \eqref{varpn}, in the next to last inequality as well as the fact that $6m-3-4m^2<0$ for all $m\in\N$. Hence, we obtain for (7) that 
\begin{align*}
&n\norm{\phi_2\star_2^0\phi_2}_2\leq \frac{4}{n\sigma_n^2}\biggl(\sum_{i=1}^m\frac{1}{p_n(i)^2}\biggr)^{1/2}
=\frac{4\sqrt{\sum_{i=1}^m\frac{1}{p_n(i)^2}}}{\norm{1/p_n-m}_2^2+2(n-1)(m-1)}\,.
\end{align*}

\end{proof}

\normalem
\bibliography{markex}{}
\bibliographystyle{alpha}
\end{document}